\documentclass{article}

\usepackage{hyperref}
\usepackage{amsfonts, amsmath, amsthm}
\usepackage[margin=3cm]{geometry}
\usepackage{tikz}
\usetikzlibrary{calc, arrows, decorations.markings, decorations.pathmorphing, positioning, decorations.pathreplacing}
\makeatletter
\tikzset{nomorepostaction/.code={\let\tikz@postactions\pgfutil@empty}}
\makeatother

\usepackage{fancyhdr}

\newtheorem{thm}{Theorem}[section]
\newtheorem{cor}[thm]{Corollary}
\newtheorem{lem}[thm]{Lemma}
\newtheorem{prop}[thm]{Proposition}

\newtheorem{defn}[thm]{Definition}

\newcommand{\To}{\longrightarrow}
\newcommand{\C}{\mathbb{C}}
\newcommand{\Z}{\mathbb{Z}}
\newcommand{\V}{\mathcal{V}}
\newcommand{\D}{\mathcal{D}}
\newcommand{\0}{{\bf 0}}
\newcommand{\1}{{\bf 1}}
\DeclareMathOperator{\Hom}{Hom}
\DeclareMathOperator{\ev}{ev}

\begin{document}

\title{Itsy bitsy topological field theory}

\author{Daniel V. Mathews \\
Department of Mathematics,
Boston College
}%

\date{}

\maketitle

\begin{abstract}
We construct an elementary, combinatorial kind of topological quantum field theory, based on curves, surfaces, and orientations. The construction derives from contact invariants in sutured Floer homology and is essentially an elaboration of a TQFT defined by Honda--Kazez--Matic. This topological field theory stores information in binary format on a surface and has ``digital'' creation and annihilation operators, giving a toy-model embodiment of ``it from bit''.
\end{abstract}

\tableofcontents

\section{Introduction}

\subsection{Its and bits from TQFT and sutured Floer homology}

In this paper we construct a combinatorial kind of topological quantum field theory, based on curves, surfaces, and orientations. 

The elementary objects of this theory are squares, with a little combinatorial structure (signs on corners), which we call \emph{occupied squares}. These squares may be glued together along their edges to form more complicated \emph{occupied surfaces}. Conversely, occupied surfaces can be decomposed into occupied squares. Curves of a certain type, known as \emph{sutures}, may be drawn on these squares and surfaces. Algebraic data is associated to all of these very simple combinatorial constructions.

Our construction is in the spirit, although not strictly satisfying the usual definition, of a topological quantum field theory \cite{Witten88}. Related objects have been discussed by the author elsewhere \cite{Me09Paper, Me10_Sutured_TQFT}. Our work is essentially an elaboration of a TQFT defined by Honda--Kazez--Mati\'{c} \cite{HKM08}. A fundamental principle of topological quantum field theories is that, to topological objects, such as manifolds with some type of structure, are associated algebraic objects, such as vector spaces or modules. To ``fillings-in'' of that structure are associated particular elements, singled out of the algebraic objects.

To an occupied surface, we associate a vector space. To a decomposition of an occupied surface into occupied squares, we associate a tensor decomposition of the vector space into elementary vector spaces. To a map between occupied surfaces (a \emph{decorated occupied surface morphism}), we associate a map of the relevant vector spaces. And when sutures are drawn on an occupied surface, we single out an element of the relevant vector space. This set of associations is required to be coherent in a natural way.

This construction, which we call \emph{sutured quadrangulated field theory} or SQFT, is thus based purely on combinatorial constructions of surfaces and curves, and the principle that these should have algebraic objects representing them. Our main theorem in this paper is a structure theorem for SQFT.
\begin{thm}
\label{thm:main_thm}
Any map of vector spaces obtained in SQFT is a composition of digital creation operators and general digital annihilation operators.
\end{thm}
We will define digital creation and annihilation operators in section \ref{sec:digital_creation_annihilation}.

The construction of SQFT is motivated by the study of contact elements in sutured Floer homology, and it is a description of that theory in the case of product manifolds. (Though it contains vacua, it does not arise from a vacuum!)
\begin{thm}
\label{thm:SFH_gives_SQFT}
The sutured Floer homology of sutured 3-manifolds $(\Sigma \times S^1, F \times S^1)$ (with $\Z_2$ coefficients), contact elements, and maps on $SFH$ induced by inclusions of surfaces $\Sigma \hookrightarrow \Sigma'$, form an SQFT.
\end{thm}

We explain this statement further in section \ref{sec:SFH_and_TQFT}. In any case we may then immediately interpret our main theorem as a theorem about sutured Floer homology.
\begin{cor}
\label{cor:SFH_digital}
Any map on sutured Floer homology $SFH(\Sigma \times S^1, F \times S^1) \To SFH(\Sigma' \times S^1, F \times S^1)$ (with $\Z_2$ coefficients) induced by a surface inclusion is a composition of digital creation operators and general digital annihilation operators.
\end{cor}

Note that our entire construction is carried out over $\Z_2$, so we effectively ignore signs in our algebra. The whole construction should also work with signs, with some minor complications, following constructions as in \cite{Me10_Sutured_TQFT}.

Our general conclusion is that the simple construction of SQFT possesses numerous curious physical analogies. In particular, occupied squares can be regarded as \emph{its}: combinatorial operations of adjoining them, folding them, gluing them, etc, give maps which are directly analogous to creation and annihilation operators in quantum field theory. On the other hand, occupied squares can be regarded as \emph{bits}: there are naturally \emph{two} distinct simplest ways to draw sutures on a square (figure \ref{fig:sutured_squares}), and the creation and annihilation operators can also be understood as performing information processing. This, combined with the extremely elementary and toy-model nature of our construction, explains our title. We refer not only to the children's spider, but also to the speculations of John Archibald Wheeler on fundamental physics \cite{Wheeler90}:
\begin{quotation}
\noindent \emph{It from bit.} Otherwise put, every \emph{it} --- every particle, every field of force, even the spacetime continuum itself --- derives its function, its meaning, its very existence entirely --- even if in some contexts indirectly --- from the apparatus-elicited answers to yes-or-no questions, binary choices, \emph{bits}.
\end{quotation}

\begin{figure}
\begin{center}

\begin{tikzpicture}[
scale=2, 
suture/.style={thick, draw=red}, 
]

\coordinate [label = above left:{$-$}] (1tl) at (0,1);
\coordinate [label = above right:{$+$}] (1tr) at (1,1);
\coordinate [label = below left:{$+$}] (1bl) at (0,0);
\coordinate [label = below right:{$-$}] (1br) at (1,0);

\draw (1bl) -- (1br) -- (1tr) -- (1tl) -- cycle;
\draw [suture] (0.5,0) to [bend left=45] (1,0.5);
\draw [suture] (0,0.5) to [bend right=45] (0.5,1);
\draw (0.2,0.8) node {$-$};
\draw (0.5,0.5) node {$+$};
\draw (0.8,0.2) node {$-$};
\draw (0.5,-0.5) node {$\0$};

\coordinate [label = above left:{$-$}] (2tl) at (3,1);
\coordinate [label = above right:{$+$}] (2tr) at (4,1);
\coordinate [label = below left:{$+$}] (2bl) at (3,0);
\coordinate [label = below right:{$-$}] (2br) at (4,0);

\draw (2bl) -- (2br) -- (2tr) -- (2tl) -- cycle;
\draw [suture] (3.5,1) to [bend right=45] (4,0.5);
\draw [suture] (3.5,0) to [bend right=45] (3,0.5);
\draw (3.8,0.8) node {$+$};
\draw (3.5,0.5) node {$-$};
\draw (3.2,0.2) node {$+$};
\draw (3.5,-0.5) node {$\1$};

\foreach \point in {1bl, 1br, 1tl, 1tr, 2bl, 2br, 2tl, 2tr}
\fill [black] (\point) circle (1pt);

\end{tikzpicture}

\caption{Sets of sutures on the occupied square.}
\label{fig:sutured_squares}
\end{center}
\end{figure}
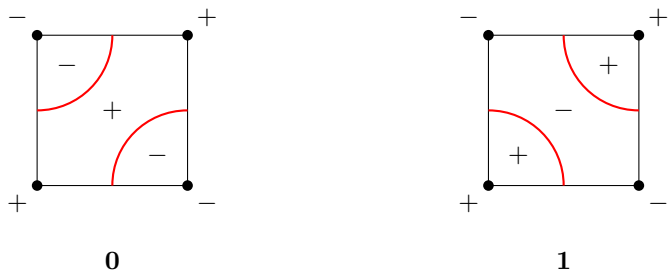

Curious physical analogies go even further. Bits of information stored discretely on a surface recall the \emph{holographic principle}. Our occupied surfaces, quadrangulated by itsy bitsy squares, obtain the structure of a ribbon graph, and sutures drawn on them can be interpreted in terms of $sl(2)$ representation theory, giving the structure of a \emph{spin network} (see section \ref{sec:spin_networks}). The ``stacking'' of surfaces discussed in the previous papers \cite{Me09Paper, Me10_Sutured_TQFT} produces a bilinear form which can be interpreted as an \emph{inner product} on the relevant vector spaces. Though we do not pursue all these questions in detail in this paper, we note them as they arise.

It goes without saying, but to be clear, we do not make any suggestion that our construction has anything to do with the physical world. This paper constructs a mathematical theory of occupied surfaces and sutures, and the algebraic representation of this category as a topological quantum field theory. Our construction is not even a topological quantum field theory, as the term is usually defined, though it is similar. Its objects are inherently 1- and 2-dimensional, so it is at best a toy model for other topological quantum field theories, which are completely understood in dimension 2 (e.g. \cite{Kock04}) and thus mostly interesting in dimensions 3 and above. But the physical analogies are worth noting, and may be of interest in the way they embody mathematically several ideas of fundamental theoretical physics. There are long-standing speculations and work on connections between topology, information, and quantum physics (e.g. \cite{Baez_Stay11, Finkelstein69, Freedman_et_al03}).

In any case, as expressed in corollary \ref{cor:SFH_digital}, our construction provides a combinatorial description of a large family of maps in sutured Floer homology, and may lead to a useful computational tool.

\subsection{Topology of digital creation and annihilation}
\label{sec:digital_creation_annihilation}

Two types of algebraic maps are central to the algebra of this paper. We call them \emph{digital creation} and \emph{digital annihilation}; as the names suggest, we consider them both as performing information processing, or alternatively as creating or destroying particles.

Let ${\bf V}$ be a $2$-dimensional vector space over the field $\Z_2$. (As mentioned, we work over $\Z_2$, but the construction should carry over to $\Z$ coefficients.) At the risk of confusion, but in order to suggest the information-theoretic content, we write a basis as $\{ {\bf 0}, {\bf 1} \}$. Thus ${\bf V} = \Z_2 {\bf 0} \oplus \Z_2 {\bf 1}$. (When we mean the zero element in the vector space, we will write $0$, without bold type.) All the vector spaces we consider will be tensor powers ${\bf V}^{\otimes n}$ of ${\bf V}$, with dimension $2^n$ and basis
\[
\0 \otimes \0 \otimes \cdots \otimes \0, \quad
\0 \otimes \0 \otimes \cdots \otimes \1, \quad
\ldots, \quad
\1 \otimes \1 \otimes \cdots \otimes \1.
\]

A \emph{digital creation operator} is a map ${\bf V}^{\otimes n} \To {\bf V}^{\otimes n} \otimes {\bf V} = {\bf V}^{\otimes (n+1)}$ and comes in two varieties: the $\0$-creation $a^*_\0$ and $\1$-creation $a^*_\1$. These maps ``create'' a $\0$ or $\1$, being defined as follows:
\[
a^*_\0 \; : \; x \mapsto x \otimes \0, \quad \quad \quad
a^*_\1 \; : \; x \mapsto x \otimes \1.
\]
Note that one of the tensor factors in ${\bf V}^{\otimes (n+1)}$ is singled out as the \emph{created} factor, and the newly created $\0$ or $\1$ lies in that factor. Thus there is a ``particle'' for each ${\bf V}$ factor, and there are two types of particle, $\0$ and $\1$

A \emph{digital annihilation operator} is a map ${\bf V}^{\otimes (n+1)} = {\bf V} \otimes {\bf V}^{\otimes n} \To {\bf V}^{\otimes n}$. Again one of the tensor factors is singled out, slated for annihilation, the \emph{annihilated} factor. Again there are two varieties, a $\0$-annihilation $a_0$ and a $\1$-annihilation $a_1$. A $\1$-annihilation attempts to delete a $\1$. If there is a $\1$ in the annihilated factor, the map is simple, $\1 \otimes x \mapsto x$. If there is a $\0$ in the annihilated factor, the factor is of course deleted, but the wrong type of particle has been annihilated. So the annihilation map seeks to compensate by changing a $\1$ to a $\0$ in other factors, so that the overall effect is to delete a $\1$; and it sums over the various possibilities. For instance,
\[
a_\1 \; : \; \0 \otimes (\0 \otimes \1 \otimes \1) \mapsto \0 \otimes \0 \otimes \1 + \0 \otimes \1 \otimes \0.
\]
In particular, $a_\1$ takes $\0 \otimes \0 \otimes \cdots \otimes \0 \mapsto 0$. 

Note that each term of the result has one fewer $\1$ and the same number of $\0$'s. The vector space ${\bf V}^{\otimes n}$ is graded by the numbers of $\0$'s and $\1$'s. Our annihilation and creation operators for $\0$'s and $\1$'s will respectively decrease or increase the relevant grading by $1$. 

Given that the fixed factor ${\bf V}$ is to be annihilated, but the correct ``particle'' may not be present there, this is a natural analogue of the usual annihilation operator in quantum field theory, which ``sums over deleting each particle of the type present''. In general $a_\1$ is given as follows.
\[
\begin{array}{ccccl}
a_1 & : & \0 \otimes x_1 \otimes x_2 \otimes \cdots \otimes x_n &\mapsto&
\sum_{x_i = \1} x_1 \otimes x_2 \otimes \cdots \otimes x_{i-1} \otimes \0 \otimes x_{i+1} \otimes \cdots \otimes x_n \\
& : & \1 \otimes x_1 \otimes x_2 \otimes \cdots \otimes x_n &\mapsto& x_1 \otimes x_2 \otimes \cdots \otimes x_n
\end{array}
\]
In either case, after applying $a_1$, each term has the same number of $\0$'s and one fewer $\1$.

Similarly, a $\0$-annihilation deletes a $\0$ in the annihilated factor, if it is present; else deletes the $\1$ there, and sums over changing $\0$'s to $\1$'s.
\[
\begin{array}{ccccl}
a_0 & : & \0 \otimes x_1 \otimes x_2 \otimes \cdots \otimes x_n &\mapsto& x_1 \otimes x_2 \otimes \cdots \otimes x_n \\
&& \1 \otimes x_1 \otimes x_2 \otimes \cdots \otimes x_n &\mapsto&
\sum_{x_i = \0} x_1 \otimes x_2 \otimes \cdots \otimes x_{i-1} \otimes \1 \otimes x_{i+1} \otimes \cdots \otimes x_n
\end{array}
\]

We will also allow ourselves to have annihilation operators which leave some factors unscathed. A \emph{general digital annihilation operator} is a map of the form
\begin{align*}
a_\0 \otimes 1^{\otimes m} \; &: \; {\bf V}^{\otimes (n+1)} \otimes {\bf V}^{\otimes m} \To {\bf V}^{\otimes n} \otimes {\bf V}^{\otimes m} \quad \text{ or } \\
a_\1 \otimes 1^{\otimes m} \; &: \; {\bf V}^{\otimes (n+1)} \otimes {\bf V}^{\otimes m} \To {\bf V}^{\otimes n} \otimes {\bf V}^{\otimes m}. \\
\end{align*}
(A digital creation operator is ``already general''.)

SQFT associates the above algebraic objects to occupied surfaces and sutures. To an occupied square will be associated the vector space ${\bf V}$. The two simplest sutures on the square, depicted in figure \ref{fig:sutured_squares}, are assigned the values $\0$ and $\1$ respectively. When we construct an occupied surface out of squares, we assign it the vector space ${\bf V}^{\otimes n}$, one factor associated to each square; and when sutures are drawn restricting to the basic sutures on each square, we assign it a tensor product of $\0$'s and $\1$'s accordingly. Thus an occupied surface constructed from $n$ occupied squares holds $n$ particles, or $n$ bits of information.

We can consider embedding one occupied surface inside another. More generally, and importantly for our results, we develop a notion of \emph{occupied surface morphism}. When we have such a morphism including one occupied surface inside another, and some sutures on the complement (which we call a \emph{decorated morphism}), we associate a map of vector spaces. When we create a new square, alongside an occupied surface, the associated map is a digital creation operator. When we close off part of an occupied surface by attaching a square along multiple edges, or fold up some edges, the map obtained is a general digital annihilation operator. Such operations on surfaces may complicate the topology of sutures and quadrangulations; we may need to manipulate and simplify a quadrangulation. We shall describe various elementary operations on occupied surfaces: we call them creations, annihilations, gluings, folds, and zips. Any morphism can be constructed out of these elementary operations. Investigating these details, we shall be able to conclude theorem \ref{thm:main_thm}.

We shall in fact prove something much more specific than theorem \ref{thm:main_thm}. The way in which digital creations and annihilations are composed corresponds in a precise way with the combinatorics of squares in occupied surfaces. In a certain sense, quadrangulations in an occupied surface morphism are nothing more than a diagram of how to apply digital creations and annihilations. In this vague sense, SQFT is the study of the topology of digital creations and annihilations.

\subsection{Quadrangulation and suture mechanics}

To illustrate the kind of situations we consider, see figures \ref{fig:topological_quadrangulation} and \ref{fig:its_and_bits}. In these diagrams we have occupied surfaces which are discs with $12$ or $10$ vertices (in green); let us call these surfaces $(D^2, V_{12})$ and $(D^2, V_{10})$ respectively, where $D^2$ is the surface and $V_{12}, V_{10}$ are the vertices. Occupied surfaces can have arbitrary topology, but we illustrate with a simple example. Each of these vertices has a sign $\pm$ attached, and they alternate around the boundary. An occupied surface is just a surface with alternating signs on boundary vertices in this way (definition \ref{def:occupied_surface_defn}). The $12$ vertices divide the boundary into $12$ \emph{boundary edges} (in thick black).

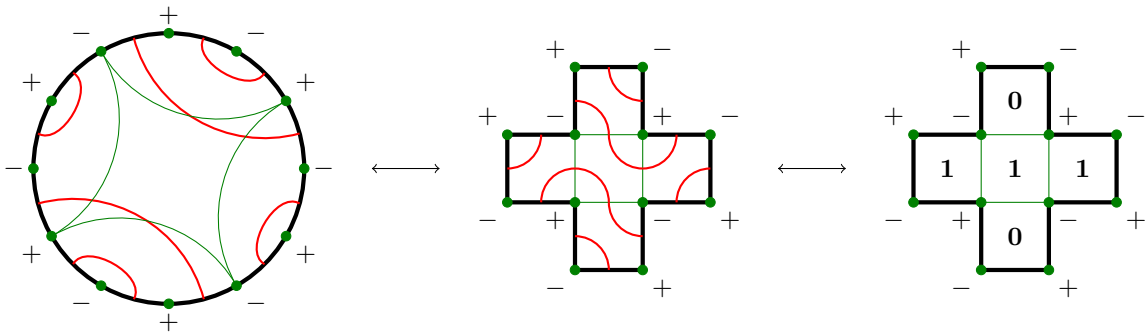
\begin{figure}
\begin{center}

\begin{tikzpicture}[
scale=0.9, 
suture/.style={thick, draw=red},
boundary/.style={ultra thick},
decomposition/.style={draw=green!50!black},
vertex/.style={draw=green!50!black, fill=green!50!black}]

\coordinate [label = above:{$+$}] (12) at (90:2);
\coordinate (1230) at (75:2);
\coordinate [label = above right:{$-$}] (1) at (60:2);
\coordinate (130) at (45:2);
\coordinate [label = above right:{$+$}] (2) at (30:2);
\coordinate (230) at (15:2);
\coordinate [label = right:{$-$}] (3) at (0:2);
\coordinate (330) at (-15:2);
\coordinate [label = below right:{$+$}] (4) at (-30:2);
\coordinate (430) at (-45:2);
\coordinate [label = below right:{$-$}] (5) at (-60:2);
\coordinate (530) at (-75:2);
\coordinate [label = below:{$+$}] (6) at (-90:2);
\coordinate (630) at (-105:2);
\coordinate [label = below left:{$-$}] (7) at (-120:2);
\coordinate (730) at (-135:2);
\coordinate [label = below left:{$+$}] (8) at (-150:2);
\coordinate (830) at (-165:2);
\coordinate [label = left:{$-$}] (9) at (-180:2);
\coordinate (930) at (-195:2);
\coordinate [label = above left:{$+$}] (10) at (-210:2);
\coordinate (1030) at (-225:2);
\coordinate [label = above left:{$-$}] (11) at (-240:2);
\coordinate (1130) at (-255:2);

\coordinate [label = above left:{$+$}] (12b) at (6,1.5);
\coordinate (1230b) at (6.5,1.5);
\coordinate [label = above right:{$-$}] (1b) at (7,1.5);
\coordinate (130b) at (7,1);
\coordinate [label = above right:{$+$}] (2b) at (7,0.5);
\coordinate (230b) at (7.5,0.5);
\coordinate [label = above right:{$-$}] (3b) at (8,0.5);
\coordinate (330b) at (8,0);
\coordinate [label = below right:{$+$}] (4b) at (8,-0.5);
\coordinate (430b) at (7.5,-0.5);
\coordinate [label = below right:{$-$}] (5b) at (7,-0.5);
\coordinate (530b) at (7,-1);
\coordinate [label = below right:{$+$}] (6b) at (7,-1.5);
\coordinate (630b) at (6.5,-1.5);
\coordinate [label = below left:{$-$}] (7b) at (6,-1.5);
\coordinate (730b) at (6,-1);
\coordinate [label = below left:{$+$}] (8b) at (6,-0.5);
\coordinate (830b) at (5.5,-0.5);
\coordinate [label = below left:{$-$}] (9b) at (5,-0.5);
\coordinate (930b) at (5,0);
\coordinate [label = above left:{$+$}] (10b) at (5,0.5);
\coordinate (1030b) at (5.5,0.5);
\coordinate [label = above left:{$-$}] (11b) at (6,0.5);
\coordinate (1130b) at (6,1);

\coordinate [label = above left:{$+$}] (12c) at (12,1.5);
\coordinate [label = above right:{$-$}] (1c) at (13,1.5);
\coordinate [label = above right:{$+$}] (2c) at (13,0.5);
\coordinate [label = above right:{$-$}] (3c) at (14,0.5);
\coordinate [label = below right:{$+$}] (4c) at (14,-0.5);
\coordinate [label = below right:{$-$}] (5c) at (13,-0.5);
\coordinate [label = below right:{$+$}] (6c) at (13,-1.5);
\coordinate [label = below left:{$-$}] (7c) at (12,-1.5);
\coordinate [label = below left:{$+$}] (8c) at (12,-0.5);
\coordinate [label = below left:{$-$}] (9c) at (11,-0.5);
\coordinate [label = above left:{$+$}] (10c) at (11,0.5);
\coordinate [label = above left:{$-$}] (11c) at (12,0.5);

\draw [boundary] (0,0) circle (2 cm);
\draw [suture] (1230) to [bend right=90] (130);
\draw [suture] (1130) to [bend right=45] (230);
\draw [suture] (330) to [bend right=90] (430);
\draw [suture] (530) to [bend right=45] (830);
\draw [suture] (630) to [bend right=90] (730);
\draw [suture] (930) to [bend right=90] (1030);
\draw [decomposition] (11) to [bend right=45] (2);
\draw [decomposition] (2) to [bend right=45] (5);
\draw [decomposition] (5) to [bend right=45] (8);
\draw [decomposition] (8) to [bend right=45] (11);

\draw[<->] (3,0) -- (4,0);

\draw [boundary] (5,0.5) -- (6,0.5) -- (6,1.5) -- (7,1.5) -- (7,0.5) -- (8,0.5) -- (8,-0.5) -- (7,-0.5) -- (7,-1.5) -- (6,-1.5) -- (6,-0.5) -- (5,-0.5) -- cycle;
\draw [decomposition] (6,0.5) -- (7,0.5) -- (7,-0.5) -- (6,-0.5) -- cycle;
\draw [suture] (1230b) arc (180:270:0.5);
\draw [suture] (1130b) arc (90:0:0.5);
\draw [suture] (6.5,0.5) arc (180:270:0.5);
\draw [suture] (230b) arc (0:-90:0.5);
\draw [suture] (330b) arc (90:180:0.5);
\draw [suture] (530b) arc (-90:-180:0.5);
\draw [suture] (6.5,-0.5) arc (0:90:0.5);
\draw [suture] (830b) arc (180:90:0.5);
\draw [suture] (630b) arc (0:90:0.5);
\draw [suture] (930b) arc (-90:0:0.5);

\draw[<->] (9,0) -- (10,0);

\draw [boundary] (12c) -- (1c) -- (2c) -- (3c) -- (4c) -- (5c) -- (6c) -- (7c) -- (8c) -- (9c) -- (10c) -- (11c) -- cycle;
\draw [decomposition] (11c) -- (2c) -- (5c) -- (8c) -- cycle;
\draw (12.5,1) node {$\0$};
\draw (11.5,0) node {$\1$};
\draw (12.5,0) node {$\1$};
\draw (13.5,0) node {$\1$};
\draw (12.5,-1) node {$\0$};

\foreach \point in {1, 2, 3, 4, 5, 6, 7, 8, 9, 10, 11, 12, 1b, 2b, 3b, 4b, 5b, 6b, 7b, 8b, 9b, 10b, 11b, 12b, 1c, 2c, 3c, 4c, 5c, 6c, 7c, 8c, 9c, 10c, 11c, 12c}
\fill [vertex] (\point) circle (2pt);

\end{tikzpicture}

\caption{Topological quadrangulation.}
\label{fig:topological_quadrangulation}
\end{center}
\end{figure}

Refer now to figure \ref{fig:topological_quadrangulation}. There are $4$ arcs drawn in the disc (in green) between vertices of opposite sign. These are \emph{decomposing arcs} and they cut the surface into $5$ occupied squares (i.e. discs with $4$ vertices), forming a \emph{quadrangulation}. The picture is shown in a ``round'' and ``square'' form, which are topologically equivalent. In section \ref{sec:quadrangulations} we study quadrangulations in detail. 

In general, we will show that a quadrangulation of an occupied surface $(\Sigma,V)$ has $G(\Sigma,V) = N - 2 \chi(\Sigma)$ decomposing arcs, where $N = \frac{1}{2} |V|$ (we call this the \emph{gluing number}, section \ref{sec:gluing_number}), and cuts $(\Sigma,V)$ into $I(\Sigma,V) = N - \chi(\Sigma)$ (we call this the \emph{index}, section \ref{sec:index}) occupied squares. The vector space associated to a quadrangulation is ${\bf V}^{\otimes I(\Sigma,V)}$, one tensor factor for each square of the quadrangulation. In this way the number of squares is the ``number of particles''; the ``it'' interpretation. Basic sutures correspond to a basis of the vector space.

In figure \ref{fig:topological_quadrangulation}, in addition to the quadrangulation, a set of curves is drawn in red: these are \emph{sutures} $\Gamma$. Note $\Gamma$ intersects each boundary edge in precisely one point; this is in general a requirement. The sutures divide the disc into several regions, and in each region all the vertices have the same sign; so we can in fact assign a sign to every complementary region of $\Gamma$. In general \emph{sutures} are curves whose complementary regions have coherent signs in this way.

In figure \ref{fig:topological_quadrangulation} we also note that the sutures intersect every decomposing arc in precisely one point. This is not in general a requirement (see e.g. the third row of figure \ref{fig:its_and_bits}); when it occurs we say the sutures are \emph{basic}. Then on each occupied square of the quadrangulation, the sutures come in one of the two standard forms of figure \ref{fig:sutured_squares}, and hence are assigned a value $\0$ or $\1 \in {\bf V}$. In this way sutures are interpreted ``as bits''. So (reading the squares from top to bottom and left to right) the \emph{suture element} of the sutures in the top row is $c(\Gamma) = \0 \otimes \1 \otimes \1 \otimes \1 \otimes \0$; the ``bit'' interpretation.

\begin{figure}
\begin{center}

\begin{tikzpicture}[
scale=1.2, 
suture/.style={thick, draw=red},
decomposition/.style={draw=green!50!black},
vertex/.style={draw=green!50!black, fill=green!50!black},
boundary/.style={ultra thick}]

\coordinate [label = above:{$+$}] (12) at (90:2);
\coordinate (1230) at (75:2);
\coordinate [label = above right:{$-$}] (1) at (60:2);
\coordinate (130) at (45:2);
\coordinate [label = above right:{$+$}] (2) at (30:2);
\coordinate (230) at (15:2);
\coordinate [label = right:{$-$}] (3) at (0:2);
\coordinate (330) at (-15:2);
\coordinate [label = below right:{$+$}] (4) at (-30:2);
\coordinate (430) at (-45:2);
\coordinate [label = below right:{$-$}] (5) at (-60:2);
\coordinate (530) at (-75:2);
\coordinate [label = below:{$+$}] (6) at (-90:2);
\coordinate (630) at (-105:2);
\coordinate [label = below left:{$-$}] (7) at (-120:2);
\coordinate (730) at (-135:2);
\coordinate [label = below left:{$+$}] (8) at (-150:2);
\coordinate (830) at (-165:2);
\coordinate [label = left:{$-$}] (9) at (-180:2);
\coordinate (930) at (-195:2);
\coordinate [label = above left:{$+$}] (10) at (-210:2);
\coordinate (1030) at (-225:2);
\coordinate [label = above left:{$-$}] (11) at (-240:2);
\coordinate (1130) at (-255:2);

\coordinate [label = above:{$+$}] (12b) at ($ (7,0) + (90:2) $);
\coordinate (1230b) at ($ (7,0) + (75:2) $);
\coordinate [label = above right:{$-$}] (1b) at ($ (7,0) + (60:2) $);
\coordinate (130b) at ($ (7,0) + (45:2) $);
\coordinate [label = above right:{$+$}] (2b) at ($ (7,0) + (30:2) $);
\coordinate (230b) at ($ (7,0) + (15:2) $);
\coordinate [label = right:{$-$}] (3b) at ($ (7,0) + (0:2) $);
\coordinate (330b) at ($ (7,0) + (-15:2) $);
\coordinate [label = below right:{$+$}] (4b) at ($ (7,0) + (-30:2) $);
\coordinate (430b) at ($ (7,0) + (-45:2) $);
\coordinate [label = below right:{$-$}] (5b) at ($ (7,0) + (-60:2) $);
\coordinate (530b) at ($ (7,0) + (-75:2) $);
\coordinate [label = below:{$+$}] (6b) at ($ (7,0) + (-90:2) $);
\coordinate (630b) at ($ (7,0) + (-105:2) $);
\coordinate [label = below left:{$-$}] (7b) at ($ (7,0) + (-120:2) $);
\coordinate (730b) at ($ (7,0) + (-135:2) $);
\coordinate [label = below left:{$+$}] (8b) at ($ (7,0) + (-150:2) $);
\coordinate (830b) at ($ (7,0) + (-165:2) $);
\coordinate [label = left:{$-$}] (9b) at ($ (7,0) + (-180:2) $);
\coordinate (930b) at ($ (7,0) + (-195:2) $);
\coordinate [label = above left:{$+$}] (10b) at ($ (7,0) + (-210:2) $);
\coordinate (1030b) at ($ (7,0) + (-225:2) $);
\coordinate [label = above left:{$-$}] (11b) at ($ (7,0) + (-240:2) $);
\coordinate (1130b) at ($ (7,0) + (-255:2) $);

\draw [boundary] (0,0) circle (2 cm);
\draw [decomposition] (11b) to [bend right=45] (2b);
\draw [decomposition] (2b) to [bend right=45] (5b);
\draw [decomposition] (5b) to [bend right=45] (8b);
\draw [decomposition] (8b) to [bend right=45] (11b);
\draw ($ (7,0) + (175:1.5) $) node {$\1$};
\draw (7,0) node {$\1$};
\draw ($ (7,0) + (75:1.5) $) node {$\0$};
\draw ($ (7,0) + (-15:1.5) $) node {$\1$};
\draw ($ (7,0) + (-105:1.5) $) node {$\0$};
\draw [<->, >=triangle 60] (105:2.4) arc (30:210:0.6);

\draw[<->] (3,0) -- (4,0);

\draw [boundary] (7,0) circle (2 cm);
\draw [suture] (1230) to [bend right=90] (130);
\draw [suture] (1130) to [bend right=45] (230);
\draw [suture] (330) to [bend right=90] (430);
\draw [suture] (530) to [bend right=45] (830);
\draw [suture] (630) to [bend right=90] (730);
\draw [suture] (930) to [bend right=90] (1030);
\draw [decomposition] (11) to [bend right=45] (2);
\draw [decomposition] (2) to [bend right=45] (5);
\draw [decomposition] (5) to [bend right=45] (8);
\draw [decomposition] (8) to [bend right=45] (11);
\draw [<->, >=triangle 60] ($ (7,0) + (105:2.4) $) arc (30:210:0.6);

\foreach \point in {1, 2, 3, 4, 5, 6, 7, 8, 9, 10, 11, 12, 1b, 2b, 3b, 4b, 5b, 6b, 7b, 8b, 9b, 10b, 11b, 12b}
\fill [vertex] (\point) circle (1.5pt);

\end{tikzpicture}

\begin{tikzpicture}[
scale=1.2, 
suture/.style={thick, draw=red},
decomposition/.style={draw=green!50!black},
vertex/.style={draw=green!50!black, fill=green!50!black},
boundary/.style={ultra thick}]

\definecolor{newPurple}{rgb}{0.5,0,0.5}

\coordinate (12) at (90:2);
\coordinate (1230) at (75:2);
\coordinate [label = above right:{$-$}] (1) at (60:2);
\coordinate (130) at (45:2);
\coordinate [label = above right:{$+$}] (2) at (30:2);
\coordinate (230) at (15:2);
\coordinate [label = right:{$-$}] (3) at (0:2);
\coordinate (330) at (-15:2);
\coordinate [label = below right:{$+$}] (4) at (-30:2);
\coordinate (430) at (-45:2);
\coordinate [label = below right:{$-$}] (5) at (-60:2);
\coordinate (530) at (-75:2);
\coordinate [label = below:{$+$}] (6) at (-90:2);
\coordinate (630) at (-105:2);
\coordinate [label = below left:{$-$}] (7) at (-120:2);
\coordinate (730) at (-135:2);
\coordinate [label = below left:{$+$}] (8) at (-150:2);
\coordinate (830) at (-165:2);
\coordinate [label = left:{$-$}] (9) at (-180:2);
\coordinate (930) at (-195:2);
\coordinate (10) at (-210:2);
\coordinate (1030) at (-225:2);
\coordinate [label = above left:{$+$}] (11) at (-240:2);
\coordinate (1130) at (-255:2);
\coordinate [label = below left:{$-$}] (extra) at (120:1.5);

\coordinate (12b) at ($ (7,0) + (90:2) $);
\coordinate (1230b) at ($ (7,0) + (75:2) $);
\coordinate [label = above right:{$-$}] (1b) at ($ (7,0) + (60:2) $);
\coordinate (130b) at ($ (7,0) + (45:2) $);
\coordinate [label = above right:{$+$}] (2b) at ($ (7,0) + (30:2) $);
\coordinate (230b) at ($ (7,0) + (15:2) $);
\coordinate [label = right:{$-$}] (3b) at ($ (7,0) + (0:2) $);
\coordinate (330b) at ($ (7,0) + (-15:2) $);
\coordinate [label = below right:{$+$}] (4b) at ($ (7,0) + (-30:2) $);
\coordinate (430b) at ($ (7,0) + (-45:2) $);
\coordinate [label = below right:{$-$}] (5b) at ($ (7,0) + (-60:2) $);
\coordinate (530b) at ($ (7,0) + (-75:2) $);
\coordinate [label = below:{$+$}] (6b) at ($ (7,0) + (-90:2) $);
\coordinate (630b) at ($ (7,0) + (-105:2) $);
\coordinate [label = below left:{$-$}] (7b) at ($ (7,0) + (-120:2) $);
\coordinate (730b) at ($ (7,0) + (-135:2) $);
\coordinate [label = below left:{$+$}] (8b) at ($ (7,0) + (-150:2) $);
\coordinate (830b) at ($ (7,0) + (-165:2) $);
\coordinate [label = left:{$-$}] (9b) at ($ (7,0) + (-180:2) $);
\coordinate (930b) at ($ (7,0) + (-195:2) $);
\coordinate (10b) at ($ (7,0) + (-210:2) $);
\coordinate (1030b) at ($ (7,0) + (-225:2) $);
\coordinate [label = above left:{$+$}] (11b) at ($ (7,0) + (-240:2) $);
\coordinate (1130b) at ($ (7,0) + (-255:2) $);
\coordinate [label = below left:{$-$}] (extrab) at ($ (7,0) + (120:1.5) $);

\draw [boundary] (0,0) circle (2 cm);
\draw [suture] (1230) to [bend right=90] (130);
\draw [suture] (330) to [bend right=90] (430);
\draw [suture] (530) to [bend right=45] (830);
\draw [suture] (630) to [bend right=90] (730);
\draw [suture] (10) .. controls (150:1.7) and ($ 0.5*(11) + 0.5*(extra) + (210:0.2) $)  .. ($ 0.5*(11) + 0.5*(extra) $) .. controls  ($ 0.5*(11) + 0.5*(extra) + (30:0.4) $) and ($ (60:1) + (90:0.4) $) .. (60:1) .. controls ($ (60:1) + (-90:0.6) $) and (15:1.8) .. (230);
\draw [decomposition] (11) -- (extra);
\draw [decomposition] (extra) to [bend right=30] (2);
\draw [decomposition] (2) to [bend right=30] (5);
\draw [decomposition] (5) to [bend right=30] (8);
\draw [decomposition] (8) to [bend right=30] (extra);
\draw [->, >=triangle 60] ($ 0.1*(1) + 0.9*(extra) $) -- ($ 0.9*(1) + 0.1*(extra) $);

\draw[<->] (3,0) -- (4,0);

\draw [boundary] (7,0) circle (2 cm);
\draw [decomposition] (11b) -- (extrab);
\draw [decomposition] (extrab) to [bend right=30] (2b);
\draw [decomposition] (2b) to [bend right=30] (5b);
\draw [decomposition] (5b) to [bend right=30] (8b);
\draw [decomposition] (8b) to [bend right=30] (extrab);
\draw ($ (7,0) + (175:1.5) $) node {$\1$};
\draw (7,0) node {$\1$};
\draw ($ (7,0) + (75:1.5) $) node {$\0$};
\draw ($ (7,0) + (-15:1.5) $) node {$\1$};
\draw ($ (7,0) + (-105:1.5) $) node {$\0$};
\draw [->, >=triangle 60] ($ 0.1*(1b) + 0.9*(extrab) $) -- ($ 0.9*(1b) + 0.1*(extrab) $);

\foreach \point in {1, 2, 3, 4, 5, 6, 7, 8, 9, 11, 1b, 2b, 3b, 4b, 5b, 6b, 7b, 8b, 9b, 11b}
\fill [vertex] (\point) circle (2pt);

\fill [newPurple] (extra) circle (2pt);
\fill [newPurple] (extrab) circle (2pt);

\end{tikzpicture}

\begin{tikzpicture}[
scale=1.2, 
suture/.style={thick, draw=red},
decomposition/.style={draw=green!50!black},
vertex/.style={draw=green!50!black, fill=green!50!black},
boundary/.style={ultra thick}]

\coordinate (12) at (90:2);
\coordinate (1230) at (75:2);
\coordinate [label = above right:{$-$}] (1) at (60:2);
\coordinate (130) at (45:2);
\coordinate [label = above right:{$+$}] (2) at (30:2);
\coordinate (230) at (15:2);
\coordinate [label = right:{$-$}] (3) at (0:2);
\coordinate (330) at (-15:2);
\coordinate [label = below right:{$+$}] (4) at (-30:2);
\coordinate (430) at (-45:2);
\coordinate [label = below right:{$-$}] (5) at (-60:2);
\coordinate (530) at (-75:2);
\coordinate [label = below:{$+$}] (6) at (-90:2);
\coordinate (630) at (-105:2);
\coordinate [label = below left:{$-$}] (7) at (-120:2);
\coordinate (730) at (-135:2);
\coordinate [label = below left:{$+$}] (8) at (-150:2);
\coordinate (830) at (-165:2);
\coordinate [label = left:{$-$}] (9) at (-180:2);
\coordinate (930) at (-195:2);
\coordinate (10) at (-210:2);
\coordinate (1030) at (-225:2);
\coordinate [label = above left:{$+$}] (11) at (-240:2);
\coordinate (1130) at (-255:2);

\coordinate (12b) at ($ (7,0) + (90:2) $);
\coordinate (1230b) at ($ (7,0) + (75:2) $);
\coordinate [label = above right:{$-$}] (1b) at ($ (7,0) + (60:2) $);
\coordinate (130b) at ($ (7,0) + (45:2) $);
\coordinate [label = above right:{$+$}] (2b) at ($ (7,0) + (30:2) $);
\coordinate (230b) at ($ (7,0) + (15:2) $);
\coordinate [label = right:{$-$}] (3b) at ($ (7,0) + (0:2) $);
\coordinate (330b) at ($ (7,0) + (-15:2) $);
\coordinate [label = below right:{$+$}] (4b) at ($ (7,0) + (-30:2) $);
\coordinate (430b) at ($ (7,0) + (-45:2) $);
\coordinate [label = below right:{$-$}] (5b) at ($ (7,0) + (-60:2) $);
\coordinate (530b) at ($ (7,0) + (-75:2) $);
\coordinate [label = below:{$+$}] (6b) at ($ (7,0) + (-90:2) $);
\coordinate (630b) at ($ (7,0) + (-105:2) $);
\coordinate [label = below left:{$-$}] (7b) at ($ (7,0) + (-120:2) $);
\coordinate (730b) at ($ (7,0) + (-135:2) $);
\coordinate [label = below left:{$+$}] (8b) at ($ (7,0) + (-150:2) $);
\coordinate (830b) at ($ (7,0) + (-165:2) $);
\coordinate [label = left:{$-$}] (9b) at ($ (7,0) + (-180:2) $);
\coordinate (930b) at ($ (7,0) + (-195:2) $);
\coordinate (10b) at ($ (7,0) + (-210:2) $);
\coordinate (1030b) at ($ (7,0) + (-225:2) $);
\coordinate [label = above left:{$+$}] (11b) at ($ (7,0) + (-240:2) $);
\coordinate (1130b) at ($ (7,0) + (-255:2) $);

\draw [boundary] (0,0) circle (2 cm);
\draw [suture] (12) to [bend right=90] (130);
\draw [suture] (330) to [bend right=90] (430);
\draw [suture] (530) to [bend right=60] (830);
\draw [suture] (630) to [bend right=90] (730);
\draw [suture] (10) to [bend right=30] (230);
\draw [decomposition] (2) to [bend right=30] (5);
\draw [decomposition] (5) to [bend right=30] (8);
\draw [decomposition] (8) to [bend right=30] (1);

\draw[<->] (3,0) -- (4,0);

\draw [boundary] (7,0) circle (2 cm);
\draw [decomposition] (2b) to [bend right=30] (5b);
\draw [decomposition] (5b) to [bend right=30] (8b);
\draw [decomposition] (8b) to [bend right=30] (1b);
\draw ($ (7,0) + (135:1) $) node {$\0/\1$};
\draw ($ (7,0) + (-60:0.8) $) node {$\1/\0$};
\draw ($ (7,0) + (-15:1.5) $) node {$\1$};
\draw ($ (7,0) + (-105:1.5) $) node {$\0$};

\foreach \point in {1, 2, 3, 4, 5, 6, 7, 8, 9, 11, 1b, 2b, 3b, 4b, 5b, 6b, 7b, 8b, 9b, 11b}
\fill [vertex] (\point) circle (1.5pt);

\end{tikzpicture}

\begin{tikzpicture}[
scale=0.6, 
suture/.style={thick, draw=red},
decomposition/.style={draw=green!50!black},
vertex/.style={draw=green!50!black, fill=green!50!black},
boundary/.style={ultra thick}]

\coordinate (12) at (90:2);
\coordinate (1230) at (75:2);
\coordinate [label = above right:{$-$}] (1) at (60:2);
\coordinate (130) at (45:2);
\coordinate [label = above right:{$+$}] (2) at (30:2);
\coordinate (230) at (15:2);
\coordinate [label = right:{$-$}] (3) at (0:2);
\coordinate (330) at (-15:2);
\coordinate [label = below right:{$+$}] (4) at (-30:2);
\coordinate (430) at (-45:2);
\coordinate [label = below right:{$-$}] (5) at (-60:2);
\coordinate (530) at (-75:2);
\coordinate [label = below:{$+$}] (6) at (-90:2);
\coordinate (630) at (-105:2);
\coordinate [label = below left:{$-$}] (7) at (-120:2);
\coordinate (730) at (-135:2);
\coordinate [label = below left:{$+$}] (8) at (-150:2);
\coordinate (830) at (-165:2);
\coordinate [label = left:{$-$}] (9) at (-180:2);
\coordinate (930) at (-195:2);
\coordinate (10) at (-210:2);
\coordinate (1030) at (-225:2);
\coordinate [label = above left:{$+$}] (11) at (-240:2);
\coordinate (1130) at (-255:2);

\coordinate (12b) at ($ (7,0) + (90:2) $);
\coordinate (1230b) at ($ (7,0) + (75:2) $);
\coordinate [label = above right:{$-$}] (1b) at ($ (7,0) + (60:2) $);
\coordinate (130b) at ($ (7,0) + (45:2) $);
\coordinate [label = above right:{$+$}] (2b) at ($ (7,0) + (30:2) $);
\coordinate (230b) at ($ (7,0) + (15:2) $);
\coordinate [label = right:{$-$}] (3b) at ($ (7,0) + (0:2) $);
\coordinate (330b) at ($ (7,0) + (-15:2) $);
\coordinate [label = below right:{$+$}] (4b) at ($ (7,0) + (-30:2) $);
\coordinate (430b) at ($ (7,0) + (-45:2) $);
\coordinate [label = below right:{$-$}] (5b) at ($ (7,0) + (-60:2) $);
\coordinate (530b) at ($ (7,0) + (-75:2) $);
\coordinate [label = below:{$+$}] (6b) at ($ (7,0) + (-90:2) $);
\coordinate (630b) at ($ (7,0) + (-105:2) $);
\coordinate [label = below left:{$-$}] (7b) at ($ (7,0) + (-120:2) $);
\coordinate (730b) at ($ (7,0) + (-135:2) $);
\coordinate [label = below left:{$+$}] (8b) at ($ (7,0) + (-150:2) $);
\coordinate (830b) at ($ (7,0) + (-165:2) $);
\coordinate [label = left:{$-$}] (9b) at ($ (7,0) + (-180:2) $);
\coordinate (930b) at ($ (7,0) + (-195:2) $);
\coordinate (10b) at ($ (7,0) + (-210:2) $);
\coordinate (1030b) at ($ (7,0) + (-225:2) $);
\coordinate [label = above left:{$+$}] (11b) at ($ (7,0) + (-240:2) $);
\coordinate (1130b) at ($ (7,0) + (-255:2) $);

\coordinate (12c) at ($ (14,0) + (90:2) $);
\coordinate (1230c) at ($ (14,0) + (75:2) $);
\coordinate [label = above right:{$-$}] (1c) at ($ (14,0) + (60:2) $);
\coordinate (130c) at ($ (14,0) + (45:2) $);
\coordinate [label = above right:{$+$}] (2c) at ($ (14,0) + (30:2) $);
\coordinate (230c) at ($ (14,0) + (15:2) $);
\coordinate [label = right:{$-$}] (3c) at ($ (14,0) + (0:2) $);
\coordinate (330c) at ($ (14,0) + (-15:2) $);
\coordinate [label = below right:{$+$}] (4c) at ($ (14,0) + (-30:2) $);
\coordinate (430c) at ($ (14,0) + (-45:2) $);
\coordinate [label = below right:{$-$}] (5c) at ($ (14,0) + (-60:2) $);
\coordinate (530c) at ($ (14,0) + (-75:2) $);
\coordinate [label = below:{$+$}] (6c) at ($ (14,0) + (-90:2) $);
\coordinate (630c) at ($ (14,0) + (-105:2) $);
\coordinate [label = below left:{$-$}] (7c) at ($ (14,0) + (-120:2) $);
\coordinate (730c) at ($ (14,0) + (-135:2) $);
\coordinate [label = below left:{$+$}] (8c) at ($ (14,0) + (-150:2) $);
\coordinate (830c) at ($ (14,0) + (-165:2) $);
\coordinate [label = left:{$-$}] (9c) at ($ (14,0) + (-180:2) $);
\coordinate (930c) at ($ (14,0) + (-195:2) $);
\coordinate (10c) at ($ (14,0) + (-210:2) $);
\coordinate (1030c) at ($ (14,0) + (-225:2) $);
\coordinate [label = above left:{$+$}] (11c) at ($ (14,0) + (-240:2) $);
\coordinate (1130c) at ($ (14,0) + (-255:2) $);

\coordinate (12d) at ($ (21,0) + (90:2) $);
\coordinate (1230d) at ($ (21,0) + (75:2) $);
\coordinate [label = above right:{$-$}] (1d) at ($ (21,0) + (60:2) $);
\coordinate (130d) at ($ (21,0) + (45:2) $);
\coordinate [label = above right:{$+$}] (2d) at ($ (21,0) + (30:2) $);
\coordinate (230d) at ($ (21,0) + (15:2) $);
\coordinate [label = right:{$-$}] (3d) at ($ (21,0) + (0:2) $);
\coordinate (330d) at ($ (21,0) + (-15:2) $);
\coordinate [label = below right:{$+$}] (4d) at ($ (21,0) + (-30:2) $);
\coordinate (430d) at ($ (21,0) + (-45:2) $);
\coordinate [label = below right:{$-$}] (5d) at ($ (21,0) + (-60:2) $);
\coordinate (530d) at ($ (21,0) + (-75:2) $);
\coordinate [label = below:{$+$}] (6d) at ($ (21,0) + (-90:2) $);
\coordinate (630d) at ($ (21,0) + (-105:2) $);
\coordinate [label = below left:{$-$}] (7d) at ($ (21,0) + (-120:2) $);
\coordinate (730d) at ($ (21,0) + (-135:2) $);
\coordinate [label = below left:{$+$}] (8d) at ($ (21,0) + (-150:2) $);
\coordinate (830d) at ($ (21,0) + (-165:2) $);
\coordinate [label = left:{$-$}] (9d) at ($ (21,0) + (-180:2) $);
\coordinate (930d) at ($ (21,0) + (-195:2) $);
\coordinate (10d) at ($ (21,0) + (-210:2) $);
\coordinate (1030d) at ($ (21,0) + (-225:2) $);
\coordinate [label = above left:{$+$}] (11d) at ($ (21,0) + (-240:2) $);
\coordinate (1130d) at ($ (21,0) + (-255:2) $);

\draw [boundary] (0,0) circle (2 cm);
\draw [suture] (330) to [bend right=90] (430);
\draw [suture] (630) to [bend right=90] (730);
\draw [suture] (10) to [bend right=60] (12);
\draw [suture] (830) to [bend right=30] (130);
\draw [suture] (230) to [bend right=60] (530);
\draw [decomposition] (2) to [bend right=30] (5);
\draw [decomposition] (5) to [bend right=30] (8);
\draw [decomposition] (8) to [bend right=30] (1);

\draw (3.5,0) node {$+$};

\draw [boundary] (7,0) circle (2 cm);
\draw [decomposition] (2b) to [bend right=30] (5b);
\draw [decomposition] (5b) to [bend right=30] (8b);
\draw [decomposition] (8b) to [bend right=30] (1b);
\draw [suture] (330b) to [bend right=90] (430b);
\draw [suture] (630b) to [bend right=90] (730b);
\draw [suture] (130b) to [bend right=90] (230b);
\draw [suture] (12b) -- (530b);
\draw [suture] (830b) to [bend right=90] (10b);

\draw [<->] (10,0) -- (11,0);

\draw [boundary] (14,0) circle (2 cm);
\draw [decomposition] (2c) to [bend right=30] (5c);
\draw [decomposition] (5c) to [bend right=30] (8c);
\draw [decomposition] (8c) to [bend right=30] (1c);
\draw ($ (14,0) + (135:1) $) node {$\1$};
\draw ($ (14,0) + (-60:0.8) $) node {$\0$};
\draw ($ (14,0) + (-15:1.5) $) node {$\1$};
\draw ($ (14,0) + (-105:1.5) $) node {$\0$};

\draw (17.5,0) node {$+$};

\draw [boundary] (21,0) circle (2 cm);
\draw [decomposition] (2d) to [bend right=30] (5d);
\draw [decomposition] (5d) to [bend right=30] (8d);
\draw [decomposition] (8d) to [bend right=30] (1d);
\draw ($ (21,0) + (135:1) $) node {$\0$};
\draw ($ (21,0) + (-60:0.8) $) node {$\1$};
\draw ($ (21,0) + (-15:1.5) $) node {$\1$};
\draw ($ (21,0) + (-105:1.5) $) node {$\0$};

\foreach \point in {1, 2, 3, 4, 5, 6, 7, 8, 9, 11, 1b, 2b, 3b, 4b, 5b, 6b, 7b, 8b, 9b, 11b, 1c, 2c, 3c, 4c, 5c, 6c, 7c, 8c, 9c, 11c, 1d, 2d, 3d, 4d, 5d, 6d, 7d, 8d, 9d, 11d}
\fill [vertex] (\point) circle (3pt);

\end{tikzpicture}

\caption{Its and bits.}
\label{fig:its_and_bits}
\end{center}
\end{figure}
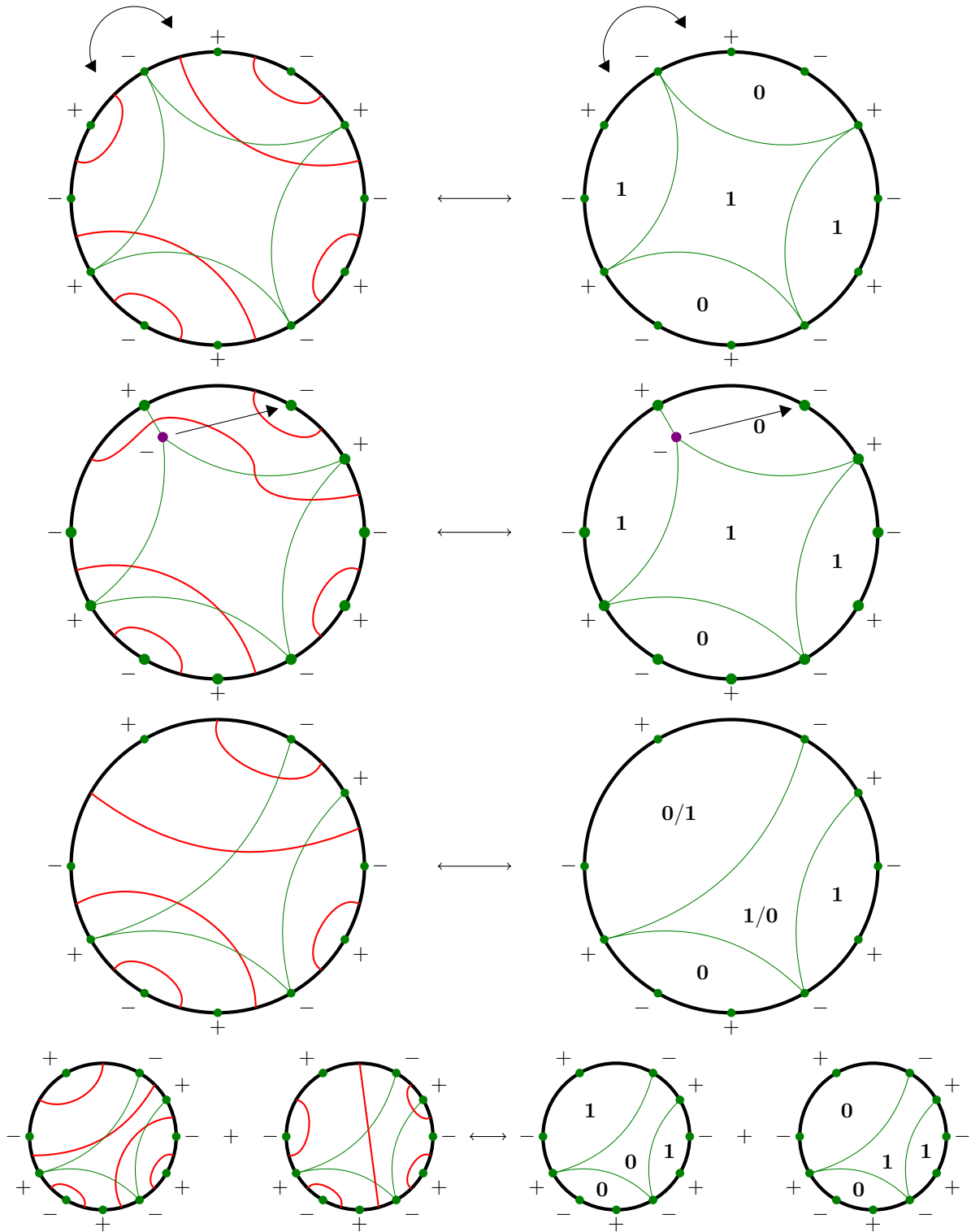

In the top left diagram of figure \ref{fig:its_and_bits}, we see the situation of figure \ref{fig:topological_quadrangulation}, with an arrow between two consecutive boundary edges. These edges are \emph{folded} together to obtain the diagrams on the second row. This fold map is an example of an \emph{occupied surface morphism} $(D^2, V_{12}) \To (D^2, V_{10})$. The fold is a surjective map, but morphisms need not be surjective in general. In essence a morphism is an embedding on the interior of an occupied surface, but edges and vertices can be glued together in coherent fashion; the details ``occupy'' section \ref{sec:morphisms}. The fold results in the arcs and sutures depicted in the second row. This introduces an \emph{internal vertex} (in purple), and what we call a \emph{slack quadrangulation}. The surface is still split into occupied squares, and on each square are basic sutures, so there are still ``bitsy'' interpretations of each.

To get rid of the internal vertex and restore ourselves to a bona fide quadrangulation, we consider moving the internal vertex across a square to a boundary vertex of the same sign, illustrated by an arrow. In doing so we can carry all vertices and edges along the way, and collapse the square we have pushed across out of existence. This \emph{slack square collapse} is the mechanism for ``annihilation of particles'' and results in annihilation operators. 

The result of the slack square collapse is shown in the third row of figure \ref{fig:its_and_bits}. We have a quadrangulation of $(D^2, V_{10})$. Although no sutures are collapsed, the sutures are no longer basic: one of the decomposing arcs intersects the sutures in $3$ points. So there are no longer specific bits associated to the two adjacent squares, ``the squares are entangled''. Note however that with respect to other quadrangulations, the sutures \emph{are} basic; purity/entanglement is in the eye of the quadrangulator--observer.

We have a method, however, for resolving non-basic sutures to basic ones. This is the \emph{bypass relation}, shown in figure \ref{fig:bypass_relation}. Three sets of sutures that are related as shown in that diagram, have suture elements which sum to zero. Note that this is not an axiom of SQFT, but it is \emph{derived} from the mere assignments of vector spaces and suture elements in coherent fashion (we give a precise statement of SQFT and its axioms in section \ref{sec:SQFT_defn}). In any case, using the bypass relation allows us to express the sutures in the third row of figure \ref{fig:its_and_bits} as a sum or ``superposition'' of the two sets of sutures in the fourth row; these are both basic, with bits as shown.

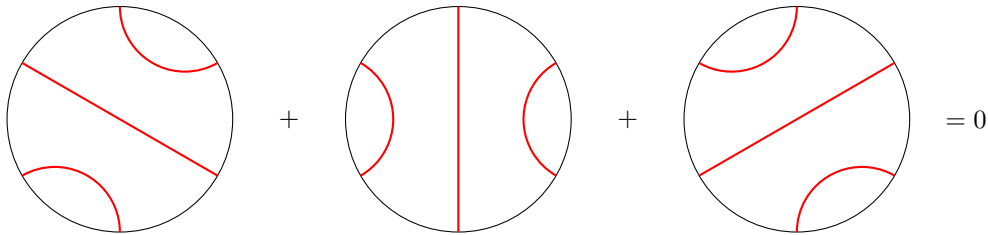
\begin{figure}
\begin{center}

\begin{tikzpicture}[
scale=1.5, 
suture/.style={thick, draw=red}]

\draw (3,0) circle (1 cm); 	
\draw (0,0) circle (1 cm);
\draw (-3,0) circle (1 cm);

\draw [suture] (30:1) arc (120:240:0.57735);
\draw [suture] (0,1) -- (0,-1);
\draw [suture] (210:1) arc (-60:60:0.57735);

\draw [suture] (3,0) ++ (150:1) arc (-120:0:0.57735);
\draw [suture] (3,0) ++ (30:1) -- ($ (3,0) + (210:1) $);
\draw [suture] (3,0) ++ (-30:1) arc (60:180:0.57735);

\draw [suture] (-3,1) arc  (180:300:0.57735);
\draw [suture] ($ (-3,0) + (-30:1) $) -- ($ (-3,0) +  (150:1) $);
\draw [suture] (-3,-1) arc (0:120:0.57735);

\draw (-1.5,0) node {$+$};
\draw (1.5,0) node {$+$};
\draw (4.5,0) node {$=0$};

\end{tikzpicture}

\caption{Bypass relation.}
\label{fig:bypass_relation}
\end{center}
\end{figure}

We conclude this illustration then by considering the result of all the operations considered on suture elements. The effect is (again reading top to bottom, left to right):
	\[
\0 \otimes \1 \otimes \1 \otimes \1 \otimes \0 \mapsto \1 \otimes \0 \otimes \1 \otimes \0 + \0 \otimes \1 \otimes \1 \otimes \0 = (\1 \otimes \0 + \0 \otimes \1) \otimes \1 \otimes \0.
\]
In fact, the operation is ${\bf V}^{\otimes 5} \To {\bf V}^{\otimes 4}$, precisely:
\[
a_\1 \otimes 1^{\otimes 2} \; : \; \left( {\bf V} \otimes {\bf V}^{\otimes 2} \right) \otimes {\bf V}^{\otimes 2}
\To
{\bf V}^{\otimes 2} \otimes {\bf V}^{\otimes 2}.
\]
That is, the effect is a general digital annihilation operator. The factor corresponding to the collapsed square has been annihilated, and the annihilation operator has been applied to the squares which were adjacent to the collapsed internal vertex. The annihilation is a $\1$-annihilation, corresponding to the fact that the collapsed internal vertex had a $-$ sign.

The point of this paper is to develop a theory to make the above discussion rigorous, and demonstrate that the above phenomena are general for occupied surfaces, quadrangulations, and sutures.

\subsection{Sutured Floer homology and TQFT}
\label{sec:SFH_and_TQFT}

As mentioned, all the constructions in this paper are inspired by the study of contact elements in sutured Floer homology ($SFH$) of product manifolds; in particular, sutured 3-manifolds of the form $(\Sigma \times S^1, F \times S^1)$, where $\Sigma$ is a surface with nonempty boundary and $F \subset \partial \Sigma$ is a finite set.

The study of the sutured Floer homology of these manifolds, and morphisms between them, was initiated by Honda--Kazez--Mati\'{c} in \cite{HKM08}. In that paper, Honda--Kazez--Mati\'{c} proved that any \emph{inclusion} of sutured manifolds $(M, \Gamma) \hookrightarrow (M', \Gamma')$, one into the interior of the other, together with a contact structure $\xi$ on the complement sutured manifold $(M' \backslash M, \Gamma \cup \Gamma')$, gives rise to a map on sutured Floer homology, $SFH(-M, -\Gamma) \To SFH(-M', -\Gamma')$. (The minus signs indicate an issue with orientations which is irrelevant for our purposes.) They then considered the \emph{dimensionally-reduced} case where the manifolds are of the form $(\Sigma \times S^1, F \times S^1)$ and the inclusion of sutured 3-manifolds is induced by an inclusion of a surface into the interior of another, $\Sigma \hookrightarrow \Sigma'$.

This study was continued by the author in \cite{MyThesis, Me09Paper, Me10_Sutured_TQFT}. In \cite{MyThesis, Me09Paper} the author was able to describe all the contact elements in $SFH(D^2 \times S^1, F \times S^1)$ completely and explicitly. In \cite{Me10_Sutured_TQFT} the author found various curious connections to noncommutative quantum field theory in $SFH(D^2 \times S^1, F \times S^1)$, although the algebra there is different from the framework of tensor products based on quadrangulations which we pursue here. (In particular the creation and annihilation operators defined there are quite different from the ones defined here.) This paper continues that study.

Theorem \ref{thm:SFH_gives_SQFT}, that SQFT describes $SFH$ of product manifolds, is in essence a reformulation of section 4 of \cite{Me10_Sutured_TQFT}, adapted to a slightly different context. In fact, our notion of occupied surface morphism includes surface inclusions, but is broader, so theorem \ref{thm:main_thm} is slightly more general than corollary \ref{cor:SFH_digital} suggests.

The sutured manifolds $(\Sigma \times S^1, F \times S^1)$ considered here are a very specific class. Connections with physics are pervasive in Floer homology and gauge theories, and of course more generally in symplectic and contact geometry; but the connections examined here seem to be of quite a different nature. We wonder whether such connections extend to a more general class of sutured $3$-manifolds or to related theories such as embedded contact homology. Mild generalisations of manifolds of this type are considered, for instance, by Golovko in \cite{Golovko09, Golovko10}, where sutures have a non-vertical slope, and by Wendl in \cite{Wendl10}; contact structures on our product manifolds essentially give the trivial-monodromy case of \emph{planar torsion domains}. Can digital creations and annihilations help us to understand contact elements in sutured Floer homology in more generality?

\subsection{Structure of this paper}

In order to prove the main theorem \ref{thm:main_thm}, a substantial amount of development is necessary. In section \ref{sec:occupied_surfaces} we define occupied surfaces and various useful notions related to them. In section \ref{sec:morphisms} we define a broad class of morphisms between these surfaces, prove that they form a category, discuss various properties they have, and introduce a set of simple elementary morphisms from which all others will be constructed. In section \ref{sec:quadrangulations}, we study quadrangulations of occupied surfaces in detail, and use them to show how to build surfaces and morphisms out of elementary pieces. In section \ref{sec:sutures} we study sutures, in detail; and then in section \ref{sec:quadrangulations_and_sutures} we study how they relate to the structure of occupied and quadrangulated surfaces.

So, not until section \ref{sec:decorated_morphisms} can we introduce decorated morphisms, which involve both occupied surfaces and sutures. It is decorated morphisms which are represented algebraically by SQFT, which is defined and studied in section \ref{sec:SQFT}, proving our main theorems.

\section{Occupied surfaces}
\label{sec:occupied_surfaces}

\subsection{Definitions}
\label{sec:occupied_surface_definitions}

\begin{defn}
\label{def:occupied_surface_defn}
An \emph{occupied surface} $(\Sigma,V)$ is a pair $(\Sigma, V)$, where $\Sigma$ is an oriented surface (possibly disconnected), every component of $\Sigma$ has nonempty boundary, and $V$ is a finite set of points on $\partial \Sigma$ called \emph{vertices}, each labelled \emph{positive} or \emph{negative}. Each component of $\partial \Sigma$ must contain vertices of $V$, and along each component of $\partial \Sigma$, vertices must alternately be labelled as positive and negative.
\end{defn}

\begin{figure}
\centering
\includegraphics[scale=0.4]{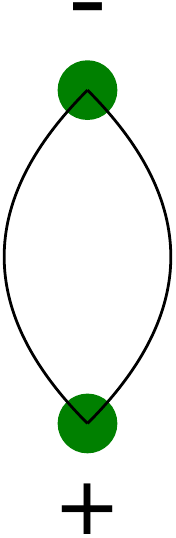}
\quad \quad
\includegraphics[scale=0.4]{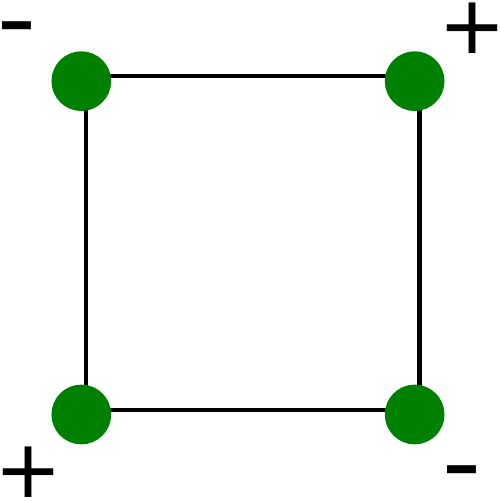}
\quad \quad 
\includegraphics[scale=0.4]{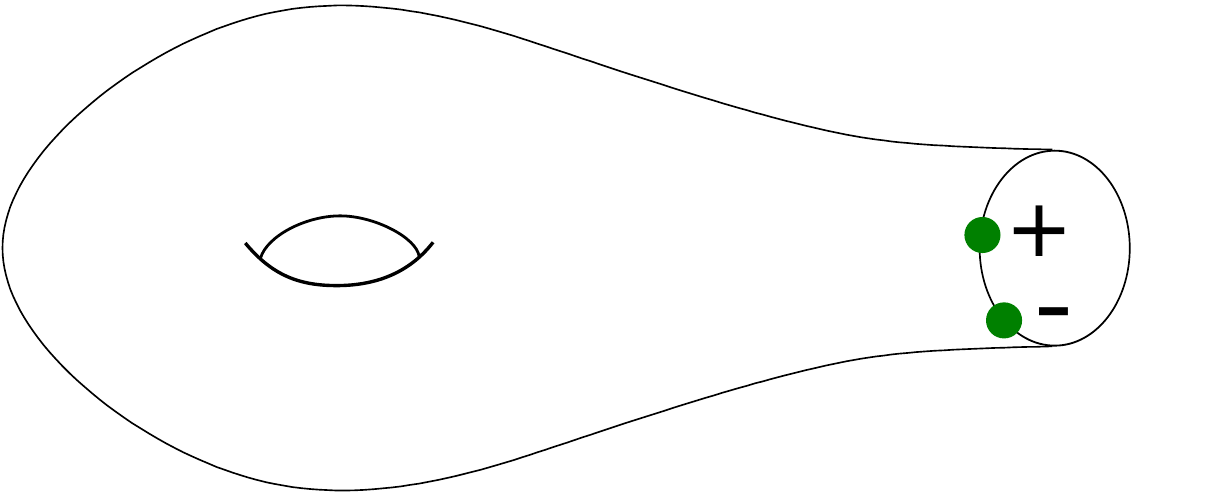}
\caption{Occupied surfaces}
\label{fig:occupied_surfaces}
\end{figure}

Figure \ref{fig:occupied_surfaces} shows examples of occupied surfaces. We can think of the vertices as corners: in an occupied square there are signs on every corner. The set of positive vertices is denoted $V_+$ and the set of negative vertices is denoted $V_-$, so $V = V_+ \sqcup V_-$. The alternating sign property implies a positive even number of vertices on each boundary component. We shall usually write $N$ for the number of positive vertices, so that there are also $N$ negative vertices, $2N$ total.

We can speak of a homeomorphism of occupied surfaces, being a surface homeomorphism which maps vertices bijectively and preserves their signs.

The connected components of an occupied surface carry the structure of occupied surfaces.

The arcs of $\partial \Sigma$ connecting consecutive vertices, we call \emph{boundary edges}. A boundary edge has one endpoint in $V_-$ and one endpoint in $V_+$, hence is naturally oriented from $V_-$ to $V_+$. As $\Sigma$ is oriented, the boundary inherits an orientation, and around each boundary component, the boundary edges alternate between agreeing and disagreeing in orientation with the boundary orientation.
\begin{defn}
A boundary edge of $(\Sigma, V)$ which disagrees in orientation with $\partial \Sigma$ is \emph{incoming}. A boundary edge which agrees in orientation with $\partial \Sigma$ is \emph{outgoing}.
\end{defn}

\begin{defn}[Simple occupied surfaces]\
\begin{enumerate}
\item
The \emph{occupied vacuum} $(\Sigma^\emptyset, V^\emptyset)$ is the disc with two vertices.
\item
The \emph{occupied square} $(\Sigma^\square, V^\square)$ is the disc with four vertices.
\end{enumerate}
\end{defn}
The occupied vacuum is a bigon, with one vertex of each sign. The vertices of the occupied square alternate, opposite corners have the same sign. Both are shown in figure \ref{fig:occupied_surfaces}.

We will often prefer occupied surfaces in which no component is the occupied vacuum: we say such surfaces are \emph{without vacua}.

\subsection{Arcs in occupied surfaces}

\begin{defn}
A \emph{decomposing arc} in an occupied surface $(\Sigma,V)$ is a properly embedded arc in $(\Sigma,V)$, with one endpoint in $V_+$, the other endpoint in $V_-$, and interior in the interior of $\Sigma$.
\end{defn}

\begin{figure}
\begin{center}
\begin{tabular}{c}
\begin{tikzpicture}[
scale=1.2, 
decomposition/.style={thick, draw=green!50!black}, 
vertex/.style = {draw=green!50!black, fill=green!50!black}
]
\draw (0:1) -- (60:1) -- (120:1) -- (180:1) -- (240:1) -- (300:1) -- cycle;
\draw [decomposition] (0:1) -- (180:1);
\foreach \angle in {0, 60, 120, 180, 240, 300}
\fill [vertex] (\angle:1) circle (2pt);
\foreach \angle in {0,120,240}
\draw (\angle:1.3) node {$-$};
\foreach \angle in {60,180,300}
\draw (\angle:1.3) node {$+$};
\end{tikzpicture}
\end{tabular}
\begin{tabular}{c}
\includegraphics[scale=0.4]{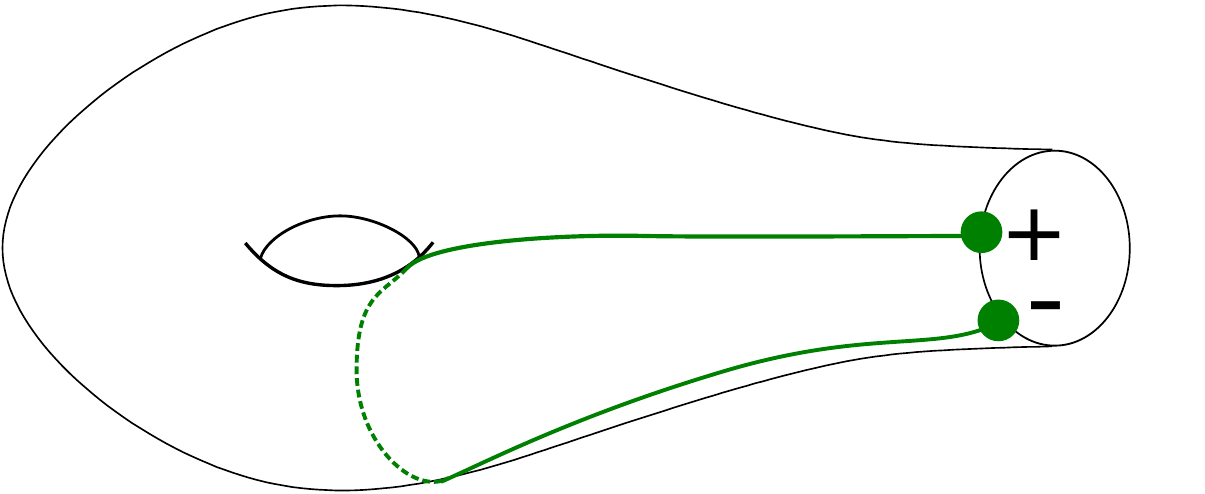}
\end{tabular}
\caption{Decomposing arcs}
\label{fig:decomposing_arcs}
\end{center}
\end{figure}
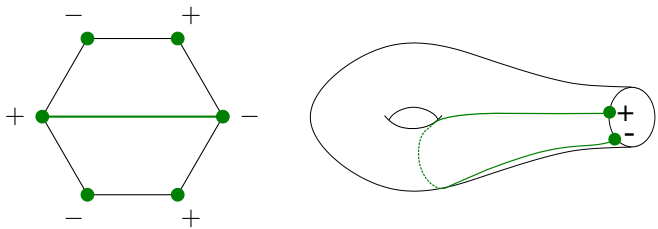

Given two consecutive vertices $v_-, v_+$ of an occupied surface $(\Sigma,V)$, there are decomposing arcs $a$ connecting them, which are isotopic (rel endpoints) to a boundary edge between $v_-$ and $v_+$. Cutting $(\Sigma, V)$ along such a decomposing arc gives two occupied surfaces, one homeomorphic to $(\Sigma,V)$, and the other an occupied vacuum. Any occupied surface thereby decomposes into itself and a vacuum; this is a trivial kind of decomposition, and we call such a decomposing arc \emph{trivial}.

Some nontrivial decomposing arcs are shown in figure \ref{fig:decomposing_arcs}. Note that every decomposing arc in an occupied vacuum, or occupied square, is trivial. The converse is also true.
\begin{lem}
\label{lem:nontrivial_arc_exists}
Let $(\Sigma,V)$ be a connected occupied surface. There exists a nontrivial decomposing arc on $(\Sigma,V)$ if and only if $(\Sigma,V)$ is not the vacuum or the occupied square.
\end{lem}

\begin{proof}
If $\Sigma$ has nonzero genus, there is a nontrivial decomposing arc around a handle. If $\Sigma$ has more than one boundary component, there is a nontrivial decomposing arc connecting distinct boundary components. If $\Sigma$ is a disc with more than $4$ vertices, there is a decomposing arc connecting non-consecutive vertices, hence nontrivial. 
\end{proof}

\subsection{Index}
\label{sec:index}

\begin{defn}
\label{def:index}
The \emph{index} of an occupied surface $(\Sigma,V)$ is $I(\Sigma,V) = N - \chi(\Sigma)$.
\end{defn}

Here $\chi(\Sigma)$ denotes the Euler characteristic of $\Sigma$. For $\Sigma$ connected, with genus $g$ and $B$ boundary components, then $I(\Sigma,V) = N - 2 + 2g + B$. In general, if $\Sigma$ has $C$ connected components, $B$ boundary components and total genus $g$, then $I(\Sigma,V) = N - 2C + 2g + B$. Note $B \geq C$ (every component of $\Sigma$ has boundary) and $N \geq B$ (each boundary component has vertices), so $I(\Sigma,V) \geq 0$ always.

Since $N$ and $\chi$ are additive under disjoint union, we have immediately that
\[
I \left( \sqcup_i \left( \Sigma_i, V_i \right) \right) = \sum_i I \left( \Sigma_i, V_i \right).
\]

The occupied vacuum has index $0$. Conversely, if a connected occupied surface has index $0$ then $N-2+2g+B =0$, so $N=B=1$ and $g=0$; hence it is the occupied vacuum. As vacua have index $0$ and may be peeled off occupied surfaces at will, we can view them as ``ephemeral'' or ``nonexistent''.

The occupied square has index $1$. Conversely, if a connected occupied surface has index $1$ then $N-2+2g+B =1$; using $N \geq B \geq 1$ this implies $g=0$, $N=2$, $B=1$. So it is the occupied square. Given their index $1$ and, as we shall see shortly, the ability to decompose any occupied surface without vacua into occupied squares, we regard them as ``atomic''.

\subsection{Gluing number}
\label{sec:gluing_number}

\begin{defn}
The \emph{gluing number} of an occupied surface $(\Sigma,V)$ is $G(\Sigma,V) = N - 2 \chi(\Sigma)$.
\end{defn}
The gluing number will tell us how many times we have glued squares together to obtain a surface. With notation as above, for connected $(\Sigma,V)$ we have $G(\Sigma,V) = N+4g+2B-4$ and in general $G(\Sigma,V)= N + 4g + 2B - 2C$.

Since $N$ and $\chi$ are additive under disjoint union, we have
\[
G \left( \sqcup_i \left( \Sigma_i, V_i \right) \right) = \sum_i G \left( \Sigma_i, V_i \right).
\]

The gluing number of the occupied vacuum is $-1$. Conversely, a connected occupied surface with $G(\Sigma,V)<0$ satisfies $N+4g+2B-4<0$, which implies $g=0$, $N=1$, $B=1$; so it is the occupied vacuum and the gluing number is  $-1$.

The gluing number of the occupied square is $0$. Conversely, a connected occupied surface with $G(\Sigma,V)=0$ satisfies $N+4g+2B-4=0$, which with $N \geq B \geq 1$ implies $g=0$, $N=2$, $B=1$; so it is the occupied square.

Thus the gluing number of any connected occupied surface other than the vacuum or square is positive. Moreover, an occupied surface without vacua has gluing number zero if and only if it is a disjoint union of squares.

\subsection{Decomposition and gluing}
\label{sec:decomposition_and_gluing}

Cutting an occupied surface $(\Sigma,V)$ along a decomposing arc $a$ produces another occupied surface $(\Sigma', V')$. Write $g,B,N,C$ and $g',B',N',C'$ for their respective topological data, as above. Any $a$ falls into precisely one of the following possibilities.
\begin{enumerate}
\item
The arc $a$ connects two distinct components of $\partial \Sigma$. Then $a$ is nonseparating and $C'=C$, $g'=g$, $B'=B-1$, $N'=N+1$.
\item
The arc $a$ has both endpoints on the same component of $\partial \Sigma$, and is nonseparating. Then $a$ is not boundary parallel, and $C'=C$, $g'=g-1$, $B'=B+1$, $N'=N+1$.
\item
The arc $a$ has both endpoints on the same component of $\partial \Sigma$, and is separating. Thus $a$ is boundary parallel and cutting along $a$ cuts off a disc; this includes the case of $a$ trivial. We have $C'=C+1$, $g'=g$, $B'=B+1$, $N'=N+1$.
\end{enumerate}

We immediately verify that $N'=N+1$ and $\chi(\Sigma',V') = \chi(\Sigma,V)+1$ in all cases. 
\begin{lem}
\label{lem:index_invariant}
Cutting along a decomposing arc preserves index, $I(\Sigma',V') = I(\Sigma,V)$, and decreases gluing number by one, $G(\Sigma',V') = G(\Sigma,V) - 1$.
\qed
\end{lem}

Consider now the inverse procedure of \emph{gluing}. Take an occupied surface $(\Sigma,V)$ and distinct boundary edges $e_1, e_2$. Glue them together, respecting signs of vertices; this gluing map is orientation-reversing, and the result $\Sigma'$ orientable, if and only if one of $e_1, e_2$ is incoming and the other is outgoing. 

Let $V' \subset \Sigma'$ denote the image of $V$ under this gluing; each point of $V'$ inherits a sign. Note that if $e_1, e_2$ were consecutive edges then there are elements of $V'$ in the interior of $\Sigma'$: in this case we say that the gluing has \emph{swallowed} those vertices. However if $e_1, e_2$ are not consecutive edges, then $V' \subset \partial \Sigma'$, and $\partial \Sigma'$ consists of the image of boundary edges of $(\Sigma,V)$ naturally oriented from $V'_-$ to $V'_+$. 
\begin{lem}
Let $e_1,e_2$ be non-consecutive boundary edges, one incoming and one outgoing, of $(\Sigma,V)$. Then the result of gluing $e_1$ to $e_2$, respecting signs, is an occupied surface $(\Sigma',V')$.
\qed
\end{lem}
In the glued surface $(\Sigma',V')$, the glued boundary edges $e_1, e_2$ are identified to a decomposing arc.

\section{Morphisms of occupied surfaces}
\label{sec:morphisms}

\subsection{Definition and properties}

We now wish to consider useful types of maps between occupied surfaces. For our purposes, we would like to allow maps which are something like embeddings, but slightly more general; they are something of a combinatorial version of an embedding. We shall call these \emph{occupied surface morphisms}. For their eventual use we shall need to combine occupied surface morphisms with sutures to obtain the notion of a \emph{decorated morphism}.

\begin{figure}
\begin{center}
\begin{tabular}{c}
\includegraphics[scale=0.3]{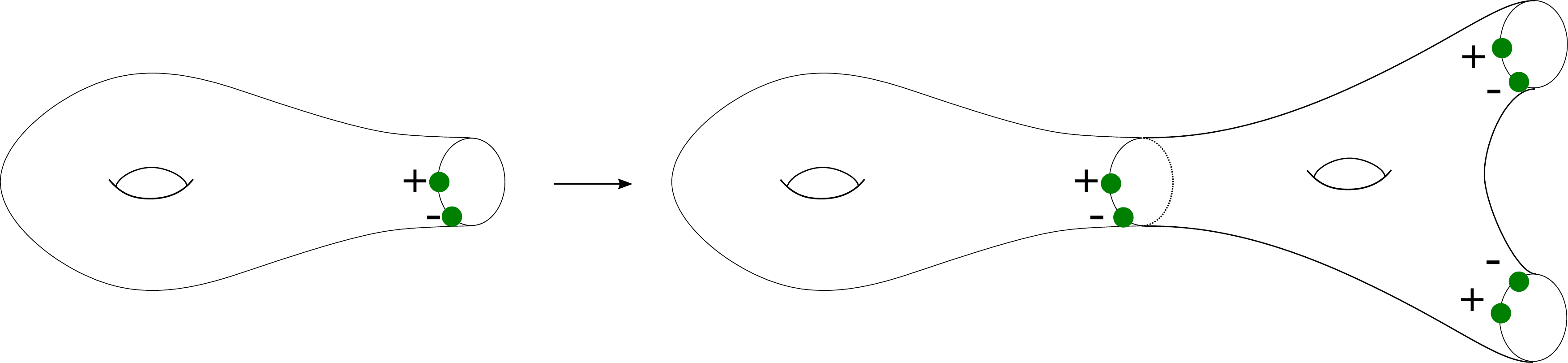}\\
\includegraphics[scale=0.4]{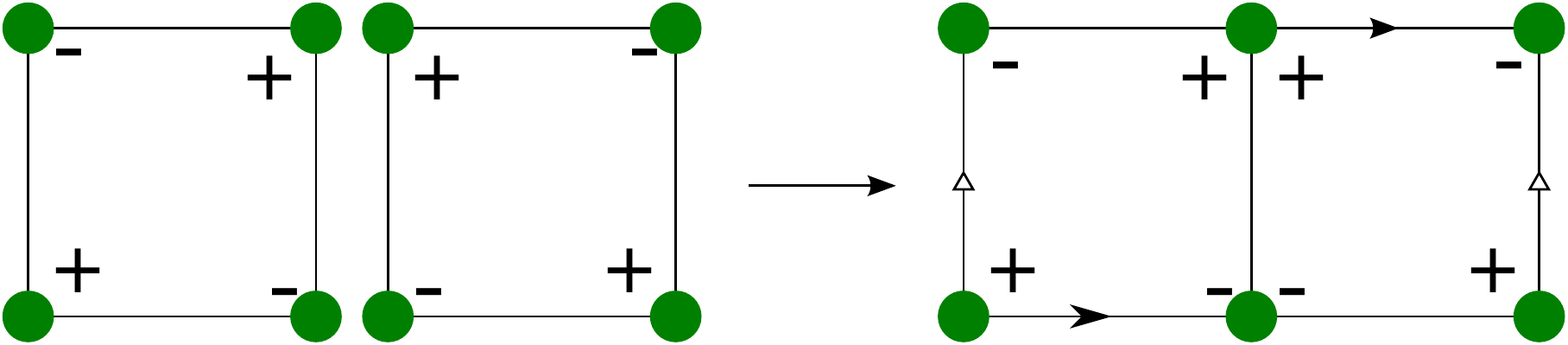}\\
\includegraphics[scale=0.4]{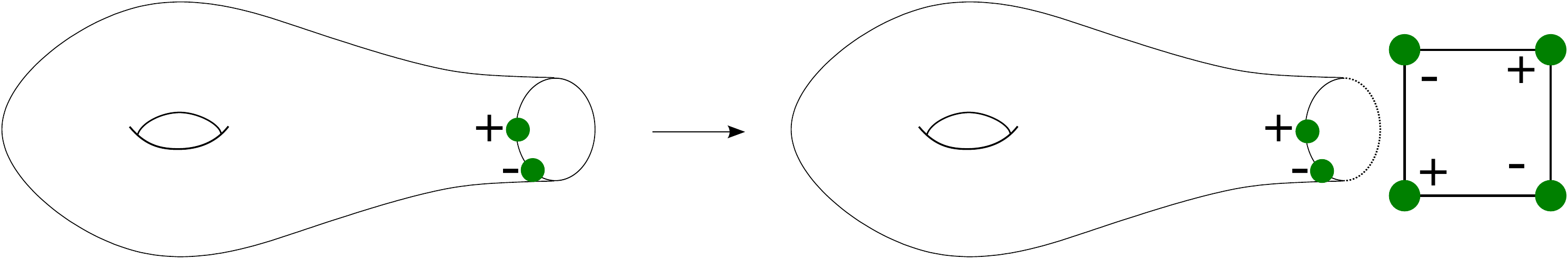}\\
\includegraphics[scale=0.4]{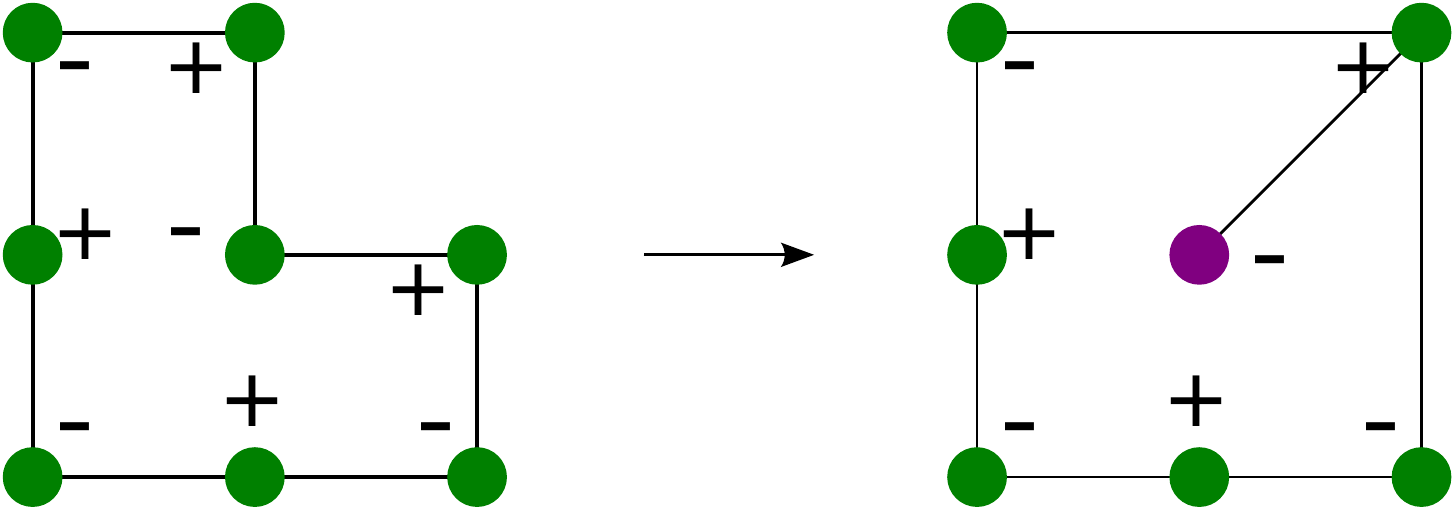}
\end{tabular}
\caption{Examples of occupied surface morphisms.}
\label{fig:examples_of_morphisms}
\end{center}
\end{figure}

Take any occupied surfaces $(\Sigma,V)$, $(\Sigma',V')$ and a continuous map $\phi: (\Sigma,V) \To (\Sigma',V')$. There are then distinguished points $\phi(V) \cup V'$ in $\Sigma'$, which can be given signs, and distinguished arcs between them, whose union is $\phi(\partial \Sigma) \cup \Sigma'$.

We shall require a morphism to be an embedding of the interior of $\Sigma$ into the interior of $\Sigma'$. We shall permit a range of behaviour on the boundary, but we require that the essential structure of signs on vertices are preserved, and edges between them must not behave badly. Non-consecutive edges may be glued as discussed above; but also consecutive edges may be glued. In this latter case, edges $e_1, e_2$ which share a vertex $v$ are ``folded'' together and $\phi(v)$ lies in the interior of $\phi(\Sigma)$, swallowed by $\phi$. 
\begin{defn}
A vertex $v \in V$ is \emph{swallowed} by $\phi$ is $\phi(v)$ lies in the interior of $\phi(\Sigma)$.
\end{defn}

\begin{defn}
\label{defn:morphism}
A \emph{morphism} $\phi: (\Sigma, V) \To (\Sigma', V')$ is a continuous map $\Sigma \To \Sigma'$ such that:
\begin{enumerate}
\item
$\phi$ is an embedding of the interior of $\Sigma$ into $\Sigma'$.
\item
For each boundary edge $e$ of $(\Sigma, V)$, $\phi|_e$ is a homeomorphism onto its image.
\item
Distinguished arcs in $\Sigma'$ (boundary edges of $(\Sigma',V')$ or images of boundary edges of $(\Sigma,V)$) which intersect other than at endpoints, coincide.
\item
Distinguished points have well-defined signs: positive signed points $\phi(V_+) \cup V'_+$ on $\Sigma'$ are disjoint from negative ones $\phi(V_-) \cup V'_-$.
\end{enumerate}
\end{defn}

Figure \ref{fig:examples_of_morphisms} shows several examples of morphisms.

Consider the conditions on edges. Take distinct boundary edges $e_1, e_2$ of $(\Sigma,V)$. As homeomorphic images of oriented arcs running from negative to positive vertices, $\phi(e_1)$ and $\phi(e_2)$ are both oriented embedded arcs in $\Sigma'$, running from negative to positive signed points. The intersection condition says they are either disjoint, or they intersect at one or both endpoints, or they coincide. Two edges $e_1, e_2$ may therefore be glued together by $\phi$, and since $(\Sigma',V')$ is by definition orientable, the gluing map must be orientation-reversing, so one of $e_1, e_2$ must be incoming and the other outgoing. Edges which are glued together may be non-consecutive, producing a gluing as discussed previously; or they may be consecutive and have the effect of folding the edges together, swallowing a vertex. It is not however possible for three distinct edges to be glued together; there are only two sides of an arc in a surface, and $\phi$ must be an embedding on the interior. 

Consider the conditions on vertices. Take a vertex $v \in V$. Its image $\phi(v)$ may be swallowed, lying in the interior of $\phi(\Sigma)$; in this case the edges adjacent to $v$ are both glued (though not necessarily to each other). If it is not swallowed, then it lies on the boundary of $\phi(\Sigma)$. The image $\phi(v)$ may or may not lie on $\partial \Sigma'$, but if $\phi(v) \in \partial \Sigma'$, then $\phi(v) \in V'$, since a vertex lying in the interior of a boundary edge of $(\Sigma',V')$ contradicts the condition on edge intersections. Many vertices of $V$ can be mapped by $\phi$ to the same point.

We can justify the category-theoretic usage of ``morphism''. The identity map $(\Sigma,V) \To (\Sigma,V)$ is the identity morphism on $(\Sigma,V)$.
\begin{lem}
Let $\phi_1: (\Sigma^1, V^1) \To (\Sigma^2, V^2)$ and $\phi_2: (\Sigma^2, V^2) \To (\Sigma^3, V^3)$ be morphisms. Then the composition $\phi_2 \circ \phi_1: (\Sigma^1, V^1) \To (\Sigma^3, V^3)$ is a morphism.
\end{lem}

\begin{proof}
We must verify the four conditions of the definition of morphism. The first is immediate: since $\phi_1$ and $\phi_2$ embed interiors, so does $\phi_2 \circ \phi_1$.

If (iv) does not hold, then $\phi_2 \circ \phi_1$ maps two vertices $v_- \in V^1_-$, $v_+ \in V^1_+$ to the same point, or maps a vertex $v \in V^1$ to a vertex of $V^3$ of opposite sign.

In the first case, $\phi_2 \circ \phi_1 (v_+) = \phi_2 \circ \phi_1 (v_-)$. As $\phi_1$ satisfies (iv), $\phi_1 (v_+) \neq \phi_1 (v_-)$. If one of these is in the interior of $\Sigma^2$, then $\phi_2$ embeds it into the interior of $\Sigma^3$, disjoint from the image of $\phi_1(v_-)$, but $\phi_2$ maps another point there also, a contradiction. Otherwise both $\phi_1 (v_+), \phi_1 (v_-)$ lie in $\partial \Sigma^2$, hence $\phi_1(v_+) \in V^2_+$ and $\phi_1 (v_-) \in V^2_-$; then as $\phi_2$ satisfies (iv), $\phi_2(\phi_1(v_+)) \neq \phi_2(\phi_1(v_-))$, a contradiction.

In the second case, suppose without loss of generality $v \in V^1_-$ and $\phi_2 \circ \phi_1 (v) \in V^3_+$. If $\phi_1(v)$ lies in the interior of $\Sigma^2$, then $\phi_2$ embeds it into the interior of $\Sigma^3$, hence not into $V^3$. So $\phi_1(v)$ lies in $\partial \Sigma^2$, hence in $V^2_-$. And then as $\phi_2$ maps $\phi_1(v)$ to $V^3$, it must preserve sign, and $\phi_2 \circ \phi_1 (v) \in V^3_-$, a contradiction.

For (ii): Let $e$ be an edge of $(\Sigma^1, V^1)$; so $\phi_1$ maps $e$ homeomorphically onto its image. By property (iii) of $\phi_1$, there are four possibilities for $\phi_1(e)$: it lies in the interior of $\Sigma^2$; or it intersects $\partial \Sigma^2$ in a single endpoint; or in both endpoints; or $\phi_1(e)$ is a boundary edge of $(\Sigma^2, V^2)$.

If $\phi_1(e)$ lies in the interior of $\Sigma^2$, where $\phi_2$ is an embedding, then it is mapped by $\phi_2$ homeomorphically onto its image. If $\phi_1(e)$ has precisely one endpoint $v^2$ in $\partial \Sigma^2$, then $\phi_2$ is obviously a homeomorphism on $\phi_1(e) \backslash \{v^2\}$; which extends to a homeomorphism of $\phi_1 (e)$. If $\phi_2(e)$ intersects $\partial \Sigma^2$ in both endpoints, then $\phi_2$ is a homeomorphism on the interior of $e$; as $\phi_2$ respects signs of vertices, the two endpoints of $\phi_1(e)$ are mapped to distinct points in $\Sigma^2$ and hence $\phi_2$ is a homeomorphism on $\phi_1(e)$. If $\phi_1(e)$ is a boundary edge of $(\Sigma^2, V^2)$, then by property (ii) of $\phi_2$, $\phi_2$ is homeomorphic on $\phi_1(e)$.

For (iii): Suppose $e, e'$ are distinct boundary edges of $(\Sigma^1, V^1)$, and suppose $\phi_2 \circ \phi_1 (e)$, $\phi_2 \circ \phi_1 (e')$ intersect at an interior point of $\phi_2 \circ \phi_1(e)$. We note that $\phi_1$, $\phi_2$ can only fail to be injective along the boundary, by identifying vertices or edges. So if $\phi_1 (e)$ is not a boundary edge of $(\Sigma^2, V^2)$, then all its interior points are embedded by $\phi_2$, and hence $\phi_1(e')$ must intersect $\phi_1(e)$ in the interior; by property (iii) of $\phi_1$ then $\phi_1(e)$ and $\phi_1(e')$ coincide, so $\phi_2 \circ \phi_1 (e)$ and $\phi_2 \circ \phi_1(e')$ coincide. On the other hand, if $\phi_1(e)$ is a boundary edge of $(\Sigma^2, V^2)$, then by property (ii) of $\phi_2$, $\phi_2$ is homeomorphic on $\phi_1(e)$, and $\phi_2$ can only fail to be injective along $\phi_1(e)$ by identifying the edge $\phi_1(e)$ with another boundary edge of $(\Sigma^2, V^2)$. If $\phi_1(e')$ is not a boundary edge, then it can only intersect $\partial \Sigma^2$ at vertices of $V^2$, which are not mapped under $\phi_2$ to the interior of $\phi_2(\phi_1(e))$, so there is no intersection of $\phi_2(\phi_1(e'))$ with the interior of $\phi_2(\phi_1(e))$. Thus $\phi_1(e')$ is also a boundary edge of $(\Sigma^2, V^2)$, and by property (iii) of $\phi_2$, since $\phi_2 \circ \phi_1(e)$ and $\phi_2 \circ \phi_1(e')$ intersect other than at endpoints, they coincide.
\end{proof}

\begin{prop}
The set of occupied surfaces, and morphisms (including identity morphisms) between them, as defined above, form a category, the \emph{category of occupied surfaces}, denoted $\mathcal{OS}$.
\qed
\end{prop}

We note that if we impose the additional condition that morphisms not swallow vertices, then we obtain a subcategory of $\mathcal{OS}$. That is, the composition of two morphisms which do not swallow vertices is again a morphism which does not swallow vertices. Given $v \in V^1$, $\phi_1(v)$ lies in $V^2$ or has a neighbourhood containing points in and out of $\phi_1(\Sigma^1)$ (possibly both). In the first case $v$ is not swallowed by $\phi_2 \circ \phi_1$ by the non-swallowing property of $\phi_2$; in the second case a neighbourhood of $\phi_2 \circ \phi_1(v)$ contains points in and out of $\phi_2 \circ \phi_1 (\Sigma^1)$.

\subsection{Image and complement of a morphism}

The image $\phi(\Sigma)$ of a morphism $\phi: (\Sigma,V) \To (\Sigma',V')$, with the subspace topology from $\Sigma'$, need not be a surface. Several vertices of $V$ may be glued under $\phi$, and a neighbourhood of this point in $\phi(\Sigma)$ with the subspace topology need not be a disc. A similar phenomenon may occur with the complement of the image of $\Sigma$ in $\Sigma'$.

However, if we cut $\Sigma'$ along the boundary of $\phi(\Sigma)$, and then glue the boundary back in, we do obtain a surface. In the process, some vertices may be cut apart, and the number of vertices may increase. We can do the same with the complement of $\phi(\Sigma)$ in $\Sigma'$. In fact both the image and complement surfaces can be given the structure of occupied surfaces.

Precisely, take a Riemannian metric on $\Sigma'$, and restrict it to the interior of $\phi(\Sigma)$. Taking the metric completion of this open surface gives a surface $\Sigma^\phi$. The boundary of $\Sigma^\phi$ consists naturally of edges and vertices $V^\phi$: the edges are images of edges of $(\Sigma,V)$, and $V^\phi$ consists of images of vertices of $V$, possibly split apart. (Note that as edges may be glued together and vertices swallowed by $\phi$, the boundary edges and vertices of $\Sigma^\phi$ do not consist of \emph{all} images under $\phi$ of edges and vertices of $(\Sigma,V)$.) The vertices $V^\phi$ inherit well-defined signs from $\phi$, and each boundary edge of $\Sigma^\phi$ runs between vertices of $V^\phi$ of opposite sign, giving us the following.
\begin{lem}
The surface $\Sigma^\phi$, with signed vertices $V^\phi$, has the structure of an occupied surface.
\qed
\end{lem}

\begin{defn}
$(\Sigma^\phi, V^\phi)$ is the \emph{image occupied surface} of $\phi$.
\end{defn}

Similarly, we may consider the complement of $\phi(\Sigma)$ in $\Sigma'$. Although it need not be a surface (see figure \ref{fig:complementary_occupied}), taking its completion with respect to a Riemannian metric on $\Sigma'$, we obtain a surface $\Sigma^c$.

The boundary of $\Sigma^c$ breaks into edges and vertices $V^c$. The boundary edges of $\Sigma^c$ coincide with boundary edges of $(\Sigma',V')$ not in the image of $\phi$, and boundary edges of $(\Sigma^\phi, V^\phi)$ not in $\partial \Sigma'$. Each vertex in $V^c$ is in $\phi(V)$ or $V'$ (possibly both), although vertices may again split apart. Not all vertices of $\phi(V)$ or $V'$ need lie in $V^c$. The vertices again inherit well-defined signs, and the edges again run between vertices of opposite sign. We obtain the following.

\begin{lem}
The surface $\Sigma^c$, with signed vertices $V^c$, has the structure of an occupied surface.
\qed
\end{lem}

\begin{defn}
$(\Sigma^c, V^c)$ is the \emph{complementary occupied surface} of $\phi$.
\end{defn}

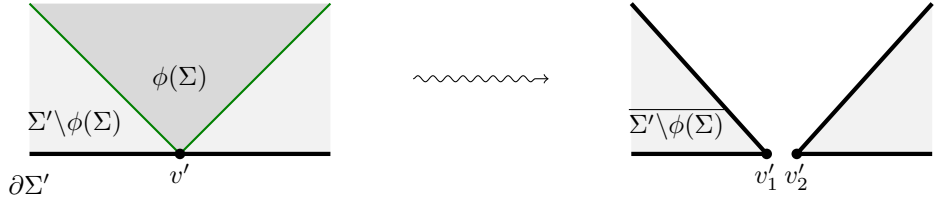
\begin{figure}
\begin{center}

\begin{tikzpicture}[
scale=2, fill = gray!10,
decomposition/.style={thick, draw=green!50!black}, 
boundary/.style={ultra thick} ]

\coordinate [label = below:{$v'$}] (v) at (1,0);

\fill (0,1) -- (v) -- (0,0) -- cycle;
\fill (2,1) -- (v) -- (2,0) -- cycle;
\fill [gray!30] (0,1) -- (v) -- (2,1) -- cycle;
\draw [boundary] (0,0) -- (2,0);
\draw [decomposition] (0,1) -- (v) -- (2,1);
\draw (1,0.5) node {$\phi(\Sigma)$};
\draw (0,-0.2) node {$\partial \Sigma'$};
\draw (0.3,0.2) node {$\Sigma' \backslash \phi(\Sigma)$};

\draw [shorten >=1mm, -to, decorate, decoration={snake,amplitude=.4mm, segment length = 2mm, pre=moveto, pre length = 1mm, post length = 2mm}]
(2.5,0.5) -- (3.5,0.5);

\coordinate [label = below:{$v'_1$}] (v1) at (4.9,0);
\coordinate [label = below:{$v'_2$}] (v2) at (5.1,0);

\fill (4,1) -- (v1) -- (4,0) -- cycle;
\fill (6,1) -- (v2) -- (6,0) -- cycle;
\draw [boundary] (4,1) -- (v1) -- (4,0);
\draw [boundary] (6,1) -- (v2) -- (6,0);
\draw (4.3,0.2) node {$\overline{\Sigma' \backslash \phi(\Sigma)}$};

\foreach \point in {v, v1, v2}
\fill [black] (\point) circle (1pt);

\end{tikzpicture}
\end{center}
\label{fig:complementary_occupied}
\caption{Vertices in the complementary occupied surface of a morphism.}
\end{figure}

It will be useful subsequently to classify the boundary edges and vertices of $(\Sigma^c, V^c)$ according to the above discussion.
\begin{defn}
Those boundary edges of $(\Sigma^c, V^c)$ which lie in $\partial \phi(\Sigma)$ we call \emph{$\Sigma$-type}. Those which lie in $\partial \Sigma'$ we call \emph{$\Sigma'$-type}.
\end{defn}

\begin{defn}
Vertices of $(\Sigma^c, V^c)$ which arise from vertices in:
\begin{enumerate}
\item
$\partial \phi(\Sigma) \backslash \partial \Sigma'$ are called \emph{$\Sigma$-type};
\item
$\partial \phi(\Sigma) \cap \partial \Sigma'$ are called \emph{$(\Sigma,\Sigma')$-type};
\item
$\partial \Sigma' \backslash \partial \phi(\Sigma)$ are called \emph{$\Sigma'$-type}.
\end{enumerate}
\end{defn}

We observe that in $(\Sigma^c, V^c)$, a vertex of $\Sigma$-type is adjacent to two boundary edges of $\Sigma$-type. Similarly, a vertex of $\Sigma'$-type is adjacent to two boundary edges of $\Sigma'$-type. A vertex of $(\Sigma,\Sigma')$-type may be adjacent to two boundary edges of $\Sigma$-type, or to one each of $\Sigma$- and $\Sigma'$-types.

(One might think that a vertex of $(\Sigma, \Sigma')$-type must be adjacent to boundary edges of both types, one of $\Sigma$-type and one of $\Sigma'$-type. But this is not true; for instance consider figure \ref{fig:punctured_torus_decomposition}, which depicts an occupied once-punctured torus with one boundary component and two vertices. Let $\phi$ be the inclusion of an occupied square to the upper square of that diagram. Then $(\Sigma^c, V^c)$ is an occupied square, with all vertices of type $(\Sigma, \Sigma')$; however $3$ of its boundary edges are of type $\Sigma$, so that there are vertices of type $(\Sigma, \Sigma')$ adjacent to $\Sigma$-type edges on both sides.)

\subsection{Isolating and vacuum-leaving morphisms}

We introduce two notions which describe properties of the complement of a morphism.

\begin{defn}
Let $\phi: (\Sigma,V) \To (\Sigma',V')$ be a morphism. A component of the complement of $\phi$ with no edges of $\Sigma'$-type (i.e. all edges of $\Sigma$-type) is called an \emph{isolated component} of $\phi$. A morphism whose complement has an isolated component is \emph{isolating}.
\end{defn}
An isolated component cannot ``escape'' out of $\Sigma'$ through a boundary edge. (Note that an isolated component can have a vertex on $\partial \Sigma'$; this is not good enough for an escape route.)

\begin{defn}
A morphism $\phi: (\Sigma, V) \To (\Sigma', V')$ \emph{leaves a vacuum} if its complement $(\Sigma^c, V^c)$ has a component which is the occupied vacuum.
\end{defn}
If $\phi$ leaves a vacuum $(\Sigma^\emptyset, V^\emptyset)$, that vacuum has two edges which may be of $\Sigma$- or $\Sigma'$-type. If both edges are $\Sigma'$-type, then in fact $(\Sigma^\emptyset, V^\emptyset)$ is a component of $(\Sigma', V')$. If both edges are $\Sigma$-type, then the $(\Sigma^\emptyset, V^\emptyset)$ is an isolated component of $\phi$; there is an isotopy of $\phi$ which pushes those edges together and squeezes the vacuum out of existence. If one edge of $(\Sigma^\emptyset, V^\emptyset)$ is $\Sigma$- and the other is $\Sigma'$-type, then $\phi$ can be isotoped so that the two boundary edges coincide and the vacuum disappears.

In a similar way, it often possible to perform an isotopy through morphisms in order to simplify the complement. Morphisms are ``simpler'' when more vertices and boundary edges of $(\Sigma,V)$ are mapped to vertices and boundary edges of $(\Sigma',V')$; we think of vertices and edges in the interior of $\Sigma'$ as ``slack'', and pushing them out to $\partial \Sigma'$ through an isotopy as ``tightening'' the morphism.

In fact, it can be proved that is is possible to perform an isotopy of morphisms, successively pushing vertices or edges of $\Sigma$-type to the boundary, until the only vacua left are components of $(\Sigma',V')$, and every non-swallowed vertex of $\Sigma$-type is adjacent to an isolated component of the complement $(\Sigma^c, V^c)$, or a component of the complement which is a disc with no vertices of $\Sigma'$-type.

\subsection{Simple morphisms}
\label{sec:simple_morphisms}
\label{sec:creation_and_annihilation}

We now consider some simple examples of morphisms; it will turn out that these morphisms are elementary examples from which all morphisms can be constructed.

\bigskip

\noindent \emph{Creations.} We can conjure a square $(\Sigma^\square, V^\square)$ out of the void, and place it beside our occupied surface. This is called creation; one is depicted in the third picture of figure \ref{fig:examples_of_morphisms}.	
\begin{defn}
\label{def:creation}
A morphism $\phi: (\Sigma,V) \To (\Sigma,V) \sqcup (\Sigma^\square, V^\square)$ which is the identity on $(\Sigma,V)$ is a \emph{creation}. The square $(\Sigma^\square, V^\square)$ is the \emph{created square}.
\end{defn}

We observe that, with notation as above, $N' = N + 2$ and $\chi(\Sigma') = \chi(\Sigma) + 1$, so $I(\Sigma',V') = I(\Sigma,V) + 1$. Creations increase index by $1$.

\bigskip

\noindent \emph{Annihilations.} An annihilation is a map $(\Sigma,V) \To (\Sigma',V')$ which does not actually ``destroy'' or ``forget'' part of $(\Sigma,V)$; the analogy to annihilation is to close off some vertices of $V$ by covering them over.

\begin{defn}
\label{def:annihilation}
A morphism $\phi: (\Sigma,V) \To (\Sigma',V')$ which is non-isolating and whose complement is an occupied square $(\Sigma^\square, V^\square)$, called the \emph{annihilator square}, such that $3$ consecutive boundary edges of the complement $(\Sigma^\square, V^\square)$ are glued to $3$ consecutive boundary edges of the image $(\Sigma^\phi, V^\phi)$.
\end{defn}
Note that as an annihilation is non-isolating, and has three edges of $\Sigma$-type by definition, the fourth edge of the annihilator square must be $\Sigma'$-type. See figure \ref{fig:annihilation}. The $2$ vertices and $3$ edges of $\Sigma$-type are called \emph{annihilated} vertices and edges.

Observe that $N' = N-1$ and $\chi(\Sigma') = \chi(\Sigma)$; indeed $\Sigma$ and $\Sigma'$ are homeomorphic; so $I(\Sigma',V') = I(\Sigma,V)-1$. Annihilations decrease index by $1$.

For both creations and annihilations, $(\Sigma',V')$ consists of the image $(\Sigma^\phi, V^\phi)$, together with the created/annihilator square; both are non-isolating and leave no vacua. 

We also note that after performing a creation, there is an annihilation which ``undoes it'' by immediately covering over the newly created vertices and annihilating them to a vacuum.

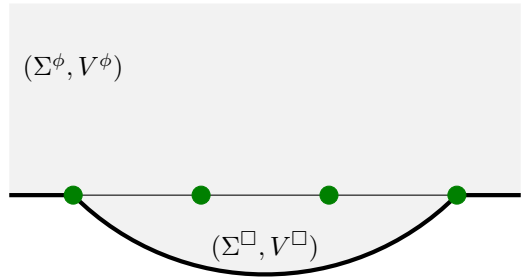
\begin{figure}
\begin{center}

\begin{tikzpicture}[
scale=1.7, 
fill = gray!10,
boundary/.style={ultra thick}, 
decomposition/.style={thick, draw=green!50!black}, 
vertex/.style={fill = green!50!black, draw=green!50!black}
]

{
\fill (-0.5,0) -- (0,0) to [bend right=45] (3,0) -- (3.5,0) -- (3.5,1.5) -- (-0.5,1.5) -- cycle;

\coordinate (p1) at (0,0);
\coordinate (q1) at (1,0);
\coordinate (q2) at (2,0);
\coordinate (p4) at (3,0);

\draw [boundary] (-0.5,0) -- (p1) to [bend right=45] (p4) -- (3.5,0);
\draw (p1) -- (p4);
\draw (0,1) node {$(\Sigma^\phi, V^\phi)$};
\draw (1.5,-0.4) node {$(\Sigma^\square, V^\square)$};

\foreach \point in {p1, q1, q2, p4}
\fill [vertex] (\point) circle (2pt);
}

\end{tikzpicture}

\caption{An annihilation.}
\label{fig:annihilation}
\end{center}
\end{figure}

\bigskip

\noindent \emph{Gluings.} If we glue two non-consecutive edges $e_1, e_2$ of an occupied surface $(\Sigma,V)$, respecting their orientations, then as discussed previously the result is an occupied surface $(\Sigma',V')$. The identification map $(\Sigma,V) \To (\Sigma',V')$ is obviously a morphism.
\begin{defn}
A morphism $\phi: (\Sigma,V) \To (\Sigma',V')$ obtained by gluing two non-consecutive boundary edges of $(\Sigma,V)$ is a \emph{standard gluing} or just \emph{gluing} morphism.
\end{defn}
(Since ``fold'' and ``zip'' morphisms, which we define next, can also be regarded as gluing maps, we use ``standard gluing'' to emphasise that we mean this particular type of gluing.) Note that $N'=N-1$ and $\chi(\Sigma') = \chi(\Sigma) - 1$ so $I(\Sigma',V') = I(\Sigma,V)$. Standard gluings preserve index.

\bigskip

\noindent \emph{Folds.} We also consider gluing edges $e_1, e_2$, again respecting orientations; but now we consider the case where $e_1, e_2$ are consecutive around some boundary component $C$ of $\Sigma$. If $e_1, e_2$ are not the only edges on $C$, then they intersect at a single vertex $v$, and the gluing folds the edges $e_1, e_2$ around $v$, swallowing $v$ in the process. The result $(\Sigma',V')$ is an occupied surface.
\begin{defn}
A morphism $\phi: (\Sigma, V) \To (\Sigma', V')$ obtained by identifying two consecutive edges $e_1, e_2$ on a boundary component of $(\Sigma,V)$ with more than $2$ edges is a \emph{fold}.
\end{defn}
One pair of vertices of $V$ is identified, and $v$ is swallowed, so $N' = N-1$; and $\Sigma'$ is homeomorphic to $\Sigma$, so $\chi(\Sigma') = \chi(\Sigma)$; so $I(\Sigma', V') = I(\Sigma,V) - 1$. Folds decrease index by $1$.

A fold is \emph{positive} or \emph{negative} accordingly as the swallowed vertex $v$ has positive or negative sign. The final picture of figure \ref{fig:examples_of_morphisms} shows a negative fold.

\bigskip

\noindent \emph{Zips.} Finally we consider gluing edges $e_1, e_2$ which form an entire boundary component of $(\Sigma,V)$. Again we glue while respecting orientations, so the effect is to ``zip up'' the boundary component. The resulting surface will be an occupied surface if and only if there are at least two boundary components on $(\Sigma,V)$; then $\Sigma'$ will have the same genus as $\Sigma$ but one fewer boundary component. The endpoints of $e_1, e_2$ are both swallowed.
\begin{defn}
Let $(\Sigma,V)$ be an occupied surface with at least two boundary components, including a boundary component consisting of only two edges $e_1, e_2$. A \emph{zip} is a morphism $\phi: (\Sigma, V) \To (\Sigma', V')$ obtained by gluing $e_1$ and $e_2$ together.
\end{defn}
We have $N' = N-1$ and $\chi(\Sigma') = \chi(\Sigma) + 1$, so $I(\Sigma',V') = I(\Sigma,V) - 2$. Zips decrease index by $2$, and are the most violent of the examples we consider here.

\begin{prop}
\label{prop:surjective_morphism_gluings}
Any surjective morphism is a composition of gluings, folds and zips. 
\end{prop}
(This includes a homeomorphism, which we regard as an identity map, and the null composition.)

\begin{proof}
Let $\phi: (\Sigma,V) \To (\Sigma',V')$ be surjective; as an embedding of the interior of $\Sigma$ into $\Sigma'$, $\phi$ can only fail to be injective by identifying vertices or identifying edges in pairs, always respecting signs. In fact, if $\phi$ identifies two vertices then, being surjective onto a surface $\Sigma'$, it must identify a pair of edges adjacent to those vertices. Thus in identifying the edges identified by $\phi$, we also identify all the vertices identified by $\phi$.

Identifying non-adjacent edges is done by a standard gluing, and adjacent edges by a fold or zip. Thus we may successively glue together edges until we have identified all pairs of edges identified by $\phi$; and the composition of these glues, folds and zips can be taken to be $\phi$.
\end{proof}

We will prove in proposition \ref{prop:morphism_composition_of_atomics} that any morphism (not necessarily surjective) is a composition of creations, gluings, folds and zips. (Annihilations are strictly therefore not necessary, but they are useful to consider in any case.)

\section{Quadrangulations}
\label{sec:quadrangulations}

\subsection{Definition}

\begin{defn}
A \emph{quadrangulation} $A$ of an occupied surface $(\Sigma,V)$ is a set of decomposing arcs on $\Sigma$, cutting along which decomposes $(\Sigma,V)$ into a set of disjoint occupied squares, the \emph{squares of the quadrangulation}. A \emph{quadrangulated surface} is a pair $(\Sigma,A)$ where $A$ is a quadrangulation of an occupied surface $(\Sigma,V)$.
\end{defn}

Our notion of a quadrangulation is very similar to the \emph{bipartite quadrangulations} of Nakamoto in \cite{Nakamoto96} and Negami--Nakamoto in \cite{Negami_Nakamoto93}; the main difference is that we prefer our vertices to lie entirely on the boundary of the surface.

Note that every decomposing arc in a quadrangulation is nontrivial. We can alternately think of a quadrangulation as being given by decomposing arcs, or by squares.

We also call the decomposing arcs of a quadrangulation its \emph{internal edges}. So each edge of a square of a quadrangulation of $(\Sigma,V)$ is an internal edge or a boundary edge.

Taking a connected occupied surface other than a vacuum or square, by lemma \ref{lem:nontrivial_arc_exists} it has a nontrivial decomposing arc. Cutting along this arc holds the index $I$ constant and decreases the gluing number $G$ by $1$ (lemma \ref{lem:index_invariant}). And such a cut cannot introduce a vacuum, so each component has non-negative gluing number. Thus we may successively cut along nontrivial decomposing arcs, decreasing $G$ by $1$ at each stage, until we arrive at $G=0$ (but no vacua), hence a disjoint union of squares.
\begin{prop}
\label{prop:build_quadrangulation}
Any occupied surface without vacua $(\Sigma,V)$ has a quadrangulation. Any quadrangulation of $(\Sigma,V)$ has $G(\Sigma,V)$ internal edges and $I(\Sigma,V)$ squares.
\qed
\end{prop}
Indeed, no matter how we find decomposing arcs, we eventually arrive at a quadrangulation. Any collection of non-parallel decomposing arcs extends to a quadrangulation. Build decomposing arcs, and a quadrangulation will come.

In a quadrangulation, edges of a square are never glued together.
\begin{lem}
\label{lem:arcs_distinct}
Let $(\Sigma^\square, V^\square)$ be a square of the quadrangulated surface $(\Sigma, A)$. The four boundary edges of $(\Sigma^\square, V^\square)$ are distinct edges of the quadrangulation.
\end{lem}

\begin{proof}
If two adjacent edges of $(\Sigma^\square, V^\square)$ are glued together in $(\Sigma,V)$, then their common vertex is swallowed and does not lie on the boundary of $\Sigma$. Two opposite edges of $(\Sigma^\square, V^\square)$ are either both incoming or both outgoing; gluing them respecting signs cannot give an orientable result.
\end{proof}

A quadrangulation may contain several decomposing arcs running between the same endpoints; but such arcs cannot be parallel, as they would then cut out a vacuum. For instance figure \ref{fig:punctured_torus_decomposition} shows a quadrangulation of the once-punctured torus with $2$ vertices on the boundary; all edges run between the same vertices.

\begin{figure}
\begin{center}

\begin{tabular}{c}
\begin{tikzpicture}[
scale=2, 
boundary/.style={ultra thick}, 
decomposition/.style={thick, draw=green!50!black}, 
vertex/.style={draw=green!50!black, fill=green!50!black},
>=triangle 90, 
decomposition glued1/.style={thick, draw=green!50!black, postaction={nomorepostaction,decorate, decoration={markings,mark=at position 0.5 with {\arrow{>}}}}},
decomposition glued2/.style={thick, draw = green!50!black, postaction={nomorepostaction, decorate, decoration={markings,mark=at position 0.5 with {\arrow{>>}}}}}
]

\coordinate [label = right:{$-$}] (0) at (0:1);
\coordinate [label = above right:{$+$}] (1) at (60:1);
\coordinate [label = above left:{$-$}] (2) at (120:1);
\coordinate [label = left:{$+$}] (3) at (180:1);
\coordinate [label = below left:{$-$}] (4) at (240:1);
\coordinate [label = below right:{$+$}] (5) at (300:1);

\fill [gray!10] (0) -- (1) -- (2) -- (3) -- (4) -- (5) -- cycle;

\draw [decomposition glued2] (0) -- (1);
\draw [decomposition glued1] (2) -- (1);
\draw [boundary] (2) -- (3);
\draw [decomposition glued2] (4) -- (3);
\draw [decomposition glued1] (4) -- (5);
\draw [boundary] (5) -- (0);
\draw [decomposition] (0) to (3);

\foreach \point in {0, 1, 2, 3, 4, 5}
\fill [vertex] (\point) circle (1pt);

\end{tikzpicture}
\end{tabular}
\begin{tabular}{c}
\includegraphics[scale=0.5]{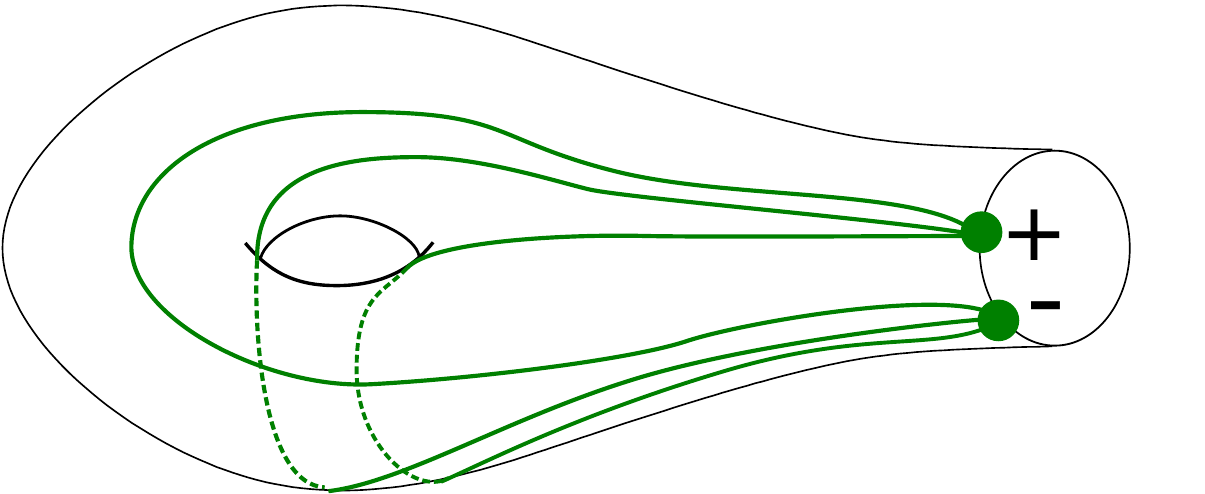}
\end{tabular}
\caption{Quadrangulation of a punctured torus with $2$ vertices; the two pictures are equivalent.}
\label{fig:punctured_torus_decomposition}
\end{center}
\end{figure}
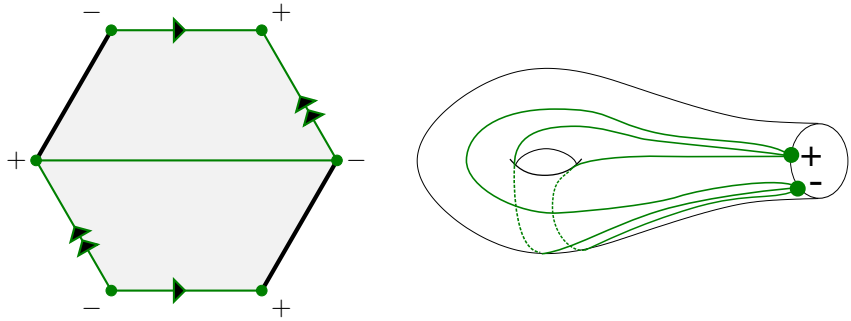

\subsection{Building occupied surfaces with quadrangulations}

A quadrangulation allows us to build an occupied surface by gluing squares together. We can construct a quadrangulated surface $(\Sigma,A)$ from nothing by creating the squares of the quadrangulation, then gluing them together. In fact these gluing maps can all be taken to be either standard gluings or folds (never zips), as we now prove.
\begin{lem}
\label{lem:quadrangulation_ordering}
Let $A$ be a quadrangulation of an occupied surface $(\Sigma,V)$ with squares $(\Sigma^\square_1, V^\square_1), \ldots, (\Sigma^\square_n, V^\square_n)$. There exists an ordering of the squares such that successively gluing on the squares in order to form $(\Sigma,V)$, no square is ever glued in along all four of its edges.
\end{lem}

\begin{proof}
Obviously some square of $A$ has an edge on $\partial \Sigma$. Remove this square to reduce to a smaller surface; and successively remove squares to reduce to the empty occupied surface, at each stage removing a square with an edge on the boundary of the surface. Order the squares so that we removed, in order, $(\Sigma^\square_n, V^\square_n), \ldots, (\Sigma^\square_1, V^\square_1)$.

Now we perform the reverse procedure. Start from $(\Sigma^\square_1, V^\square_1)$, and glue on $(\Sigma^\square_2, V^\square_2), \ldots, (\Sigma^\square_n, V^\square_n)$. At each stage, our ordering guarantees that no square is glued in along all its edges; when $(\Sigma^\square_i, V^\square_i)$ is glued on, at least one edge is a boundary edge.
\end{proof}

\begin{prop}
\label{prop:quadrangulated_surface_construction}
Let $(\Sigma,V)$ be an occupied surface. Then there exists a sequence of morphisms
\[
\emptyset = (\Sigma_0, V_0) \stackrel{\phi_1}{\To} (\Sigma_1, V_1) \stackrel{\phi_2}{\To} \cdots \stackrel{\phi_n}{\To} (\Sigma_n, V_n) = (\Sigma,V)
\]
where each $\phi_i$ is a creation, standard gluing, or fold.
\end{prop}

\begin{proof}
Take any quadrangulation of $(\Sigma,V)$, and then order the squares $(\Sigma^\square_1, V^\square_1), \ldots, (\Sigma^\square_n, V^\square_n)$ as in the preceding lemma. Proceeding in order from $i=1$ to $n$, we conjure up $(\Sigma^\square_i, V^\square_i)$ using a creation, then attach its edges to any edges of the previously created squares as in the quadrangulation of $(\Sigma,V)$. We only need show that this edge attachment can be done with standard gluings and folds.

Let $(\Sigma_p, V_p)$ be the union of the previously created and glued squares. By the lemma, we need to glue $(\Sigma^\square_i, V^\square_i)$ to $(\Sigma_p, V_p)$ along at most $3$ edges. The problem is to ensure we never glue the edges in such a way as to zip up a boundary component; this can happen when consecutive edges of $(\Sigma^\square_i, V^\square_i)$ are glued to consecutive edges of $(\Sigma_p, V_p)$. So call a boundary edge $e$ of $(\Sigma^\square_i, V^\square_i)$ \emph{troublesome} if $e$, and a consecutive edge along the boundary of $(\Sigma^\square_i, V^\square_i)$, are glued to two consecutive edges of $(\Sigma_p, V_p)$. The troublesome edges of $(\Sigma^\square_i, V^\square_i)$ are then either:
\begin{enumerate}
\item
three consecutive edges, glued to three consecutive boundary edges of $(\Sigma_p, V_p)$, with the fourth edge not glued;
\item
two consecutive edges, glued to two consecutive edges of $(\Sigma_p, V_p)$, with other edges not troublesome;
\item
no troublesome edges.
\end{enumerate}
If there is trouble, use a standard gluing to attach one of the troublesome edges. Then the consecutive troubling edges can be glued with folds. Non-troublesome edges can then be glued by standard gluings. This gives a construction with the desired properties.
\end{proof}

\subsection{Slack quadrangulations}

If we allow vertices in the interior, we obtain a different type of quadrangulation, essentially equivalent to the bipartite quadrangulations of \cite{Nakamoto96, Negami_Nakamoto93} (which consider only closed surfaces). We call such quadrangulations ``slack''.

\begin{defn}
A \emph{slack quadrangulation} $Q$ of an occupied surface $(\Sigma,V)$ is an embedded $1$-complex in $\Sigma$, such that:
\begin{enumerate}
\item
Each vertex of $Q$ is signed, and each edge of $Q$ connects vertices of opposite sign.
\item
$Q \cap \partial \Sigma = V$, i.e. the vertices of $Q$ on $\partial \Sigma$ are precisely the vertices $V$, and agree with their signs.
\item
$Q$ cuts $\Sigma$ into squares, with each edge a $1$-cell of $Q$ or a boundary edge of $(\Sigma,V)$.
\end{enumerate}
\end{defn}
Vertices of $Q$ in the interior of $\Sigma$ are \emph{internal vertices}. If there are no internal vertices then we obtain a bona fide quadrangulation.

We will develop a procedure to ``tighten up'' a slack quadrangulation into a bona fide quadrangulation, by performing what we call \emph{slack square collapse}; in \cite{Nakamoto96} and \cite{Negami_Nakamoto93} it is called \emph{face contraction}. This involves removing an internal vertex by collapsing one of the squares incident at that vertex.

More precisely, given a slack quadrangulation $Q$ of $(\Sigma,V)$, let $v$ be an internal vertex and let $(\Sigma^\square, V^\square)$ be a square of $Q$ incident at $v$. Let $w$ be the vertex of $(\Sigma^\square, V^\square)$ opposite to $v$ (hence the same sign). We first show that it is possible to arrange the situation so that $w \in V$.

\begin{lem}
Let $Q$ be a slack quadrangulation on the connected occupied surface $(\Sigma,V)$ which contains a positive internal vertex. Then there exists a square of the quadrangulation with an internal positive vertex at one corner, and a vertex of $V_+$ at the opposite corner.
\end{lem}
(Obviously there is a corresponding result with negative signs.)

\begin{proof}
Consider in each square of $Q$ drawing in the diagonal connecting the positive signed vertices. These diagonals form an embedded $1$-complex $D$ in $(\Sigma,V)$ whose vertices are the positive vertices of $Q$. We claim $D$ is connected. To see this, consider starting from a single square and successively gluing on squares to form the slack quadrangulation $Q$, constructing $(\Sigma,V)$; at each stage we draw the diagonals between positive vertices to construct $D$, and at each stage this $1$-complex remains connected.

Now $D$ is a connected 1-complex whose vertices are all the positive vertices of $Q$; each vertex is either internal or in $V_+$. Thus there is an edge of $D$ connecting an internal vertex to a vertex of $V_+$; and there is then a square of $Q$ with these vertices in opposite corners.
\end{proof}

So take $v$, an internal vertex, and $w$ the vertex opposite to $v$ in $(\Sigma^\square, V^\square)$, as is now guaranteed. First suppose that the two edges of $(\Sigma^\square, V^\square)$ adjacent to $v$ are not identified in $Q$, and also the two edges adjacent to $w$ are not identified in $Q$. Then in the slack square collapse, the two distinct edges of $(\Sigma^\square, V^\square)$ incident to $v$ are isotoped onto the two distinct edges incident to $w$, so that the four edges of $(\Sigma^\square, V^\square)$ are collapsed to two. All the other squares of the quadrangulation remain intact, and the two vertices $v,w$ are identified. See figure \ref{fig:slack_square_collapse}. The result is another (possibly slack) quadrangulation of $(\Sigma,V)$ with fewer internal vertices.

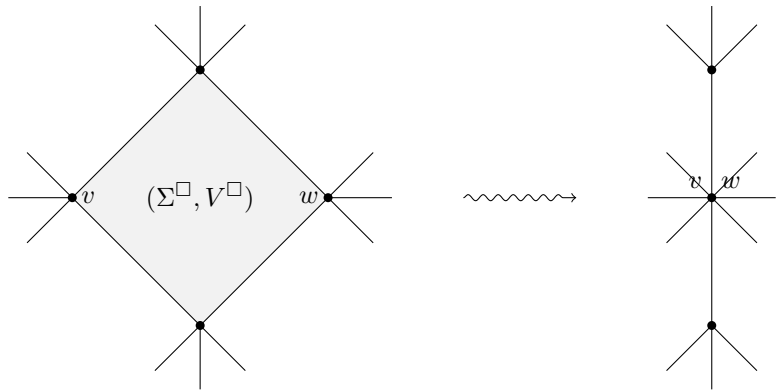
\begin{figure}
\begin{center}
\begin{tikzpicture}[
scale=1.7, 
fill = gray!10,
]

\fill (1,0) -- (2,1) -- (3,0) -- (2,-1) -- cycle;

\coordinate [label = right:{$v$}] (v) at (1,0);
\coordinate [label = left:{$w$}] (w) at (3,0);
\coordinate (1top) at (2,1);
\coordinate (1bot) at (2,-1);

\draw (v) -- (1bot) -- (w) -- (1top) -- cycle;
\foreach \angle in {0, 45, 90} {
\draw (v) -- ($ (1,0) + (135 + \angle : 0.5) $) ;
\draw (1bot) -- ($ (2,-1) + (225 + \angle : 0.5)  $) ;
\draw (w) -- ($ (3,0) + (-45 + \angle: 0.5 ) $) ;
\draw (1top) -- ($ (2,1) + (45 + \angle: 0.5 ) $) ; }
\draw (2,0) node {$(\Sigma^\square, V^\square)$};

\draw [shorten >=1mm, -to, decorate, decoration={snake,amplitude=.4mm, segment length = 2mm, pre=moveto, pre length = 1mm, post length = 2mm}] (4,0) -- (5,0);

\coordinate [label = above left:{$v$}, label = above right:{$w$}] (v2) at (6,0);
\coordinate (2top) at (6,1);
\coordinate (2bot) at (6,-1);

\draw (2top) -- (2bot);
\foreach \angle in {0, 45, 90} {
\draw (v2) -- ($ (v2) + (135 + \angle : 0.5) $) ;
\draw (2bot) -- ($ (2bot) + (225 + \angle : 0.5)  $) ;
\draw (v2) -- ($ (v2) + (-45 + \angle: 0.5 ) $) ;
\draw (2top) -- ($ (2top) + (45 + \angle: 0.5 ) $) ; }

\foreach \point in {v, w, 1top, 1bot, v2, 2top, 2bot}
\fill [black] (\point) circle (1pt);

\end{tikzpicture}

\end{center}
\caption{Slack square collapse.}
\label{fig:slack_square_collapse}
\end{figure}

It is also possible that the two edges of $(\Sigma^\square, V^\square)$ adjacent to $v$ are identified. Then we again isotope $v$ to $w$, and collapse the edges of $(\Sigma^\square, V^\square)$ down to a single edge, as shown in figure \ref{fig:degenerate_slack_square_collapse}. Effectively we have isotoped away a vacuum. Similarly if the two edges adjacent to $w$ are identified.

\begin{figure}
\begin{center}
\begin{tikzpicture}[
scale=1.7, 
fill = gray!10,
]

\fill (1,0) to [bend left=60] (3,0) to [bend left=60] (1,0);

\coordinate (u) at (1,0);
\coordinate [label = right:{$v$}] (v) at (2,0);
\coordinate [label = left:{$w$}] (w) at (3,0);

\draw (u) to [bend left=60] (w);
\draw (u) to [bend right=60] (w);
\draw (u) -- (v);

\foreach \angle in {0, 45, 90} {
\draw (u) -- ($ (u) + (135 + \angle : 0.5) $) ;
\draw (w) -- ($ (w) + (-45 + \angle: 0.5 ) $) ;
}
\draw (2,0.3) node {$(\Sigma^\square, V^\square)$};

\draw [shorten >=1mm, -to, decorate, decoration={snake,amplitude=.4mm, segment length = 2mm, pre=moveto, pre length = 1mm, post length = 2mm}] (4,0) -- (5,0);

\coordinate [label = above left:{$v$}, label = below left:{$w$}] (v2) at (8,0);
\coordinate (u2) at (6,0);

\draw (u2) -- (v2);
\foreach \angle in {0, 45, 90} {
\draw (u2) -- ($ (u2) + (135 + \angle : 0.5) $) ;
\draw (v2) -- ($ (v2) + (-45 + \angle: 0.5 ) $) ;
}

\foreach \point in {u, v, w, u2, v2}
\fill [black] (\point) circle (1pt);

\end{tikzpicture}

\end{center}
\caption{Degenerate slack square collapse.}
\label{fig:degenerate_slack_square_collapse}
\end{figure}
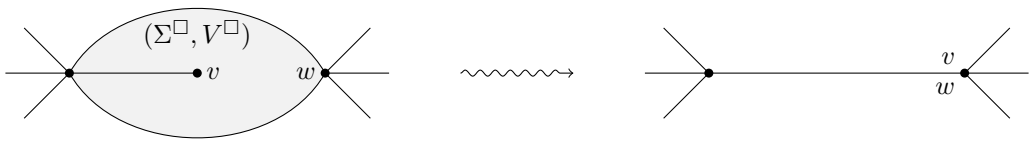

In any case, after collapsing enough slack squares we arrive at a bona fide quadrangulation of $(\Sigma,V)$.
\begin{prop}
\label{prop:slack_square_collapse_exists}
In any slack quadrangulation $Q$ of $(\Sigma,V)$, slack square collapses can successively be performed taking $Q$ to a bona fide quadrangulation of $(\Sigma,V)$.
\qed
\end{prop}

\subsection{Building morphisms with quadrangulations}
\label{sec:building_morphisms}

We return to the question, raised in section \ref{sec:simple_morphisms}, of building general types of morphisms out of our elementary types of creation, annihilation, gluing, fold, and zip. The construction of proposition \ref{prop:quadrangulated_surface_construction} shows how to construct a morphism from the empty occupied surface to any occupied surface; we now extend this to general morphisms. In fact the following proposition is really just a relative version of proposition \ref{prop:quadrangulated_surface_construction}; and so we begin with a relative version of lemma \ref{lem:quadrangulation_ordering}.

\begin{lem}
\label{lem:building_non-isolating_morphisms}
Let $\phi: (\Sigma, V) \To (\Sigma', V')$ be a non-isolating morphism which leaves no vacua, with a quadrangulated complement $(\Sigma^c, A^c)$. Then there exists an ordering of the squares of $A^c$, $(\Sigma^\square_1, V^\square_1)$, $\ldots$, $(\Sigma^\square_n, V^\square_n)$, such that successively gluing the squares on to $(\Sigma,V)$ in order and constructing $(\Sigma',V')$, no square is ever glued along all four of its edges.
\end{lem}

\begin{proof}
As $\phi$ leaves no vacua, $(\Sigma^c, V^c)$ has no vacuum components. As $\phi$ is non-isolating, there exists a square with a $\Sigma'$-type edge. Remove this square; we obtain a morphism with a lower-index complement that is still non-isolating and leaves no vacua, and successively remove squares with $\Sigma'$-type edges. Now as in lemma \ref{lem:quadrangulation_ordering}, reverse this process, gluing the squares on to $(\Sigma,V)$ in the reverse order; no square is ever glued along all four edges.
\end{proof}

\begin{prop}
\label{prop:morphism_composition_of_atomics}
Let $\phi: (\Sigma, V) \To (\Sigma', V')$ be a morphism. Then $\phi$ is a composition of creations, folds, gluings, and zips.
\end{prop}
(Note annihilations do not appear here; we will see later in section \ref{sec:defns_decorated_morphisms} that annihilations and folds are, in a certain sense, interchangeable.)

\begin{proof}
First, if any boundary edges of $(\Sigma,V)$ are identified under $\phi$, then as in proposition \ref{prop:surjective_morphism_gluings} we may express this as a composition of gluings, folds and zips. Having done that, we may assume $\phi$ does not identify boundary edges of $(\Sigma,V)$. (It may still identify vertices.)

We consider the components of the complement of $\phi$, and construct them one by one. At each stage, we have a previously existing surface $(\Sigma_p, V_p)$ constructed from $(\Sigma,V)$ by creating squares and gluing or folding or zipping them together; to this we adjoin a component $(\Sigma^T, V^T)$ of the complement by more creations, gluings, folds and zips.

First suppose $(\Sigma^T, V^T)$ is a vacuum. If may have edges both of $\Sigma'$-type, one each $\Sigma$- and $\Sigma'$-type, or both $\Sigma$-type. In the first case our vacuum is a component of $(\Sigma', V')$; we create a square and fold two of its edges together. In the second case we create a square, attach the $\Sigma$-type edge with a standard gluing, and fold two edges together. In the third case we create the square, attach one edge with a standard gluing, fold two edges together, and zip the final edge. We may now assume $\phi$ leaves no vacua.

Now suppose $(\Sigma^T, V^T)$ is a non-isolating component. Using lemma \ref{lem:building_non-isolating_morphisms} on the morphism which adds $(\Sigma^T, V^T)$ to the previously constructed surface $(\Sigma_p, V_p)$, this morphism can be constructed by successively creating squares and gluing their edges. As in the proof of proposition \ref{prop:quadrangulated_surface_construction}, each new square either has three consecutive troublesome edges, or two consecutive troublesome edges, or no troublesome edges, and in every case we can attach edges of the new square with standard gluings and folds.

Finally, suppose $(\Sigma^T, V^T)$ is an isolating component. Take a common edge $e$ of $(\Sigma^T, V^T)$ and $(\Sigma_p, V_p)$, i.e. an edge along which $(\Sigma^T, V^T)$ is to be glued. Consider splitting apart $(\Sigma_p, V_p)$ and $(\Sigma^T, V^T)$ at $e$; the morphism which attaches $(\Sigma^T, V^T)$ to $(\Sigma_p, V_p)$ along all desired edges except $e$ is a non-isolating morphism, and so by the above we construct it with creations, gluings and folds. Having done this, we glue the two split-open edges back to $e$ with a zip.
\end{proof}

Note that the sequence of creations, gluings, folds and zips is by no means canonical; there may be many choices of quadrangulation of $(\Sigma^c, V^c)$, and many ways to order its squares to make a good gluing order. 

However we can count the number of each type of elementary morphism from the topology of $\phi$. For instance, if $\phi$ leaves no vacua, then the number of creations is $I(\Sigma^c, V^c)$; swallowed vertices are produced precisely one for each fold and two for each zip; the number of zips is the number of isolated components plus number of collapsed boundary components of $(\Sigma,V)$; and the number of gluings + folds + zips is the number of edges glued along, which is the number of  $\Sigma$-type edges (common edges of $(\Sigma,V)$ and $(\Sigma^c, V^c)$), plus $G(\Sigma^c,V^c)$ (internal edges of $(\Sigma^c,V^c)$), plus the number of pairs of edges of $(\Sigma,V)$ identified under $\phi$.

\subsection{Quadrangulations and ribbon graphs}
\label{sec:ribbon_graphs}

A quadrangulation $A$ of an occupied surface $(\Sigma,V)$ without vacua describes $(\Sigma,V)$ as a 2-complex, with $V$ the $0$-skeleton, and all $2$-cells being squares. The $1$-skeleton is a graph consisting of internal and boundary edges. This graph has the extra structure that around each vertex $v \in V$ the edges are totally ordered, anticlockwise, from one boundary edge incident to $v$, through internal edges, to the other boundary edge incident to $v$. This is similar to the structure of a ribbon graph, where edges incident to a given vertex are cyclically ordered. Similarly, a slack quadrangulation also describes $(\Sigma,V)$ as a $2$-complex with all cells being squares. At the boundary vertices $V$, incident edges are totally ordered. At internal vertices, incident edges are cyclically ordered. This ribbon graph structure is enough to reconstruct the (bona fide or slack) quadrangulated $(\Sigma,V)$. So the data of a (bona fide or slack) quadrangulated occupied surface is equivalent to the data of a ribbon-type graph, with cyclic and/or total orderings of edges around vertices.

In the case of a bona fide quadrangulation, this ribbon-type graph, or thickened 1-skeleton, has $G(\Sigma,V)$ internal edges, $2N$ boundary edges, and $I(\Sigma,V) + B$ boundary components.

There is also a \emph{dual graph} to a quadrangulation. This graph has a vertex for each square of the quadrangulation, and an edge connecting any two vertices corresponding to squares whose sides are glued. So the dual graph has $I(\Sigma,V)$ vertices and $G(\Sigma,V)$ edges; it also has the structure of a ribbon graph, with a cyclic ordering of edges around each vertex. Each vertex has degree at most $4$. In fact $\Sigma$ deformation retracts onto the dual graph, and the ribbon graph, considered as a thickened graph, is homeomorphic to $\Sigma$.

\subsection{Adjusting quadrangulations}
\label{sec:adjusting_quadrangulations}

A given occupied surface $(\Sigma,V)$ can have many quadrangulations, possibly infinitely many. However they are all related by a particular elementary move, which we shall call a \emph{diagonal slide} (following graph theory literature e.g. \cite{Nakamoto96, Negami_Nakamoto93, Nakamoto_Suzuki10}). 

Consider an occupied disc with $6$ vertices, i.e. a hexagon. A quadrangulation of the hexagon is given by a single decomposing arc running between opposite vertices of the hexagon. Passing from the quadrangulation obtained from one of these diagonals to another is called a \emph{diagonal slide}; see figure \ref{fig:diagonal_slide}. In a regular hexagon, the diagonal is rotated $60^\circ$ clockwise or anticlockwise; in general we may speak of a clockwise or anticlockwise diagonal slide. 

We can perform a diagonal slide in any two adjacent squares in a quadrangulated occupied surface; it is a local adjustment of a quadrangulation. It has order $3$, in the sense that taking a quadrangulation and repeating a diagonal slide in the same hexagon in the same direction $3$ times yields the original quadrangulation.

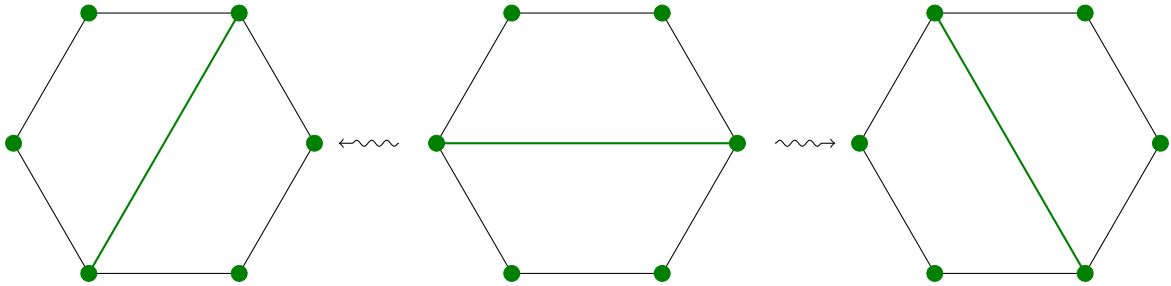
\begin{figure}
\begin{center}

\begin{tikzpicture}[
scale=2, 
boundary/.style={ultra thick}, 
vertex/.style={draw=green!50!black, fill=green!50!black},
decomposition/.style={thick, draw=green!50!black}, 
]

\foreach \x/\rot in {-80/60, 0/0, 80/-60}
{
\draw [xshift=\x, rotate=\rot] (0:1) -- (60:1) -- (120:1) -- (180:1) -- (240:1) -- (300:1) -- cycle;
\draw [xshift=\x, rotate=\rot, decomposition] (0:1) -- (180:1);

\foreach \angle in {0, 60, 120, 180, 240, 300}
\fill [vertex, xshift=\x, rotate=\rot] (\angle:1) circle (1.5pt);
}

\draw [shorten >=1mm, -to, decorate, decoration={snake,amplitude=.4mm, segment length = 2mm, pre=moveto, pre length = 1mm, post length = 2mm}]
(1.2,0) -- (1.7,0);
\draw [shorten >=1mm, -to, decorate, decoration={snake,amplitude=.4mm, segment length = 2mm, pre=moveto, pre length = 1mm, post length = 2mm}] (-1.2,0) -- (-1.7,0);

\end{tikzpicture}

\caption{Diagonal slides.}
\label{fig:diagonal_slide}
\end{center}
\end{figure}

\begin{thm}
\label{thm:quadrangulations_related_by_slides}
Let $(\Sigma,V)$ be an occupied surface without vacua. Then any two quadrangulations $A,A'$ of $(\Sigma,V)$ are related by diagonal slides.
\end{thm}

Our proof of this theorem uses a technique similar to Penner's proof of the related fact that any two ideal triangulations of a punctured surface are related by a finite sequence of elementary moves which switch diagonals in a square.

When $\Sigma$ is a disc, we can give an elementary proof.
\begin{proof}[Sketch of proof, when $\Sigma$ is a disc]
Given a vertex $v \in V$ and two edges $e,e'$ of a quadrangulation $A$ incident to $v$, we may use diagonal slides to adjust until $e,e'$ are consecutive at $v$. Using this fact, given any four distinct vertices $v_1, v_2, v_3, v_4$ successively connected by edges $e_1, e_2, e_3$ (internal or boundary) of a quadrangulation $A$, we can make edges $e_1, e_2$ consecutive at $v_2$ and $e_2, e_3$ consecutive at $v_3$; then we can obtain an edge connecting $v_1$ to $v_4$, isotopic to the path $e_1 \cup e_2 \cup e_3$; see figure \ref{fig:length_3_path}. So if we have distinct vertices and edges around the boundary $v_1 \stackrel{e_1}{\leftrightarrow} v_2 \stackrel{e_2}{\leftrightarrow} v_3 \stackrel{e_3}{\leftrightarrow}v_4$, we can perform diagonal slides to obtain a quadrangulation in which there is a square with vertices $v_1, v_2, v_3, v_4$, three edges $e_1, e_2, e_3$, and a final edge connecting $v_1$ to $v_4$.

Now in any quadrangulated disc $(\Sigma,V)$, there is always a square with three edges on $\partial \Sigma$. So given two quadrangulations $A,A'$, take a square $(\Sigma^\square, V^\square)$ of $A$ which has three boundary edges, and by the above we may adjust $A'$ until it also contains $(\Sigma^\square, V^\square)$. In this way we reduce to a smaller case and are done by induction, the simplest cases being easy.
\end{proof}

The above proof applies equally to more complicated surfaces, to reduce to the case where each boundary component contains $2$ vertices; however at this point things become a little more difficult, and so we take a different approach.

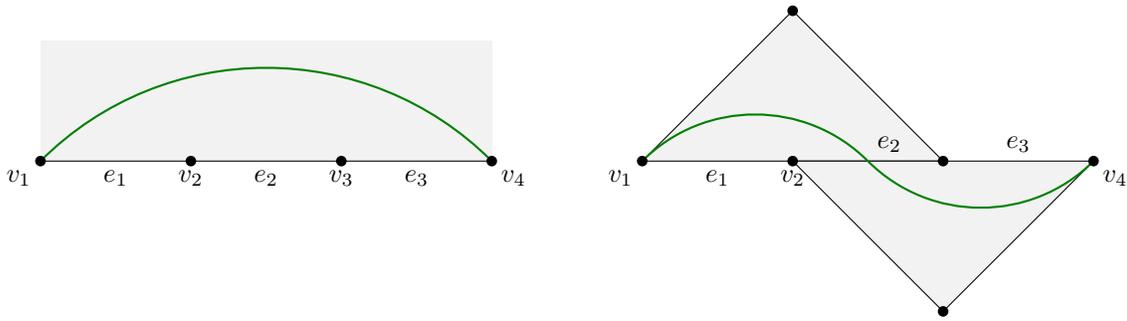
\begin{figure}
\begin{center}

\begin{tikzpicture}[
scale=2, 
fill = gray!10,
decomposition/.style={thick, draw=green!50!black} ]

\coordinate [label = below left:{$v_1$}] (v1) at (0,0);
\coordinate [label = below:{$v_2$}] (v2) at (1,0);
\coordinate [label = below:{$v_3$}] (v3) at (2,0);
\coordinate [label = below right:{$v_4$}] (v4) at (3,0);

\fill (0,0) rectangle (3,0.8);
\draw (v1) -- node [midway, below] {$e_1$} (v2) -- node [midway, below] {$e_2$} (v3) -- node [midway, below] {$e_3$} (v4);
\draw [decomposition] (v1) to [bend left=45] (v4);

\coordinate [label = below left:{$v_1$}] (w1) at (4,0);
\coordinate [label = below:{$v_2$}] (w2) at (5,0);
\coordinate [label = below:{$v_3$}] (w3) at (6,0);
\coordinate [label = below right:{$v_4$}] (w4) at (7,0);
\coordinate (upper) at (5,1);
\coordinate (lower) at (6,-1);

\filldraw (w1) -- node [midway, below] {$e_1$} (w2) -- node [midway, above right] {$e_2$} (w3) -- (upper) -- cycle;
\filldraw (w2) -- (w3) -- node [midway, above] {$e_3$} (w4) -- (lower) -- cycle;
\draw [decomposition] (w1) to [bend left=45] ($ (w2) ! 0.5 ! (w3) $)
to [bend right=45] (w4);

\foreach \point in {v1, v2, v3, v4, w1, w2, w3, w4, lower, upper}
\fill [black] (\point) circle (1pt);

\end{tikzpicture}

\end{center}
\caption{Joining two vertices by an arc of a quadrangulation.}
\label{fig:length_3_path}
\end{figure}

We will describe Penner's technique from \cite{Penner87} first, then our own use of similar techniques. We think of vertices as punctures in a surface.

Recall the Teichmuller space $\mathcal{T}$ of a punctured surface is the space of marked hyperbolic structures on the surface, with cusps at the punctures. Decorate each vertex puncture with a horocycle; the space of such decorated hyperbolic structures is the decorated Teichmuller space $\tilde{\mathcal{T}}$. Penner shows $\tilde{\mathcal{T}}$ is a ball.

A generic point in $\tilde{\mathcal{T}}$ determines an ideal triangulation of the surface. This triangulation can be obtained by expanding the horocycles about the punctures and noting when they intersect. Each new intersection determines an arc between punctures, which is added to the triangulation; eventually enough arcs are seen to form an ideal triangulation. Generically these intersections all occur at different times in the expansion of horocycles, so there is no ambiguity in which arcs to choose.

Moving about in $\tilde{\mathcal{T}}$, different ideal triangulations are obtained, as different horocycle intersections are seen in different orders. Along a generic path in $\tilde{\mathcal{T}}_{(\Sigma,V)}$, horocycle intersections are seen at equal times only in pairs, and only at a finite set of decorated hyperbolic structures along the path.

As we cross a point on this generic path in $\tilde{\mathcal{T}}_{(\Sigma,V)}$ where two horocycles intersect at the same time, the construction of the ideal triangulation changes; a certain horocycle intersection now appears earlier, while another horocycle intersection now appears later. This change in the order of appearance of horocycle intersections may or may not affect the final ideal triangulation achieved; if it does, then the arc corresponding to the horocycle intersection which now appears earlier is added to the triangulation, while the arc corresponding to the horocycle intersection which now appears later is deleted form the triangulation. It is seen that the change in triangulation is a local move, which switches a diagonal within a square of the triangulation.

For any triangulation, there is a point in $\tilde{\mathcal{T}}$ which gives that triangulation. Given two triangulations, we may take two points in $\tilde{\mathcal{T}}$ which give those triangulations, and connect them by a generic path. As we traverse that path and consider the triangulations obtained at points on that path, we see it changes only by local moves which switch diagonals in squares. In this way it is shown that any two triangulations are related by these moves.

\begin{proof}[Sketch of proof]
We consider the Teichmuller space of hyperbolic structures on $(\Sigma,V)$, with cusps at the vertices, and geodesic edges between vertices. Decorate each vertex puncture with a horocycle; the space of such decorated hyperbolic structures is the decorated Teichmuller space $\tilde{\mathcal{T}}_{(\Sigma,V)}$. It is simply connected, in fact a ball, for the same reason as in Penner: taking a triangulation of $(\Sigma,V)$ by arcs running between vertices/punctures, the arcs may be given $\lambda$-lengths which may be chosen freely as positive reals; and so a homeomorphism constructed between $\tilde{\mathcal{T}}_{(\Sigma,V)}$ and $\mathbb{R}_+^{\# \text{arcs}}$.

Now, a point in $\tilde{\mathcal{T}}_{(\Sigma,V)}$, i.e. a choice of marked hyperbolic structure and horocycles about vertices, determines a \emph{quadrangulation} of $(\Sigma, V)$ as follows. As above we consider the horocycles expanding from the punctures. But we now distinguish between the \emph{positive} horocycles around positive vertices (from $V_+$), and \emph{negative} horocycles around negative vertices (from $V_-$). As we expand the horocycles, we take note of where horocycles \emph{of opposite sign} meet. Each such intersection determines a geodesic arc between vertices of opposite sign, hence a decomposing arc. We only take note of those decomposing arcs which are nontrivial, i.e. are not boundary edges. We know that adding enough decomposing arcs (and all horoballs eventually meet when expanded far enough), we eventually arrive at a quadrangulation.

Again, for a generic point in $\tilde{\mathcal{T}}_{(\Sigma,V)}$, the intersections of (oppositely signed) horocycles (not along boundary edges) occur at distinct times in the expansion of horocycles, so there is no ambiguity in which arcs to choose. And again, along a generic path in $\tilde{\mathcal{T}}_{(\Sigma,V)}$, simultaneous (opposite sign) horocycle intersections (not along boundary edges) are seen only at a finite set of decorated marked hyperbolic structures along the path; at one of this finite set of decorated marked hyperbolic structures, there is only one time at which multiple intersections occur, and at this time only two pairs of horocycles simultaneously intersect.

As we cross a point on this generic path in $\tilde{\mathcal{T}}_{(\Sigma,V)}$ where horocycles simultaneously intersect, the construction of the quadrangulation changes; a certain horocycle appears earlier, and a certain horocycle appears later. If the final quadrangulation changes, the effect is that the arc corresponding to the horocycle intersection which now appears earlier is added to the quadrangulation, and the arc corresponding to the horocycle intersection which now appears later is deleted from the quadrangulation.

Thus the effect on the quadrangulation is to remove one arc and add another, with the result being another quadrangulation. Removing the arc creates a hexagonal complementary region, where we have just removed a main diagonal; the new arc added must be a different main diagonal. Hence the effect on the quadrangulation is a diagonal slide.

For any quadrangulation $A$, there is a point in $\tilde{\mathcal{T}}_{(\Sigma,V)}$ which gives that quadrangulation $A$. For by adding extra arcs, $A$ can be extended to an ideal triangulation $I$. Just as in the punctured case (Penner section 6), if we set $\lambda$-lengths on all the edges to be $1$, then the corresponding decorated marked hyperbolic structure has the property that expanding horocycles gives intersections which determine the arcs of $I$. If we expand the same horocycles in the same way but only pay attention to non-boundary intersections between horocycles of opposite sign, then we obtain the quadrangulation $A$.

So, given two quadrangulations $A,A'$, we take two points in $\tilde{\mathcal{T}}_{(\Sigma,V)}$ which determine those quadrangulations, and connect them by a generic path. We traverse the path and consider the quadrangulations obtained. The quadrangulations are constant along the path, except at a finite set of points where the quadrangulation changes by a diagonal slide. Thus $A,A'$ are related by a sequence of diagonal slides.
\end{proof}

\section{Sutures}
\label{sec:sutures}

\subsection{Sutured background background}

Here we follow definitions from \cite{Me10_Sutured_TQFT}. (See also the similar ideas, but different terminology, of \cite{Zarev09, Zarev10}.) Sutures on a surface are a set of curves which divide the surface into signed regions, in a coherent way.

\begin{defn}
A \emph{sutured surface} $(\Sigma, \Gamma)$ is an oriented (topological) surface $\Sigma$ (possibly disconnected) with a properly embedded oriented 1-submanifold $\Gamma \subset \Sigma$, such that:
\begin{enumerate}
\item
$\Sigma \backslash \Gamma = R_+ \cup R_-$, where $R_\pm$ are surfaces oriented as $\pm \Sigma$;
\item
$\overline{\partial R_\pm \backslash \partial \Sigma} = \Gamma$ as oriented 1-manifolds; and
\item
For every component $C$ of $\partial \Sigma$, $C \cap \Gamma \neq \emptyset$.
\end{enumerate}
\end{defn}
Condition (i) says that $\Gamma$ cuts $\Sigma$ into positive and negative regions; (ii) says that the positive and negative regions are coherent; (iii) says that sutures hit every boundary component. (The requirement (iii) is not standard, but makes sense in our context.) Figure \ref{fig:sutured_nonsutured_surfaces} shows an example and a non-example.

\begin{figure}
\centering
\includegraphics[scale=0.4]{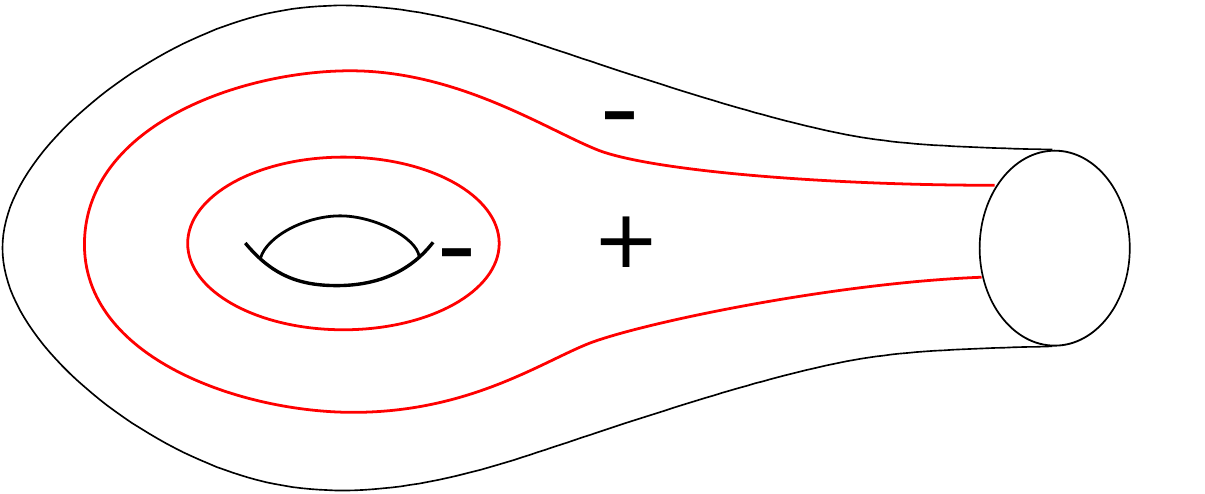}
\includegraphics[scale=0.4]{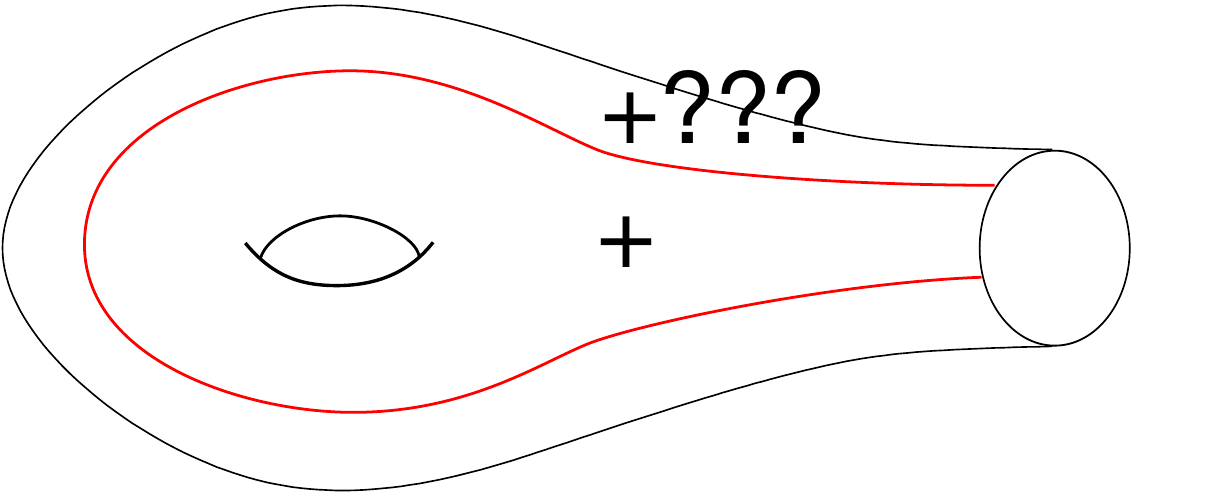}
\caption{An example and a non-example of a sutured surface.}
\label{fig:sutured_nonsutured_surfaces}
\end{figure}

We say that $\Gamma$ is a \emph{set of sutures} on $\Sigma$, and a component of $\Gamma$ is called a \emph{suture}. So as we cross a suture, we pass from a positive region $R_+$ into a negative region $R_-$, or vice versa.

Consider the boundary data of a sutured surface $(\Sigma, \Gamma)$. It is just a distinguished set of signed points $F = \partial \Gamma$ on the boundary, and signed intervals $C_\pm = \partial R_\pm$.
\begin{defn}
A \emph{sutured background} $(\Sigma,F)$ is an oriented (topological) surface $\Sigma$ with nonempty boundary, and a finite set $F$ of signed points on $\partial \Sigma$, such that
\begin{enumerate}
\item
$\partial \Sigma \backslash F = C_+ \cup C_-$ where $C_\pm$ is oriented as $\pm \partial \Sigma$;
\item
$\partial C_\pm = - F$ as signed points. (Here $\partial \Sigma$ inherits an orientation from $\Sigma$.)
\item
For every component $C$ of $\partial \Sigma$, $C \cap F \neq \emptyset$.
\end{enumerate}
\end{defn}
Note that each component $C$ of $\partial \Sigma$ contains an even number of points, at least $2$, of $F$, alternating in sign; crossing a point of $F$, we pass from a positive arc $C_+$ into a negative arc $C_-$ or vice versa. 

It's easy to verify that the boundary data $F = \partial \Gamma$, $C_\pm = \partial R_\pm \cap \partial \Sigma$ of a sutured surface determines a sutured background. See figure \ref{fig:sutured_background}. A \emph{set of sutures on a sutured background $(\Sigma,F)$} is a set of sutures which restricts to $(\Sigma,F)$ on the boundary in this way. Thus, sutures are a way of ``filling in'' the boundary data of a sutured background, ``joining the dots'' of $F$ with arcs and closed curves compatible with signs. Note that for a point $f \in F$, there is precisely one suture of $\Gamma$ ending at $f$; if there were more than one, $\Gamma$ would not be properly embedded.

\begin{figure}
\centering
\includegraphics[scale=0.4]{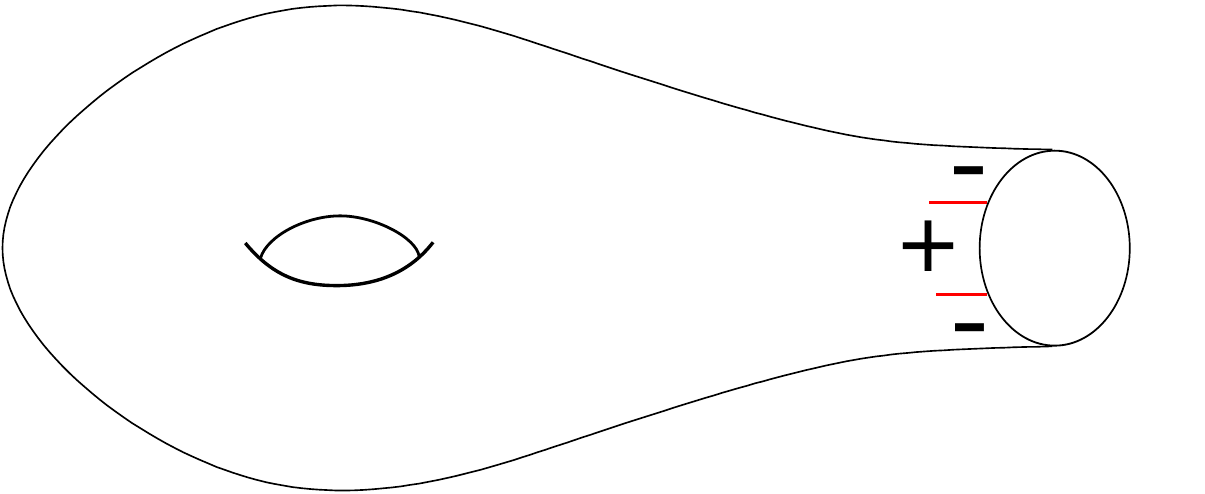}
\caption{A sutured background surface.}
\label{fig:sutured_background}
\end{figure}

A neighbourhood of a suture which is a properly embedded arc looks like figure \ref{fig:neighbourhood_of_arc_suture} below. (The reader might wonder why we set $\partial C_\pm = -F$, rather than $F$. The answer is that, if we want $\partial R_\pm \cap \Gamma = C_\pm$ and $\partial \Gamma = F$, then $C_\pm$ is naturally oriented from $F_+$ to $F_-$. We cannot have both $F = \partial \Gamma$ and $F = \partial C_\pm$.)

\begin{figure}
\begin{center}

\begin{tikzpicture}[
scale=2, 
fill = gray!10,
decomposition/.style={thick, draw=green!50!black}, 
vertex/.style={draw=green!50!black, fill=green!50!black},
suture/.style={thick, draw=red, postaction={nomorepostaction,decorate, decoration={markings,mark=at position 0.5 with {\arrow{>}}}}},
midarrow/.style={thick, postaction={nomorepostaction,decorate, decoration={markings,mark=at position 0.5 with {\arrow{>}}}}} ]

\coordinate [label = below left:{$V_-$}] (bl) at (0,0);
\coordinate [label = left:{$F_-$}] (ml) at (0,1);
\coordinate [label = above left:{$V_+$}] (tl) at(0,2);
\coordinate [label = below right:{$V_-$}] (br) at (3,0);
\coordinate [label = right:{$F_+$}] (mr) at (3,1);
\coordinate [label = above right:{$V_+$}] (tr) at (3,2);

\fill (0,0) rectangle (3,2);

\draw [suture] (ml) to node [midway, above] {$\Gamma$} (mr);
\draw [midarrow] (bl) to node [midway, left] {$C_-$} (ml);
\draw [midarrow] (tl) to node [midway, left] {$C_+$} (ml);

\draw [midarrow] (mr) to node [midway, right] {$C_-$} (br);
\draw [midarrow] (mr) to node [midway, right] {$C_+$} (tr);

\draw (1.5,0.5) node {$R_-$};
\draw (1.5,1.5) node {$R_+$};

\foreach \point in {tl, bl, tr, br}
\fill [vertex] (\point) circle (2pt);
\foreach \point in {ml, mr}
\fill [red] (\point) circle (2pt);

\end{tikzpicture}

\end{center}
\caption{Neighbourhood of a suture which is an arc.}
\label{fig:neighbourhood_of_arc_suture}
\end{figure}
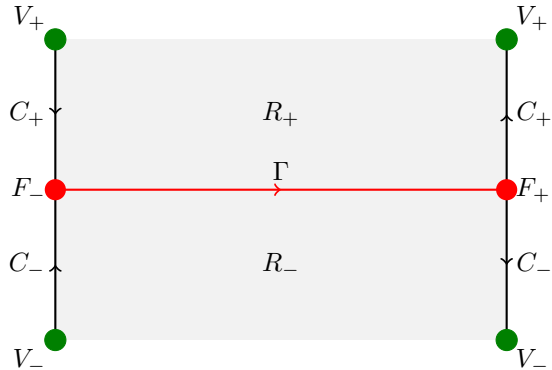

An embedding of sutured surfaces $(\Sigma, \Gamma) \To (\Sigma', \Gamma')$ is an embedding $\Sigma \To \Sigma'$ which sends $\Gamma \mapsto \Gamma'$ and $R_\pm \mapsto R'_\pm$. Similarly we may speak of a homeomorphism of sutured surfaces. A homeomorphism of sutured background surfaces $(\Sigma,F) \To (\Sigma',F')$ is a homeomorphism $\Sigma \To \Sigma'$ which sends $F \mapsto F'$ preserving signs.  

Two sets of sutures on $(\Sigma, F)$ are \emph{equivalent} if they are isotopic rel boundary in $(\Sigma, F)$.  

We now introduce two important properties of sutures. Both involve subsurfaces bounded by sutures.
\begin{defn}
A suture is \emph{trivial} if it is a contractible closed curve. A set of sutures is \emph{trivial} if it contains a trivial suture.
\end{defn}

\begin{defn}
A set of sutures $\Gamma$ is \emph{confining} if there is a component of $\Sigma \backslash \Gamma$ which does not intersect $\partial \Sigma$. \end{defn}
Confining sutures thus confine a region of $\Sigma \backslash \Gamma$, which cannot escape to $\partial \Sigma$. The disc bounded by a trivial suture is a special case, so trivial implies confining, and nonconfining implies nontrivial.

\subsection{Cutting and gluing sutures}

Sutured surfaces can be cut and glued together. 

To cut, we may take a properly embedded arc, or closed curve $c$ in a sutured surface $(\Sigma, \Gamma)$. Assuming that $c$ is transverse to $\Gamma$, and does not intersect $c$ on $\partial \Sigma$, then cutting along $c$ produces a surface $\Sigma'$ with a properly embedded 1-submanifold $\Gamma'$, and subsurfaces $\Sigma' \backslash \Gamma' = R'_- \sqcup R'_+$ with the required orientation properties, but $\Gamma'$ may not intersect every component of $\partial \Sigma'$. However, if $c$ intersects $\Gamma$, then this is guaranteed.

Conversely, to glue, we may take two disjoint arcs $c,c' \subset \partial \Sigma$ on a (possibly disconnected) sutured surface $(\Sigma, \Gamma)$, and a homeomorphism $\phi: c \To c'$ which preserves the sutured structure, $c \cap \Gamma \mapsto c' \cap \Gamma$ and $c \cap R_\pm \mapsto c' \cap R_\pm$. Gluing via $\phi$ produces a surface $\Sigma'$ with a properly embedded 1-submanifold $\Gamma'$. In order that $\Sigma'$ be orientable, $\phi$ must be orientation-reversing, where $c,c'$ inherit orientations from $\partial \Sigma$. Then we obtain subsurfaces $R'_\pm$ with coherent signs, but $\Gamma'$ may not intersect every component of $\partial \Sigma'$. However if $|c \cap \Gamma| = |c' \cap \Gamma| = 1$, and these intersection points are not consecutive points around $\partial \Sigma$, then $\Gamma'$ intersects every component of $\partial \Sigma'$.

To summarise:
\begin{lem}
\label{lem:cutting_sutures}
\label{lem:gluing_sutures}
Let $(\Sigma,\Gamma)$ be a (possibly disconnected) sutured surface.
\begin{enumerate}
\item
If $c$ is a simple closed curve or properly embedded arc in $\Sigma$ which is transverse to $\Gamma$ and intersects $\Gamma$ in at least one point in its interior, then cutting along $c$ gives a sutured surface.
\item
If $c,c' \subset \partial \Sigma$ are disjoint arcs, each inheriting an orientation from $\partial \Sigma$ and intersecting $\Gamma$ precisely once at non-consecutive points of $\partial \Gamma$ around $\partial \Sigma$, and $\phi: c \To c'$ is an orientation-reversing homeomorphism sending $c \cap R_\pm \mapsto c' \cap R_\pm$ and $c \cap \Gamma \mapsto c' \cap \Gamma$, then gluing $\Sigma$ along $\phi$ gives a sutured surface.
\end{enumerate}
\qed
\end{lem}

\subsection{Euler class}

\begin{defn}
The \emph{Euler class} $e$ of a sutured surface $(\Sigma, \Gamma)$ is $e = \chi(R_+) - \chi(R_-)$.
\end{defn}
(This terminology comes from contact geometry. A sutured surface $(\Sigma, \Gamma)$ determines a contact structure $\xi$ in a neighbourhood of $\Sigma$ embedded in a 3-manifold, and $e(\Gamma)$ is the evaluation of its Euler class on $\Sigma$. See \cite{Gi91}.)

By additivity of Euler characteristic, we have immediately under disjoint union,
\[
e \left( (\Sigma, \Gamma) \sqcup (\Sigma', \Gamma') \right) = e(\Sigma, \Gamma) + e(\Sigma', \Gamma').
\]
Cutting or gluing sutured surfaces often preserves Euler class; however it's not difficult to see that cutting along or gluing together arcs which intersect an even number of points, changes the Euler class. The following lemma is clear upon taking an appropriate cell decomposition of $\Sigma$, compatible with $R_+$ and $R_-$.
\begin{lem}
\label{lem:euler_class_invariant}
Let $(\Sigma, \Gamma)$ and $(\Sigma', \Gamma')$ be sutured surfaces where $(\Sigma', \Gamma')$ is obtained from $(\Sigma, \Gamma)$ by:
\begin{enumerate}
\item
cutting along a simple closed curve transverse to $\Gamma$ (which necessarily intersects $\Gamma$ in an even number of points);
\item
cutting along a properly embedded arc in $\Sigma$ which intersects $\Gamma$ in an odd number of points;
\item
gluing together two boundary components $C,C'$ of $\partial \Sigma$ (which necessarily intersect $\Gamma$ in an even number of points each), via an orientation-reversing homeomorphism $\phi: C \To C'$, which sends $C \cap R_\pm \mapsto C' \cap R_\pm$, $C \cap \Gamma \mapsto C' \cap \Gamma$;
\item
gluing together two disjoint arcs $c,c' \subset \partial \Sigma$ (oriented from $\partial \Sigma$), such that $|c \cap \Gamma| = |c' \cap \Gamma|$ is odd, via an orientation-reversing homeomorphism $\phi: c \To c'$ which sends $c \cap R_\pm \mapsto c' \cap R_\pm$, $c \cap \Gamma \mapsto c' \cap \Gamma$.
\end{enumerate}
Then $e(\Gamma') = e(\Gamma)$.
\qed
\end{lem}
(Note that the assumption that cutting or gluing produces a sutured surface implies that not all topological gluings or cuttings described above are possible. For instance in case (i) the simple closed curve intersects $\Gamma$ nontrivially. Other restrictions apply in cases (iii) and (iv).)

\subsection{Adjusting sutures}

We consider now the elementary surgery on sutures, \emph{bypass surgery}, introduced by Honda in \cite{Hon00I}.

A \emph{bypass disc} is a sutured disc homeomorphic to one of the discs of figure \ref{fig:bypass_surgery}. \emph{Bypass surgery} is performed on a sutured surface $(\Sigma, \Gamma)$ along an embedded bypass disc, by removing the bypass disc and replacing it with a different bypass disc, as depicted in figure \ref{fig:bypass_surgery}. There are two possible replacements, up to equivalence; we call the operation \emph{upwards} or \emph{downwards} surgery as shown.

A bypass disc arises as the neighbourhood of an arc $c$ intersecting sutures in three points, drawn in figure \ref{fig:bypass_surgery}. An \emph{attaching arc} (see \cite{Hon00I}) is an embedded arc which intersects $\Gamma$ at its endpoints and at precisely one interior point. Isotopy classes of attaching arcs correspond bijectively to isotopy classes of bypass discs; a bypass disc determines an attaching arc connecting its sutures, and an attaching arc thickens to a bypass disc. So we may speak of bypass surgery along an attaching arc.

Obviously if two bypass discs are isotopic, or equivalently two attaching arcs are isotopic, then the sutures resulting from upwards (resp. downwards) bypass surgery along them are equivalent.

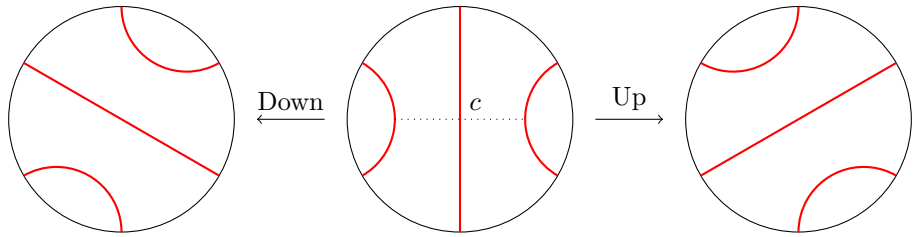
\begin{figure}
\begin{center}

\begin{tikzpicture}[
scale=1.5, 
suture/.style={thick, draw=red}]

\draw (3,0) circle (1 cm); 	
\draw (0,0) circle (1 cm);
\draw (-3,0) circle (1 cm);

\draw [suture] (30:1) arc (120:240:0.57735);
\draw [suture] (0,1) -- (0,-1);
\draw [suture] (210:1) arc (-60:60:0.57735);
\draw [dotted] (-0.57735,0) -- node [above right] {$c$} (0.57735,0);

\draw [suture] (3,0) ++ (150:1) arc (-120:0:0.57735);
\draw [suture] (3,0) ++ (30:1) -- ($ (3,0) + (210:1) $);
\draw [suture] (3,0) ++ (-30:1) arc (60:180:0.57735);

\draw [suture] (-3,1) arc  (180:300:0.57735);
\draw [suture] ($ (-3,0) + (-30:1) $) -- ($ (-3,0) +  (150:1) $);
\draw [suture] (-3,-1) arc (0:120:0.57735);

\draw[->] (-1.2,0) -- node [above] {Down} (-1.8,0);
\draw[->] (1.2,0) -- node [above] {Up} (1.8,0);

\end{tikzpicture}

\caption{Bypass surgery.}
\label{fig:bypass_surgery}
\end{center}
\end{figure}

Sets of sutures related by bypass surgeries naturally come in triples, which we call \emph{bypass triples}. Bypass surgery has order $3$, in the sense that performing bypass surgery in the same direction $3$ times on the same disc results in sutures equivalent to the original. Upwards and downwards surgeries on the same disc are inverse operations.

Taking appropriate cell decompositions, we have:
\begin{prop}
\label{prop:bypass_surgery_preserves_Euler}
Bypass surgery preserves $\chi(R_+)$, $\chi(R_-)$, and Euler class $e$.
\qed
\end{prop}

\subsection{Maintaining nontrivial sutures}

There are situations in which, provided we begin with nontrivial sutures $\Gamma$, we can ensure that a bypass surgery results again in nontrivial sutures.

Throughout this section, let $(\Sigma, \Gamma)$ be a sutured surface with $\Gamma$ nontrivial, and $(D, \Gamma_3)$ an embedded bypass disc. As in figure \ref{fig:distinct_sutures_surgery}, let the three components of $\Gamma_3$ be $\gamma_1, \gamma_2, \gamma_3$. Each $\gamma_i$ forms part of a suture $\delta_i$ of $\Gamma$, where the $\delta_i$ need not be distinct.

\begin{lem}
\label{lem:distinct_sutures_surgery}
If $\delta_1, \delta_2, \delta_3$ are distinct sutures of $\Gamma$, then bypass surgery on $(\Sigma, \Gamma)$ at $(D, \Gamma_3)$, in either direction, gives a nontrivial set of sutures.
\end{lem}

\begin{proof}
Each suture $\delta_i$ is either a properly embedded arc in $\Sigma$, or a closed curve.

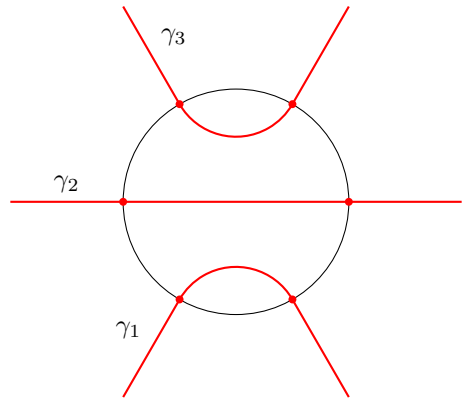
\begin{figure}
\begin{center}

\begin{tikzpicture}[
scale=1.5, 
suture/.style={thick, draw=red}]

\coordinate (a) at (-60:1);
\coordinate (b) at (0:1);
\coordinate (c) at (60:1);
\coordinate (d) at (120:1);
\coordinate (e) at (180:1);
\coordinate (f) at (240:1);

\draw (0,0) circle (1 cm);

\draw [suture] (120:2) -- node [above right] {$\gamma_3$} (120:1) arc (-150:-30:0.57735) -- (60:2);
\draw [suture] (-2,0) -- node [above] {$\gamma_2$} (-1,0) -- (2,0);
\draw [suture] (-60:2) -- (-60:1) arc (30:150:0.57735) -- node [above left] {$\gamma_1$} (240:2);

\foreach \point in {a,b,c,d,e,f}
\fill [red] (\point) circle (1pt);

\end{tikzpicture}

\caption{Bypass disc in lemma \ref{lem:distinct_sutures_surgery}}
\label{fig:distinct_sutures_surgery}
\end{center}
\end{figure}

Suppose $\delta_1$ and $\delta_2$ are both arcs. Then, after surgery (in either direction), all of the sutures intersecting $D$ connect to an endpoint of $\delta_1$ or $\delta_2$, and hence must be arcs. As all other sutures are unchanged, there are no new closed sutures, and hence the sutures remain nontrivial. A similar argument applies when any two of the  $\delta_i$ are arcs; in these cases bypass surgery also produces nontrivial sutures.

Next suppose $\delta_1$ is an arc but $\delta_2$ and $\delta_3$ are closed curves. Then we observe that after surgery (in either direction), these three sutures are merged into one arc. So no new closed curve arises, and the sutures remain nontrivial. A similar argument applies when any of the $\delta_i$ is an arc, and the other two are closed curves.

We may now assume that all the $\delta_i$ are closed curves. Then performing surgery (in either direction) merges the $\delta_i$ into one closed curve $\delta$. If this curve is contractible, then it bounds an embedded disc $E$ in $\Sigma$; so $\delta$ bounds an embedded disc on one side. This disc $E$ passes through the bypass disc $D$, and is cut by it into four components, which we label $E_1, E_2, E_3, E_4$ as in figure \ref{fig:closed_curve_surgery}. Each of these regions $E_i$ must themselves be embedded discs in $\Sigma$. But this implies that at least one of the original curves was contractible. For the arrangement in figure \ref{fig:closed_curve_surgery}, the disc $E_4$ implies $\delta_1$ is contractible; if $\delta$ bounds a disc on the other side, or we perform surgery in the other direction, a similar argument applies.

\begin{figure}
\begin{center}

\begin{tikzpicture}[
scale=1.2, 
suture/.style={thick, draw=red}]

\draw (0,0) circle (1 cm);
\draw [suture] (120:1) arc (-150:-30:0.57735);
\draw [suture, dotted] (60:1) arc (-30:210:0.57735);
\draw [suture] (-1,0) -- (1,0);
\draw [suture, dotted, rounded corners] (1,0) -- (2.5,0) arc (0:180:2.5) -- (-1,0);
\draw [suture] (-60:1) arc (30:150:0.57735);
\draw [suture, dotted] (240:1) arc (-210:30:0.57735);

\draw[->] (3,0) -- node [above] {Up} (4,0);

\draw [xshift = 7 cm] (0,0) circle (1 cm);
\draw [xshift = 7 cm, suture] (60:1) arc (150:270:0.57735);
\draw [xshift = 7 cm, suture, dotted] (60:1) arc (-30:210:0.57735);
\draw [xshift = 7 cm, suture] (120:1) -- (-60:1);
\draw [xshift = 7 cm, suture, dotted, rounded corners] (1,0) -- (2.5,0) arc (0:180:2.5) -- (-1,0);
\draw [xshift = 7 cm, suture] (-120:1) arc (-30:90:0.57735);
\draw [xshift = 7 cm, suture, dotted] (240:1) arc (-210:30:0.57735);

\draw (5.5,1) node {$E_2$};
\draw (6.8, -0.3) node {$E_3$};
\draw (7, -1.4) node {$E_4$};
\draw (7.7, 0.4) node {$E_1$};

\end{tikzpicture}

\caption{Surgery in lemma \ref{lem:distinct_sutures_surgery} when all $\gamma_i$ are closed curves.}
\label{fig:closed_curve_surgery}
\end{center}
\end{figure}
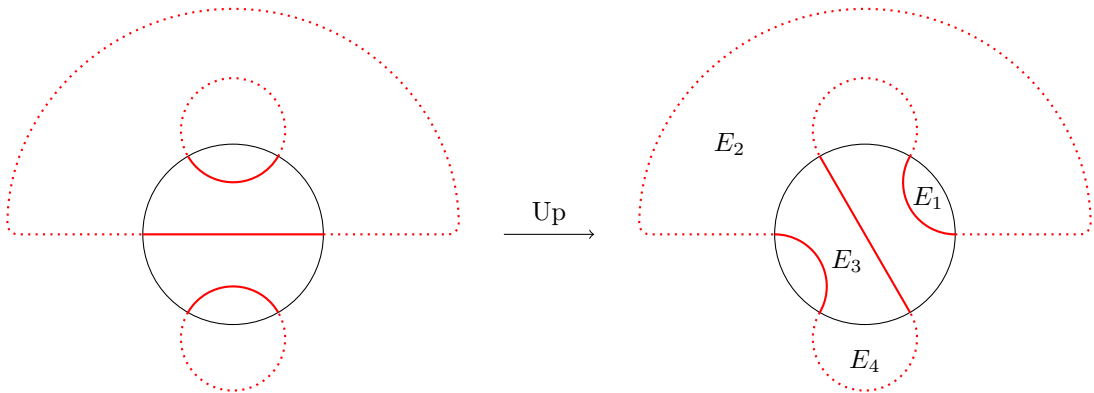

\end{proof}

More generally, attaching arcs which intersect sutures ``efficiently'' maintain nontrivial sutures. An attaching arc $c$ is \emph{inefficient} if there is an embedded disc in $\Sigma$ bounded by a segment of $c$ and a segment of $\Gamma$; otherwise it is \emph{efficient}. Bypass discs can thereby also be called efficient or inefficient.

This use of ``efficient'' corresponds to the usual meaning of minimal number of intersections, in the following sense. Let $C$ be a simple closed curve or properly embedded arc in $\Sigma$ which intersects $\Gamma$ transversely and efficiently, i.e. $|C \cap \Gamma|$ is minimal in the isotopy class of $C$. Then any sub-arc $c$ of $C$ which is an attaching arc is efficient: if $c$ were inefficient, then the disc bounded by $c \cup \Gamma$ would allow us to isotope sutures off $C$ and reduce $|C \cap \Gamma|$.

\begin{prop}
\label{prop:efficient_surgery}
If $(\Sigma, \Gamma)$ is a surface with nontrivial sutures, then bypass surgery along an efficient embedded bypass disc $(D, \Gamma_3)$, in either direction, results in a nontrivial set of sutures.
\end{prop}

\begin{proof}
If all the $\delta_i$ are distinct, then by lemma \ref{lem:distinct_sutures_surgery}, the result of bypass surgery is nontrivial. So we may assume that at least two of the $\delta_i$ coincide.

First suppose that $\delta_1 = \delta_2$ but these do not coincide with $\delta_3$. If $\delta_1, \delta_3$ are both arcs then (considering how sutures respect orientations) the situation is as in figure \ref{fig:efficient_surgery_1} (or a reflection thereof). Bypass surgery along $D$ in one direction produces no sutures which are closed curves, so the result is nontrivial; in the other direction (depicted in the figure), part of $\delta_1$ is closed off to form a suture $\delta'$. If the result is trivial, then $\delta'$ bounds a disc; but then there was a disc bounded by $\delta_1$ and $c$ before surgery, contradicting efficiency. A similar argument applies if one or both of $\delta_1, \delta_3$ is a closed curve. And a similar argument applies if $\delta_2 = \delta_3 \neq \delta_1$ or if $\delta_1 = \delta_3 \neq \delta_2$. In all these cases the result of bypass surgery is a nontrivial set of sutures.

\begin{figure}
\begin{center}

\begin{tikzpicture}[
scale=1.2, 
suture/.style={thick, draw=red}]

\draw (0,0) circle (1 cm);
\draw [dotted] (0,-0.57735) -- node [above right] {$c$} (0,0.57735);
\draw [suture] (120:1) arc (-150:-30:0.57735);
\draw [suture] (-1,0) -- (1,0);
\draw [suture] (-60:1) arc (30:150:0.57735);
\draw [suture, dotted] (120:2) -- (120:1);
\draw [suture, dotted] (60:2) -- (60:1);
\draw [suture, dotted] (-120:2) -- (-120:1);
\draw [suture, dotted] (180:2) -- (180:1);
\draw [suture, dotted] (-60:1) arc (-150:90:0.57735);

\draw[->] (3,0) -- node [above] {Down} (4,0);

\draw [xshift = 7 cm] (0,0) circle (1 cm);
\draw [xshift = 7 cm, suture] (180:1) arc (-90:30:0.57735);
\draw [xshift = 7 cm, suture] (-120:1) -- (60:1);
\draw [xshift = 7 cm, suture] (0:1) arc (90:210:0.57735);
\draw [xshift = 7 cm, suture, dotted] (120:2) -- (120:1);
\draw [xshift = 7 cm, suture, dotted] (60:2) -- (60:1);
\draw [xshift = 7 cm, suture, dotted] (-120:2) -- (-120:1);
\draw [xshift = 7 cm, suture, dotted] (180:2) -- (180:1);
\draw [xshift = 7 cm, suture, dotted] (-60:1) arc (-150:90:0.57735);

\end{tikzpicture}

\caption{Surgery in proposition \ref{prop:efficient_surgery} when $\delta_1 = \delta_2 \neq \delta_3$.}
\label{fig:efficient_surgery_1}
\end{center}
\end{figure}

Finally, suppose $\delta_1 = \delta_2 =  \delta_3$. Then (up to symmetry of the diagram) the situation is as in figure \ref{fig:efficient_surgery_2}. So all the $\delta_i$ are a single closed suture. If, after surgery, sutures bound a disc, then we obtain a disc (in fact several) bounded by the $\delta_i$ and $c$ before surgery, contradicting efficiency. So again we obtain nontrivial sutures.

\begin{figure}
\begin{center}

\begin{tikzpicture}[
scale=1.2, 
suture/.style={thick, draw=red}]

\draw (0,0) circle (1 cm);
\draw [dotted] (0,-0.57735) -- node [above right] {$c$} (0,0.57735);
\draw [suture] (120:1) arc (-150:-30:0.57735);
\draw [suture] (-1,0) -- (1,0);
\draw [suture] (-60:1) arc (30:150:0.57735);
\draw [suture, dotted] (0:1) arc (-90:150:0.57735);
\draw [suture, dotted] (180:1) arc (90:330:0.57735);
\draw [suture, dotted, rounded corners] (120:1) -- (120:2.5) arc (120:300:2.5) -- (300:1);

\draw[->] (3,0) -- node [above] {Down} (4,0);

\draw [xshift = 7 cm] (0,0) circle (1 cm);
\draw [xshift = 7 cm, suture] (180:1) arc (-90:30:0.57735);
\draw [xshift = 7 cm, suture] (-120:1) -- (60:1);
\draw [xshift = 7 cm, suture] (0:1) arc (90:210:0.57735);
\draw [xshift = 7 cm, suture, dotted] (0:1) arc (-90:150:0.57735);
\draw [xshift = 7 cm, suture, dotted] (180:1) arc (90:330:0.57735);
\draw [xshift = 7 cm, suture, dotted, rounded corners] (120:1) -- (120:2.5) arc (120:300:2.5) -- (300:1);

\end{tikzpicture}

\caption{Surgery in proposition \ref{prop:efficient_surgery} when $\delta_1 = \delta_2 \neq \delta_3$.}
\label{fig:efficient_surgery_2}
\end{center}
\end{figure}

\end{proof}

\section{Quadrangulations and sutures}
\label{sec:quadrangulations_and_sutures}

Both quadrangulations and sutures involve ``joining the dots'' on a surface, where ``dots'' mean vertices $V$ of an occupied surface or points $F$ of a background surface. Moreover, occupied surfaces and sutured background surfaces have very similar definitions; in fact, they equivalent structures. However different rules apply in joining vertices of $V$ with the arcs of a quadrangulation, to those applying in joining points of $F$ with sutures. We now consider these relationships in more detail.

\subsection{Sutures on occupied surfaces}

The structure of a background surface is equivalent to that of an occupied surface in a natural way. Given a background surface $(\Sigma,F)$, we have signed points $F$ which split $\partial \Sigma$ into arcs $C_\pm$ of alternating orientations; placing one vertex of $V_\pm$ in each component of $C_\pm$, we have an occupied surface. Conversely, given an occupied surface $(\Sigma,V)$, we have vertices around each boundary component of $\Sigma$ of alternating signs. Place one point of $F$ between each pair of consecutive vertices of $V$; each component of $\partial \Sigma \backslash F$ then contains a vertex of $V_\pm$; label components of $\partial \Sigma \backslash F$ in $C_\pm$ accordingly as they contain vertices of $V_\pm$; orient $F$ as $-\partial C_\pm$. In this way we obtain a background surface. It is clear that from an occupied $(\Sigma,V)$, the background surface $(\Sigma,F)$ associated to it is unique up to homeomorphism; similarly from a background $(\Sigma,F)$, the occupied $(\Sigma,V)$ is unique up to homeomorphism; and these two constructions are inverses up to homeomorphism. See figure \ref{fig:occupied_background}.

\begin{figure}
\centering
\includegraphics[scale=0.4]{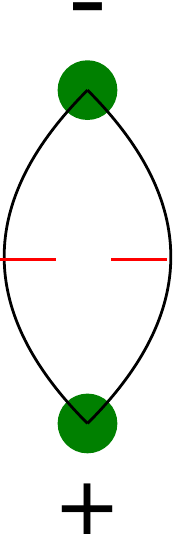} \quad
\includegraphics[scale=0.4]{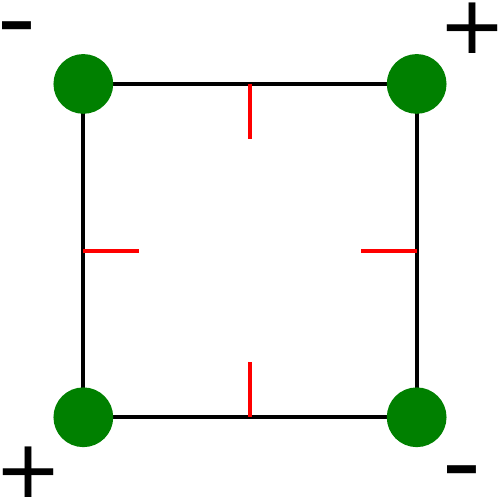} \quad 
\includegraphics[scale=0.4]{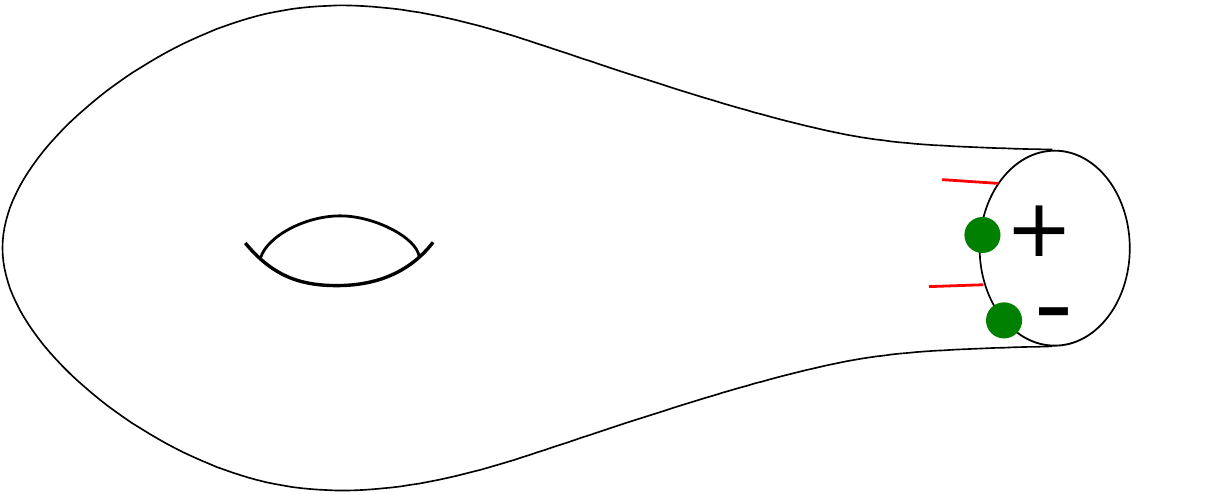}
\caption{Occupied surfaces are equivalent to sutured backgrounds.}
\label{fig:occupied_background}
\end{figure}

Thus we may speak of an \emph{occupied background} $(\Sigma,F,V)$, which is a triple such that $(\Sigma,V)$ and $(\Sigma,F)$ are associated occupied and background surface structures on $\Sigma$. Henceforth when we write $(\Sigma,V)$, it will be implicit that we are also dealing with a sutured background; and when we write $(\Sigma,F)$, it will be implicit that we are also dealing with an occupied surface. The key property connecting occupied and sutured structures is that vertices $V_\pm$ lie in the regions $R_\pm$.

From an occupied surface $(\Sigma,F,V)$, we may ``join the dots'', whether of $V$ to obtain a quadrangulation, or of $F$ to obtain sutures, or both. Thus we may speak of a \emph{quadrangulated background surface} $(\Sigma,F,A)$, a \emph{sutured occupied surface} $(\Sigma,\Gamma,V)$, and a \emph{sutured quadrangulated surface} $(\Sigma,\Gamma,A)$.

\begin{tabular}{cccc}
&& $\To$  & \\
&& Quadrangulate & \\
& Occupied background & & Quadrangulated background \\
& $(\Sigma,F,V)$ & $\To$ & $(\Sigma,F,A)$ \\
$\downarrow$ Suture & $\downarrow$  && $\downarrow$ \\
& Sutured occupied && Sutured quadrangulated \\
& $(\Sigma,\Gamma,V)$ & $\To$ & $(\Sigma, \Gamma, A)$
\end{tabular}

When we consider sutures and quadrangulations together, i.e. a sutured quadrangulated surface $(\Sigma, \Gamma, A)$, we shall usually require that arcs of $A$ and sutures of $\Gamma$ intersect transversely and efficiently.

We can now indicate the origin of the terms incoming and outgoing, applied to boundary edges of an occupied surface.
\begin{lem}
In a sutured occupied surface $(\Sigma, \Gamma, V)$, every suture $\gamma$ of $\Gamma$ which is an arc runs from an incoming boundary edge to an outgoing boundary edge of $(\Sigma, V)$.
\end{lem}

\begin{proof}
Traversing the oriented arc $\gamma$ there is a region of $R_+$ to our left, and a region of $R_-$ to our right. Thus, $\gamma$ begins on a boundary edge $e$ of $(\Sigma, V)$ with a positive vertex to its left and a negative vertex to its right. Thus the orientation on $e$ from negative to positive endpoint disagrees with the orientation inherited from $\partial \Sigma$, and $e$ is incoming. Similarly the edge at which $\gamma$ ends is outgoing. See figure \ref{fig:neighbourhood_of_arc_suture}.
\end{proof}

We next consider the possible sutures on the simplest occupied surfaces.

Sutures on the occupied vacuum $(\Sigma^\emptyset, V^\emptyset)$ join two endpoints on the boundary of a disc. Hence there is precisely one nontrivial set of sutures, up to equivalence, which we call \emph{vacuum sutures} or just \emph{the vacuum}. (Any closed curve sutures would be contractible and hence trivial.) The vacuum sutures have Euler class $e=0$.

For the occupied square $(\Sigma^\square, V^\square)$, we have a sutured background disc with $4$ points, and see there are precisely two nontrivial sets of sutures, up to equivalence, as illustrated in figure \ref{fig:sutured_squares}. One of these has $e=1$, which we call the \emph{standard positive sutures} on the occupied square; the other has $e=-1$, the \emph{standard negative sutures}.

\subsection{Sutures on quadrangulated surfaces}

We now consider in more detail sutures on quadrangulated surfaces. We might have a sutured surface and wish to quadrangulate it; or we might have a quadrangulated surface and wish to suture it.

First suppose we have a sutured surface, and wish to quadrangulate it. Note that any decomposing arc $a$ runs between vertices $v_+, v_-$ of opposite sign, hence between regions $R_+, R_-$ of opposite sign. If $a$ is transverse to $\Gamma$ then it intersects an odd number of times. Thus from lemmas 
\ref{lem:cutting_sutures} and \ref{lem:euler_class_invariant}, we immediately have the following.
\begin{lem}
Cutting a sutured occupied surface $(\Sigma, \Gamma, V)$ along a decomposing arc $a$ transverse to $\Gamma$ gives another (possibly disconnected) sutured occupied surface $(\Sigma',\Gamma',V')$ with $e(\Gamma') = e(\Gamma)$.
\qed
\end{lem}
We can also consider gluing two boundary edges together. As discussed previously (section \ref{sec:decomposition_and_gluing}), this is simplest when we glue two non-consecutive edges, i.e. perform (section \ref{sec:simple_morphisms}) standard gluing morphisms. 

Note that when gluing boundary edges $e_1, e_2$ of a sutured occupied surface together, the result does not make any sense as a sutured surface unless we glue $e_1 \cap \Gamma$ to $e_2 \cap \Gamma$ and $e_1 \cap R_\pm$ to $e_2 \cap R_\pm$. Henceforth, when considering morphisms which glue edges together (i.e. standard gluings, folds, zips), we will always assume this consistency with sutures.

With this assumption, lemmas \ref{lem:gluing_sutures} and \ref{lem:euler_class_invariant} immediately give:
\begin{lem}
\label{lem:gluing_sutures_preserves_euler}
Performing a standard gluing morphism on a sutured occupied surface $(\Sigma,\Gamma,V)$ gives another sutured occupied surface $(\Sigma',\Gamma',V')$ with $e(\Gamma')=e(\Gamma)$.
\qed
\end{lem}

Now suppose, alternatively, that we have a quadrangulated surface $(\Sigma,A)$ and wish to draw sutures $\Gamma$ on it. In general each decomposing arc $a$ of the quadrangulation intersects $\Gamma$ an odd number of times; the simplest way situation is if $|a \cap \Gamma| = 1$. Then on each square of the quadrangulation, the only possible nontrivial sutures are the standard positive or negative sutures.
\begin{defn}
A sutured quadrangulated surface $(\Sigma, \Gamma, A)$ on which each square has standard positive or negative sutures is called \emph{basic}.
\end{defn}

As there are $I(\Sigma,V)$ squares in a quadrangulation of $(\Sigma,V)$, there are precisely $2^{I(\Sigma,V)}$ basic sets of sutures on $(\Sigma,A)$. In this sense we may say that a basic sutured quadrangulated surface $(\Sigma, \Gamma, A)$ holds $I(\Sigma,V)$ bits of information. Obviously the basic sutures form a small subset of the possible sets of sutures.

Each basic square has Euler class $\pm 1$ according to its sign; gluing them together preserves Euler class (lemma \ref{lem:gluing_sutures_preserves_euler}), so the Euler class is given by summing the signs of the basic sutures.

\begin{prop}
\label{prop:sutures_made_basic}
Let $(\Sigma, \Gamma, V)$ be a sutured occupied surface without vacua, with nonconfining sutures $\Gamma$. There exists a quadrangulation $A$ of $(\Sigma, V)$ for which $\Gamma$ is basic.
\end{prop}

(Note: this statement essentially appears in \cite{HKM08}.)

\begin{proof}
For general $(\Sigma, \Gamma, V)$ without vacua, with nonconfining $\Gamma$ and $(\Sigma, V)$ not a disjoint union of occupied squares, we will find a nontrivial decomposing arc $a$ which intersects $\Gamma$ in one point. Cutting along it gives another sutured occupied surface which is nonconfining and without vacua, and by proposition \ref{prop:build_quadrangulation} repeating this procedure eventually reduces to occupied squares on which $\Gamma$ is basic, giving the desired quadrangulation $A$.

First, suppose there is a suture $\gamma$ which is a closed curve. On either side of $\gamma$ lie components of $R_\pm$; as $\Gamma$ is nonconfining, both these regions intersect $\partial \Sigma$. We thus take $a$ running between vertices on $\partial \Sigma$, and intersecting $\gamma$ precisely once. Now $a$ is not boundary parallel; if it were, we could isotope it off $\gamma$. Hence this $a$ is a nontrivial decomposing arc.

Next suppose we have a non-boundary-parallel arc $\gamma$. We take $a$ running along $\gamma$, but intersecting it once. (See figure \ref{fig:arcs_from_sutures}.) Then $a$ is not boundary parallel, hence a nontrivial decomposition arc.

\begin{figure}
\begin{center}
\begin{tikzpicture}[
scale=2, 
suture/.style={thick, draw=red}, 
decomposition/.style={thick, draw=green!50!black}, 
boundary/.style={ultra thick}
]

\coordinate (tl) at (0,2);
\coordinate (suture top) at (1,2);
\coordinate [label = above:{$+$}] (arc top) at (1.5,2);
\coordinate (tr) at (2,2);
\coordinate (bl) at (0,0);
\coordinate [label = below:{$-$}] (arc bottom) at (0.5,0);
\coordinate (suture bottom) at(1,0);
\coordinate (br) at (2,0);

\draw [boundary] (bl) -- node [midway, below=10pt] {$\partial \Sigma$} (br);
\draw [boundary] (tl) -- node [midway, above=8pt] {$\partial \Sigma$} (tr);
\draw [suture] (suture bottom) -- node[near end, above left] {$\gamma$} (suture top);
\draw [decomposition] (arc bottom) to [bend left=30] (1,1) to [bend right=30] node[midway, right] {$a$} (arc top);
\draw (0,1) node {$-$};
\draw (2,1) node {$+$};

\foreach \point in {suture top, arc top, arc bottom, suture bottom}
\fill [black] (\point) circle (1pt);

\end{tikzpicture}
\caption{Decomposing arcs found near sutures.}
\label{fig:arcs_from_sutures}
\end{center}
\end{figure}
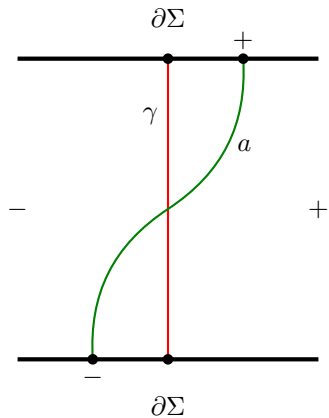

So, we may now assume every suture is boundary-parallel. If $(\Sigma,V)$ has more than one boundary component, we can find an arc running from one boundary component to another, which does not intersect $\Gamma$, and perturb it to intersect $\Gamma$ once. If $(\Sigma,V)$ has genus, we can find an arc running around the handle which does not intersect $\Gamma$, and again perturb for one intersection point.

Thus we may assume $(\Sigma, \Gamma, V)$ is a disc. We may then easily find nontrivial decomposing arcs which intersect $\Gamma$ once.
\end{proof}

The converse is also true. Nonconfining sutures imply the existence of a basic quadrangulation, and the existence of a basic quadrangulation implies nonconfining sutures.
\begin{prop}
Let $(\Sigma, \Gamma, A)$ be a basic sutured quadrangulated surface. Then $\Gamma$ is nonconfining.
\end{prop}

\begin{proof}
Suppose $\Gamma$ were confining. Consider a confined region $R$ of $\Sigma \backslash \Gamma$. Then some arc $a$ of the quadrangulation must pass through $R$; else $R$ lies entirely in one square of the quadrangulation, contradicting $\Gamma$ being basic. Take a point of $a \cap R$; as $a$ must end at vertices of the occupied surface, proceeding from that point in either direction along $a$, we must intersect $\Gamma$. Thus $|a \cap \Gamma| \geq 2$, contradicting $\Gamma$ being basic.
\end{proof}

In particular, if $(\Sigma, \Gamma, A)$ is basic then $\Gamma$ is nontrivial.

\subsection{Bypass surgeries on quadrangulated sutured surfaces}

There are two types of bypass surgeries nicely adapted to the structure of a quadrangulation.

First, suppose we have a set of sutures $\Gamma$ on the quadrangulated background $(\Sigma, F, A)$. Each internal edge of $A$ intersects $\Gamma$ in an odd number of points. If an edge $a$ of $A$ intersects $\Gamma$ in $3$ or more points, consider a subinterval $c$ of $a$ joining $3$ consecutive points of $a \cap \Gamma$. Then $c$ is an attaching arc, a neighbourhood of $c$ is a bypass disc, and we can perform bypass surgery there. After doing so, and simplifying (isotoping) sutures so as to intersect $a$ efficiently, we have a quadrangulated sutured surface $(\Sigma, \Gamma', A)$ where $|\Gamma' \cap A| < |\Gamma \cap A|$. Continuing in this way, we can reduce to basic sutures.

Moreover, if $\Gamma$ is nontrivial, then since we assume arcs of $A$ to intersect $\Gamma$ efficiently, by proposition \ref{prop:efficient_surgery} the resulting sutures are always nontrivial.
\begin{lem}
\label{lem:sutures_to_basic}
Let $\Gamma$ be a nontrivial set of sutures on the quadrangulated background $(\Sigma,F,A)$. Then there exists a sequence of bypass surgeries (and simplifying isotopies of sutures) which, applied to $\Gamma$, gives a basic set of sutures. Each set of sutures obtained in the process is nontrivial.
\qed
\end{lem}

The second type of bypass surgery relevant to a quadrangulation is as follows. Suppose the sutured quadrangulated surface $(\Sigma, \Gamma, A)$ has sutures $\Gamma$ basic and nontrivial. Suppose there are two squares $\Sigma^\square_0$ and $\Sigma^\square_1$ which share an edge $a$ of the quadrangulation, such that $\Sigma^\square_0$ has negative sutures $\Gamma^-_0$ and $\Sigma^\square_1$ has positive sutures $\Gamma^+_1$. (The two squares may have other edges in common as well in $\Sigma$.) There is then an attaching arc $c$ joining the sutures in these two squares. Performing bypass surgery there, we obtain a new set of sutures which is basic, nontrivial, and now $\Sigma_0^\square$ has positive sutures and $\Sigma_1^\square$ has negative sutures. See figure \ref{fig:surgery_swapping_signs}. So we can use bypass surgery to swap the signs of sutures in two adjacent squares of a basic quadrangulated sutured surface.

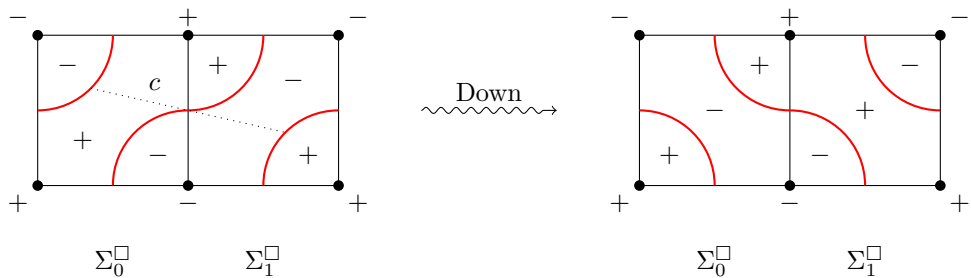
\begin{figure}
\begin{center}
\begin{tikzpicture}[
scale=2, 
fill = gray!10,
suture/.style={thick, draw=red} 
]

\coordinate [label = below left:{$+$}] (1bl) at (0,0);
\coordinate [label = below:{$-$}] (1bm) at (1,0);
\coordinate [label = below right:{$+$}] (1br) at (2,0);
\coordinate [label = above left:{$-$}] (1tl) at (0,1);
\coordinate [label = above:{$+$}] (1tm) at (1,1);
\coordinate [label = above right:{$-$}] (1tr) at (2,1);

\draw (1bl) -- (1br) -- (1tr) -- (1tl) -- cycle;
\draw (1tm) -- (1bm);
\draw [suture] (0,0.5) arc (-90:0:0.5);
\draw [suture] (0.5,0) arc (180:90:0.5) arc (-90:0:0.5);
\draw [suture] (2,0.5) arc (90:180:0.5);
\draw [dotted] ($ (0,1) + (-45:0.5) $) -- node [near start, above right] {$c$} ($ (2,0) + (135:0.5) $);

\draw [shorten >=1mm, -to, decorate, decoration={snake,amplitude=.4mm, segment length = 2mm, pre=moveto, pre length = 1mm, post length = 2mm}] (2.5,0.5) -- node [above] {Down} (3.5,0.5);

\coordinate [label = below left:{$+$}] (2bl) at (4,0);
\coordinate [label = below:{$-$}] (2bm) at (5,0);
\coordinate [label = below right:{$+$}] (2br) at (6,0);
\coordinate [label = above left:{$-$}] (2tl) at (4,1);
\coordinate [label = above:{$+$}] (2tm) at (5,1);
\coordinate [label = above right:{$-$}] (2tr) at (6,1);

\draw (2bl) -- (2br) -- (2tr) -- (2tl) -- cycle;
\draw (2tm) -- (2bm);
\draw [suture] (4,0.5) arc (90:0:0.5);
\draw [suture] (4.5,1) arc (180:270:0.5) arc (90:0:0.5);
\draw [suture] (6,0.5) arc (270:180:0.5);

\draw (0.2,0.8) node {$-$};
\draw (0.3,0.3) node {$+$};
\draw (0.8,0.2) node {$-$};
\draw (1.2,0.8) node {$+$};
\draw (1.7,0.7) node {$-$};
\draw (1.8,0.2) node {$+$};
\draw (4.2,0.2) node {$+$};
\draw (4.5,0.5) node {$-$};
\draw (4.8,0.8) node {$+$};
\draw (5.2,0.2) node {$-$};
\draw (5.5,0.5) node {$+$};
\draw (5.8,0.8) node {$-$};

\draw (0.5,-0.5) node {$\Sigma^\square_0$};
\draw (1.5,-0.5) node {$\Sigma^\square_1$};
\draw (4.5,-0.5) node {$\Sigma^\square_0$};
\draw (5.5,-0.5) node {$\Sigma^\square_1$};

\foreach \point in {1bl, 1bm, 1br, 1tl, 1tm, 1tr, 2bl, 2bm, 2br, 2tl, 2tm, 2tr}
\fill [black] (\point) circle (1pt);

\end{tikzpicture}

\end{center}
\caption{Bypass surgery swapping signs in adjacent squares.}
\label{fig:surgery_swapping_signs}
\end{figure}

In this way we can permute the signs of basic sutures within a connected component of $(\Sigma,V)$. Thus any two basic sets of sutures on a connected quadrangulated surface $(\Sigma, A)$ with the same Euler class are related by a sequence of bypass surgeries, through a sequence of basic sets of sutures. Combining this with the above lemma gives the following.

\begin{prop}
Any two nontrivial sets of sutures $\Gamma_0, \Gamma_1$ with the same Euler class on a connected background $(\Sigma, F)$ are related by a sequence of bypass surgeries (and isotopy). At every stage the set of sutures remains nontrivial.
\end{prop}

\begin{proof}
Take a quadrangulation. Each $\Gamma_i$ can be reduced to basic sutures by bypass surgeries and isotopies. These two basic sets of sutures are related by a sequence of bypass surgeries.
\end{proof}

\subsection{Properties of the Euler class}

\begin{prop}
The Euler class $e(\Gamma)$ of a nontrivial set of sutures $\Gamma$ on a sutured background $(\Sigma, F)$ satisfies
\[
- I(\Sigma,V) \leq e(\Gamma)\leq I(\Sigma, V), \quad e(\Gamma) \equiv I(\Sigma, V) \text{ mod } 2.
\]
\end{prop}
(Note this applies even if $\Gamma$ is confining.) This is a well-known result in contact geometry; it is essentially the Bennequin inequality (see \cite{Bennequin83}, also \cite{ElMartinet}).

\begin{proof}
For each component of $(\Sigma, F)$ which is a vacuum, we have $I(\Sigma,V) = 0$ and $e(\Gamma) = 0$. For those components which are not vacua, we take any quadrangulation $A$ of $(\Sigma, V)$ which intersects $\Gamma$ efficiently; then by lemma \ref{lem:sutures_to_basic} $\Gamma$ is related via bypass surgeries and isotopies to a nontrivial basic set of sutures for $A$. As discussed above, each square then has standard positive or negative suturing, with Euler class $\pm 1$, and summing these $\pm 1$'s over the $I(\Sigma,V)$ squares of the quadrangulation gives the result.
\end{proof}

In fact we can be a bit more precise and establish relations between Euler characteristics of $R_+$, $R_-$ and $\Sigma$. Write $\chi(\Sigma) = \chi$ and $e(\Gamma) = e$. Suppose we have a basic sutured quadrangulated surface, with $S_+$ positively sutured squares, and $S_-$ negatively sutured squares. Then we have
\[
S_+ + S_- = I(\Sigma,V) = N - \chi, \quad S_+ - S_- = e.
\]
Each $\pm$ square is cut by its sutures into two discs of $R_\pm$ and one disc of $R_\mp$. Letting $(\Sigma', \Gamma')$ be the disjoint union of the squares of the quadrangulation, we then have $\chi(R'_+) = 2S_+ + S_-$ and $\chi(R'_-) = S_+ + 2S_-$. As we then glue edges together to form $(\Sigma,\Gamma)$, along the $G(\Sigma,V) = N - 2\chi$ internal edges, we see $\chi(R_+)$ and $\chi(R_-)$ decrease by $1$ with each gluing. Thus
\begin{align*}
\chi(R_+) &= 2 S_+ + S_- - (N-2\chi) = (N-\chi + e) + \frac{1}{2} (N-\chi-e) - (N-2\chi) = \frac{1}{2}(N+\chi+e) \\
\chi(R_-) &= S_+ + 2S_- - (N-2\chi) = \frac{1}{2} (N-\chi+e) + (N-\chi-e) - (N-2\chi) = \frac{1}{2}(N+\chi-e).
\end{align*}
Corresponding equalities are also true on vacua, and even for nontrivial confining sutures, since they can be reduced to basic sutures by bypass surgeries (lemma \ref{lem:sutures_to_basic}). We obtain the following.
\begin{prop}
\label{prop:chi_relations}
Let $(\Sigma,\Gamma)$ be a sutured surface with $\Gamma$ nontrivial. Writing $\chi(\Sigma) = \chi$ and $e(\Gamma) = e$, we have
\[
2 \chi(R_+) = N + \chi + e, \quad 2 \chi(R_-) = N + \chi - e.
\]
\qed
\end{prop}

Combining the two propositions above gives immediately that 
\[
\chi(\Sigma) \leq \chi(R_\pm) \leq N,
\]
for any nontrivial sutures on a sutured surface. Moreover, equalities hold when $e = \pm I(\Sigma,V)$, or equivalently $\chi(R_\pm) = N$, or equivalently $\chi(R_\mp) = \chi(\Sigma)$. When such equalities hold, and the Euler class is extremal, we say $\Gamma$ is \emph{extremal}.

It is not difficult to see that nonconfining extremal sutures must consist entirely of boundary-parallel sutures connecting consecutive endpoints along each boundary component of $\Sigma$, cutting off $N$ discs of the same sign. (For $\chi(R_+) = N$ in this case, and the nonconfining property implies that $R_+$ has at most $N$ components, each with Euler characteristic $\leq 1$.) Such a $\Gamma$ is basic for \emph{any} quadrangulation.

However it is possible for confining sutures to be extremal; for instance, adding pairs of parallel sutures to any set of sutures does not change the Euler class.

\section{Decorated morphisms}
\label{sec:decorated_morphisms}

\subsection{Definitions}
\label{sec:defns_decorated_morphisms}

The following type of structure follows that of \cite{HKM08}, where one sutured manifold is mapped inside another and a contact structure is taken in the complementary region.
\begin{defn}
A pair $(\phi, \Gamma_c)$, where $\phi$ is an occupied surface morphism and $\Gamma_c$ a set of sutures on its complement, is called a \emph{decorated morphism}.
\end{defn}

Some examples are shown in figure \ref{fig:decorated_morphisms}. Note that when $\phi$ is surjective, its complement is empty, so no sutures are required; a surjective morphism is a decorated morphism. In particular, the identity morphism is a decorated morphism.

\begin{figure}
\begin{center}
\begin{tabular}{c}
\includegraphics[scale=0.3]{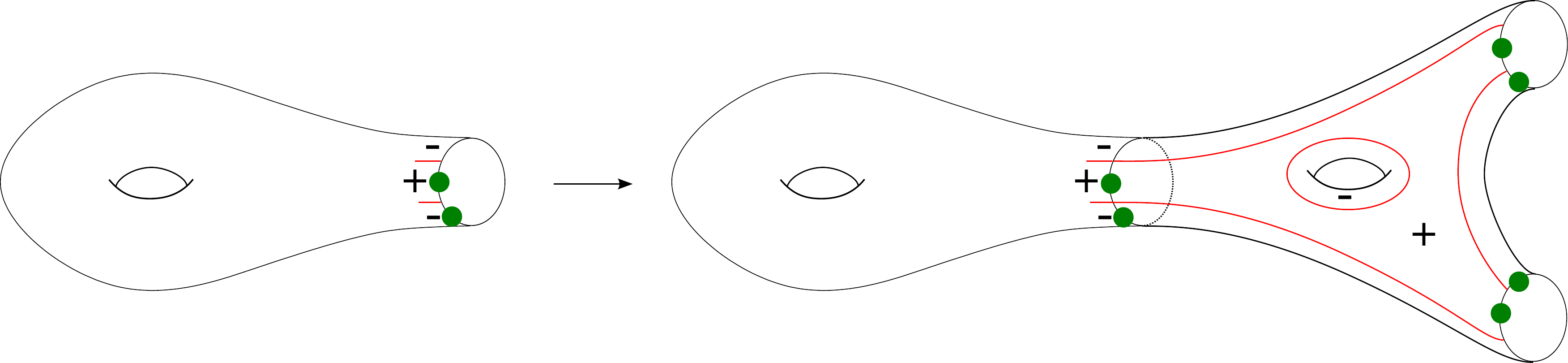}\\
\includegraphics[scale=0.4]{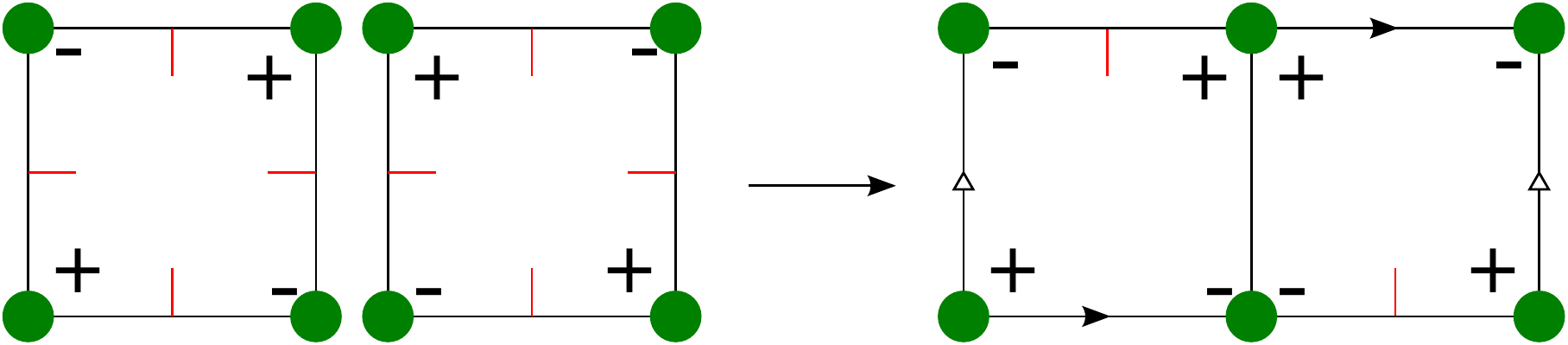}\\
\includegraphics[scale=0.4]{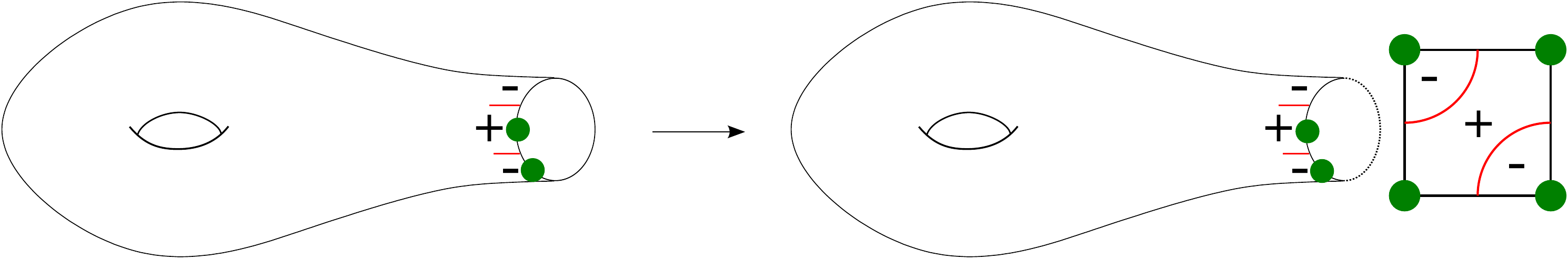}\\
\includegraphics[scale=0.4]{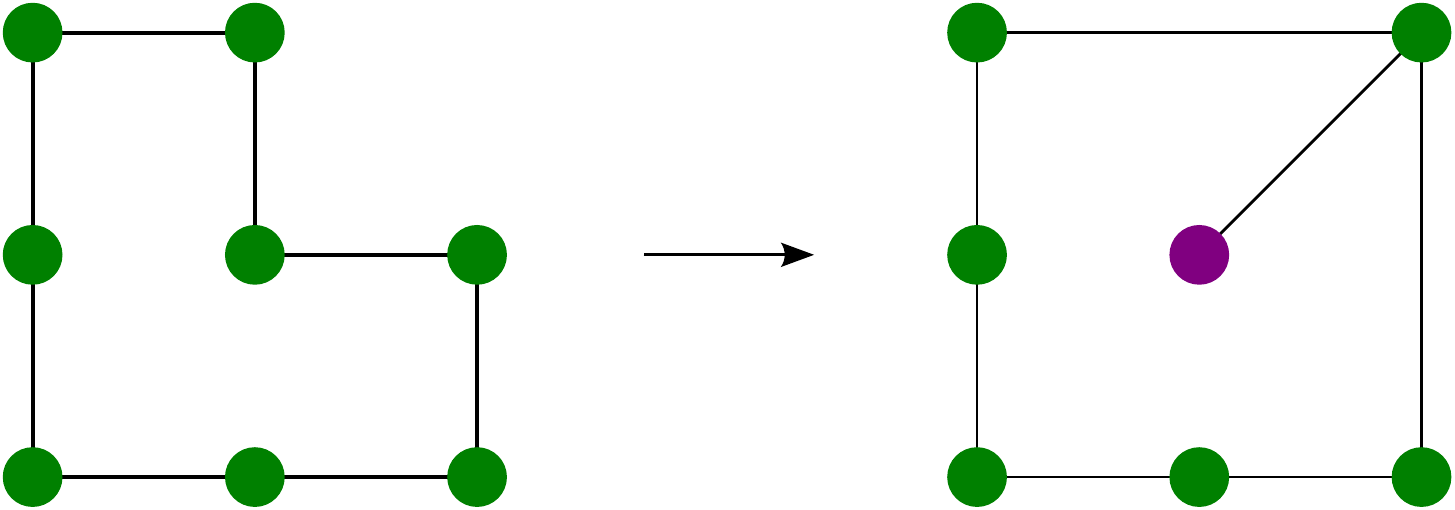}
\end{tabular}
\caption{Examples of decorated morphisms.}
\label{fig:decorated_morphisms}
\end{center}
\end{figure}

If $\phi_1: (\Sigma^1, V^1) \To (\Sigma^2, V^2)$ and $\phi_2: (\Sigma^2, V^2) \To (\Sigma^3, V^3)$ are morphisms, then the complement of $\phi_2 \circ \phi_1$ is the union of the complements of $\phi_1$ and $\phi_2$, viewed on $\Sigma^3$. These two complements intersect only along boundary edges. So if $\Gamma_1$ is a set of sutures on the complement of $\phi_1$, and $\Gamma_2$ is a set of sutures on the complement of $\phi_2$, then $\Gamma_1 \cup \Gamma_2$ is a set of sutures on the complement of $\phi_2 \circ \phi_1$. Thus, the composition of two decorated morphisms has the structure of a decorated morphism.
\begin{prop}
Occupied surfaces and decorated morphisms form a category, called the \emph{decorated occupied surface} category, denoted $\mathcal{DOS}$.
\qed
\end{prop}

As a decorated morphism $(\phi, \Gamma_c)$ includes a morphism $\phi$, notions for morphisms (such as isolating, vacuum-leaving, etc.) carry over immediately to decorated morphisms.

Note that if we have a set of sutures $\Gamma$ on $(\Sigma,V)$, and a decorated morphism consisting of $\phi: (\Sigma, V) \To (\Sigma', V')$ and $\Gamma_c$, then $\phi(\Gamma) \cup \Gamma_c$ forms a set of sutures on $(\Sigma',V')$.

This idea gives a notion of \emph{confining decorated morphism}.
\begin{defn}
A decorated morphism $(\phi, \Gamma_c)$ is \emph{confining} if for every set of sutures $\Gamma$ on $(\Sigma,V)$, the set of sutures $\Gamma' = \phi(\Gamma) \cup \Gamma_c$ obtained on $(\Sigma',V')$ is confining.
\end{defn}

It's clear that if $\Gamma_c$ is a confining set of sutures on the complement of $\phi$, then $(\phi, \Gamma_c)$ is confining, but the converse is not true.

The simple occupied surface morphisms all have analogues as decorated morphisms. Standard gluings, folds and zips are all surjective, so are automatically decorated morphisms. With a creation or annihilation, we have a choice of sutures on the new square.
\begin{defn}\
\label{def:decorated_creation_annihilation}
\begin{enumerate}
\item
A \emph{decorated creation} is a creation with a standard basic set of sutures on the created square.
\item
A \emph{decorated annihilation} is an annihilation with a standard basic set of sutures on the annihilator square.
\end{enumerate}
\end{defn}
We say a decorated creation or annihilation is \emph{positive} or \emph{negative} according to the sign of the standard sutures on the created or annihilator square. See figure \ref{fig:decorated_annihilations}.

\begin{figure}
\begin{center}

\begin{tikzpicture}[
scale=1.7, 
boundary/.style={ultra thick}, 
decomposition/.style={thick, draw=green!50!black}, 
vertex/.style={draw=green!50!black, fill=green!50!black},
suture/.style={thick, draw=red}
]

{
\fill [gray!10] (-0.5,0) -- (3.5,0) -- (3.5,1) -- (-0.5,1) -- cycle;
\fill [gray!30] (0,0) arc (-180:0:1.5) -- cycle;

\coordinate (p1) at (0,0);
\coordinate [label = above:{$v_-$}] (v-) at (1,0);
\coordinate [label = above:{$v_+$}] (v+) at (2,0);
\coordinate (p4) at (3,0);

\draw [boundary] (-0.5,0) -- (p1) arc (-180:0:1.5) -- (3.5,0);
\draw (p1) -- node [below left] {$e_1$} (v-) -- node [below left] {$e_2$} (v+) -- node [below left] {$e_3$} (p4);
\draw [suture] (1.5,0) arc (-180:0:0.5);
\draw [suture] (0.5,0) to [bend right = 45] (1,-0.75) to [bend left = 45] (1.5,-1.5);
\draw (-0.2,0.7) node {$\phi(\Sigma)$};
\draw (2,-0.75) node {$(\Sigma^\square, V^\square)$};

\fill [gray!10](4.5,0) -- (8.5,0) -- (8.5,1) -- (4.5,1) -- cycle;
\fill [gray!30](5,0) arc (-180:0:1.5) -- cycle;

\coordinate (q1) at (5,0);
\coordinate [label = above:{$v_-$}] (v-2) at (6,0);
\coordinate [label = above:{$v_+$}] (v+2) at (7,0);
\coordinate (q4) at (8,0);

\draw [boundary] (4.5,0) -- (q1) arc (-180:0:1.5) -- (8.5,0);
\draw (q1) -- node [below left] {$e_1$} (v-2) -- node [below left] {$e_2$} (v+2) -- node [below left] {$e_3$} (q4);
\draw [suture] (5.5,0) arc (-180:0:0.5);
\draw [suture] (7.5,0) to [bend left = 45] (7,-0.75) to [bend right = 45] (6.5,-1.5);
\draw (4.8,0.7) node {$\phi(\Sigma)$};
\draw (6,-0.75) node {$(\Sigma^\square, V^\square)$};

\foreach \point in {p1, v-, v+, p4, q1, v-2, v+2, q4}
\fill [vertex] (\point) circle (2pt);
}

\end{tikzpicture}

\caption{Decorated positive (left) and negative (right) annihilation.}
\label{fig:decorated_annihilations}
\end{center}
\end{figure}
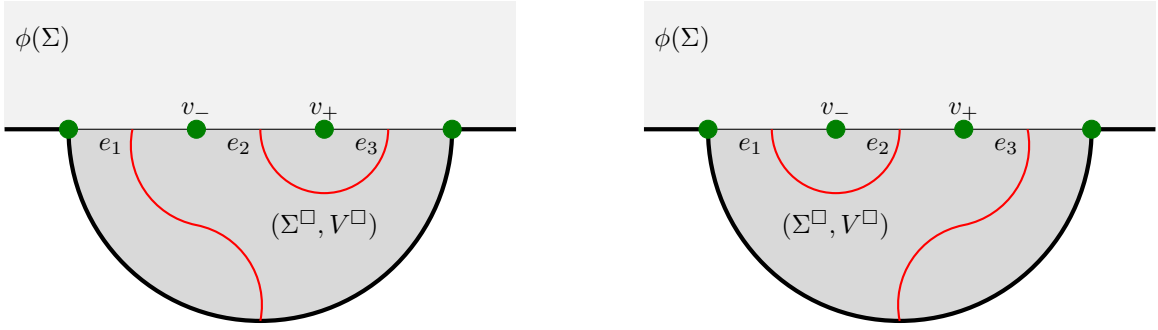

We will consider certain isotopies of decorated morphisms. We may isotope the map $\phi$ and complementary sutures $\Gamma_c$ continuously, through decorated morphisms. but we also allow sutures to be ``pushed away'' in a specific sense.
\begin{defn}
\label{def:decorated_morphism_isotopy}
An \emph{isotopy of decorated morphisms} is a family of decorated morphisms $(\phi_t, \Gamma_t)$, where $\phi_t: (\Sigma,V) \To (\Sigma',V')$ is a family of occupied surface morphisms, and $\Gamma_t$ is a family of sutures on the complement of $\phi_t$, which at all times either varies smoothly or admits the following types of singularities (see figure \ref{fig:decorated_surface_isotopy_singularity}):
\begin{enumerate}
\item
If $\phi_t$ leaves a vacuum $(\Sigma^\emptyset, V^\emptyset)$ with standard vacuum sutures and at least one edge of $\Sigma$-type, then one edge may be pushed on to the other and the suture removed.
\item
Conversely, if $\phi_t$ maps an edge of $(\Sigma,V)$ to an edge of $(\Sigma',V')$, or glues two edges of $(\Sigma,V)$ together, then, holding the vertices constant, one edge may be pushed off the other, leaving a vacuum with standard sutures.
\end{enumerate}
We say $(\phi_0, \Gamma_0)$ and $(\phi_1, \Gamma_1)$ are \emph{decorated-isotopic} morphisms or \emph{isotopic as decorated morphisms}.
\end{defn}

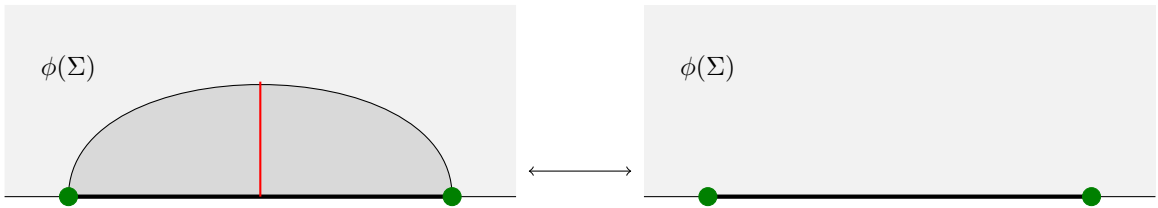
\begin{figure}
\begin{center}

\begin{tikzpicture}[
scale=1.7, 
boundary/.style={ultra thick}, 
vertex/.style={draw=green!50!black, fill=green!50!black},
suture/.style={thick, draw=red}
]

{
\fill [gray!10] (-0.5,0) -- (0,0) to [bend left=90] (3,0) -- (3.5,0) -- (3.5,1.5) -- (-0.5,1.5) -- cycle;
\fill [gray!30] (0,0) to [bend left=90] (3,0) -- cycle;

\draw (-0.5,0) -- (0,0) to [bend left=90] (3,0) -- (3.5,0);
\draw [boundary] (0,0) -- (3,0);
\draw [suture] (1.5,0.9) -- (1.5,0);
\draw (0,1) node {$\phi(\Sigma)$};

\fill [gray!10] (4.5,0) -- (8.5,0) -- (8.5,1.5) -- (4.5,1.5) -- cycle;

\draw (4.5,0) -- (5,0) (8,0) -- (8.5,0);
\draw [boundary] (5,0) -- (8,0);
\draw (5,1) node {$\phi(\Sigma)$};

\draw [<->] (3.6,0.2) -- (4.4,0.2);
}

\fill [vertex] (0,0) circle (2pt);
\fill [vertex] (3,0) circle (2pt);
\fill [vertex] (5,0) circle (2pt);
\fill [vertex] (8,0) circle (2pt);

\end{tikzpicture}

\caption{Singularities in an isotopy of decorated surface morphisms.}
\label{fig:decorated_surface_isotopy_singularity}
\end{center}
\end{figure}

If we have an isotopy $(\phi_t, \Gamma_t)$ of decorated morphisms $(\Sigma,V) \To (\Sigma',V')$, and $\Gamma$ is a set of sutures on $(\Sigma,V)$, then $\phi_t(\Gamma) \cup \Gamma_t$ forms an isotopy of sutures on $(\Sigma',V')$. The singularities permitted in such an isotopy help us to simplify decorated morphisms by collapsing vacua and removing them from the complement.

For instance, figure \ref{fig:decorated_annihilation_as_fold} demonstrates that a positive decorated annihilation is isotopic to a positive fold. The second step in that diagram involves two singularities; this is a decorated version of a slack square collapse, which also collapses sutures in a compatible way.

\begin{figure}
\begin{center}
\begin{tikzpicture}[
scale=1.3, 
fill = gray!10,
vertex/.style={draw=green!50!black, fill=green!50!black},
suture/.style={thick, draw=red},
boundary/.style={ultra thick} ]

\fill (0,0) -- (3,0) -- (3,-1) -- (0,-1) -- cycle;
\fill [gray!30] (0.5,0) arc (180:0:0.75) -- cycle;

\coordinate (s) at (0,0);
\coordinate [label = below:{$+$}] (t) at (0.5,0);
\coordinate [label = below:{$-$}] (u) at (1,0);
\coordinate [label = below:{$+$}] (v) at (1.5,0);
\coordinate [label = below:{$-$}] (w) at (2,0);
\coordinate (x) at (2.5,0);
\coordinate (y) at (3,0);

\draw [boundary] (s) -- (y);
\draw (t) arc (180:0:0.75);
\draw [suture] (1.25,0) arc (180:0:0.25);
\draw [suture] (0.75,0) to [bend left=45] (1,0.375) to [bend right=45] (1.25,0.75);

\draw [shorten >=1mm, -to, decorate, decoration={snake,amplitude=.4mm, segment length = 2mm, pre=moveto, pre length = 1mm, post length = 2mm}] (3.5,0) -- (4.5,0);

\fill (5,0) -- (7.5,0) -- (7.5,-1) -- (5,-1) -- cycle;
\fill [gray!30] (5.5,0) -- (6,0) -- (6.25,-0.75) -- (6.5,0) to [bend right=45] (5.5,0);

\coordinate (s2) at (5,0);
\coordinate [label = below:{$+$}] (t2) at (5.5,0);
\coordinate [label = below :{$-$}] (u2) at (6,0);
\coordinate [label = below:{$+$}] (v2) at (6.25,-0.75);
\coordinate [label = below :{$-$}] (w2) at (6.5,0);
\coordinate (x2) at (7,0);
\coordinate (y2) at (7.5,0);

\draw [boundary] (s2) -- (t2);
\draw [boundary] (w2) -- (y2);
\draw (t2) -- (u2) -- (v2) -- (w2) to [bend right=45] (t2);
\draw [suture] (6.125,-0.375) -- (6.375,-0.375);
\draw [suture] (5.75,0) -- (5.75,0.15);

\draw [shorten >=1mm, -to, decorate, decoration={snake,amplitude=.4mm, segment length = 2mm, pre=moveto, pre length = 1mm, post length = 2mm}] (8,0) -- (9,0);

\fill (9.5,0) -- (11.5,0) -- (11.5,-1) -- (9.5,-1) -- cycle;

\coordinate (s3) at (9.5,0);
\coordinate [label = above:{$+$}] (t3) at (10,0);
\coordinate [label = above:{$-$}] (u3) at (10.5,0);
\coordinate [label = below:{$+$}] (v3) at (10.5,-0.5);
\coordinate (w3) at (10.5,0);
\coordinate (x3) at (11,0);
\coordinate (y3) at (11.5,0);

\draw [boundary] (s3) -- (y3);
\draw (u3) -- (v3);

\foreach \point in {s,t,u,v,w,x,y,s2,t2,u2,v2,w2,x2,y2,x3,s3,t3,u3,v3,w3,x3,y3}
\fill [vertex] (\point) circle (2pt);

\end{tikzpicture}

\end{center}
\caption{A positive decorated annihilation is isotopic to a positive fold.}
\label{fig:decorated_annihilation_as_fold}
\end{figure}
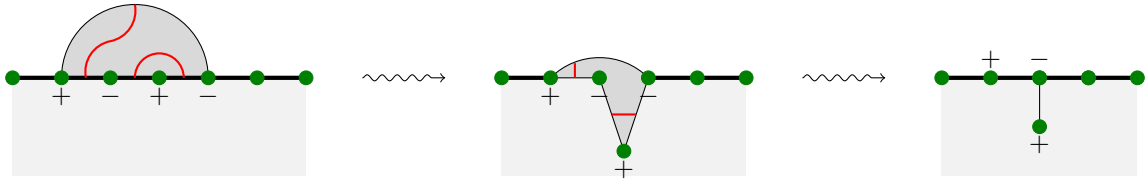

Similarly, whenever there is a suture $\gamma$ in a decorated morphism $(\phi, \Gamma_c)$ which connects two adjacent edges $e_1, e_2$ of $\Sigma$-type by an innermost boundary-parallel suture, we may simplify $(\phi, \Gamma_c)$ via an isotopy which folds $e_1, e_2$ together. See figure \ref{fig:simplifying_decorated_innermost}.

\begin{figure}
\begin{center}

\begin{tikzpicture}[
scale=1.7, 
fill = gray!10,
vertex/.style={draw=green!50!black, fill=green!50!black},
boundary/.style={ultra thick}, 
suture/.style={thick, draw=red}
]

{
\fill (0,0) -- (3,0) -- (3,-1) -- (0,-1) -- cycle;
\fill [gray!30] (0.5,0) -- (2.5,0) arc (0:180:1);

\coordinate (p1) at (0.5,0);
\coordinate (p2) at (1.5,0);
\coordinate (p3) at (2.5,0);

\draw (0,0) -- (p1) -- node [below] {$e_1$} (p2) -- node [below] {$e_2$} (p3) -- (3,0);
\draw [suture] (1,0) arc (180:0:0.5);
\draw (1.5,-0.5) node {$\phi(\Sigma)$};
\draw (0.5,1) node {$(\Sigma^c,V^c)$};

\fill (5,0) -- (8,0) -- (8,-1) -- (5,-1) -- cycle;

\coordinate (q1) at (6.5,0);
\coordinate (q2) at (6.5,-0.6);
\coordinate (s) at (6.5,-0.3);

\draw (5,0) -- (8,0);
\draw (q1) -- node [left] {$e_1$} node [right] {$e_2$} (q2);
\draw (5.5,-0.5) node {$\phi_1(\Sigma)$};

\draw [shorten >=1mm, -to, decorate, decoration={snake,amplitude=.4mm, segment length = 2mm, pre=moveto, pre length = 1mm, post length = 2mm}]
(3.5,0.2) -- (4.5,0.2);

\foreach \point in {p1, p2, p3, q1, q2}
\fill [vertex] (\point) circle (2pt);
\fill [red] (s) circle (2pt);
}

\end{tikzpicture}

\caption{Simplifying a decorated morphism by folding up an innermost boundary-parallel suture.}
\label{fig:simplifying_decorated_innermost}
\end{center}
\end{figure}
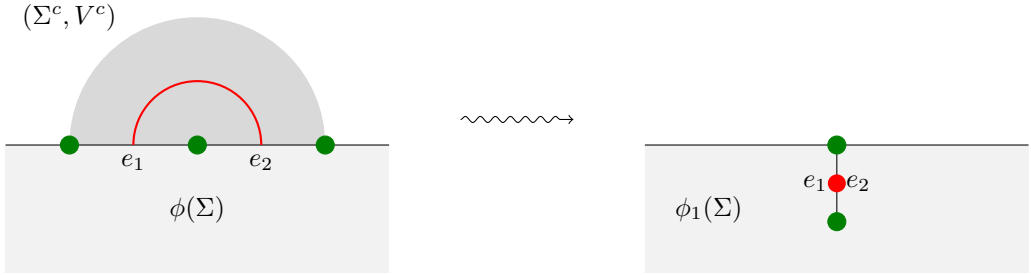

\subsection{Building decorated morphisms}
\label{sec:building_decorated_morphisms}

We now show how decorated morphisms can be built from elementary decorated morphisms.

Suppose we have an occupied surface morphism $\phi: (\Sigma,V) \To (\Sigma',V')$, and sutures $\Gamma_c$ forming a non-confining decorated surface morphism. We proceed as in the proof of proposition \ref{prop:morphism_composition_of_atomics}. 

First, if any boundary edges of $(\Sigma,V)$ are identified under $\phi$, this is expressed as a composition of gluings, folds and zips; and then we may assume $\phi$ does not identify boundary edges.

As $(\phi, \Gamma_c)$ is non-confining, $\Gamma_c$ is a non-confining set of sutures on the complement $(\Sigma^c, V^c)$; so the sutures on vacuum components of $(\Sigma^c, V^c)$ are standard, and on other components we may take a basic quadrangulation. We consider components $(\Sigma^T, V^T)$ of $(\Sigma^c, V^c)$ separately, and construct them one by one. At each stage, we have a previously existing surface $(\Sigma_p, V_p)$ constructed from $(\Sigma,V)$ by previous creations, gluings, folds and zips; to this we adjoin $(\Sigma^T, V^T)$ by more creations, gluings, folds and zips.

Suppose first that $(\Sigma^T, V^T)$ is a vacuum. If both edges are $\Sigma'$-type, the vacuum is a component of $(\Sigma',V')$; we can create a square (of either sign) and then, with a fold of the opposite sign, reduce it to a standard vacuum. If one edge is $\Sigma$- and one $\Sigma'$-type, we create a square (of either sign), attach the $\Sigma$-type edge with a standard gluing, and then perform a fold (of the opposite sign) to reduce to the vacuum. If both edges are $\Sigma$-type, then it is an isolated component, and we create the square, attach one edge with a standard gluing, perform a fold (of the opposite sign), and zip up the final edge.

(These vacuum cases are however avoided in certain circumstances: if $(\Sigma',V')$ is without vacua then the first case does not arise; and by isotopy of $(\phi, \Gamma_c)$ we may avoid vacua with $\Sigma$-type edges.)

We may now assume $(\Sigma^T, V^T)$ is not a vacuum, and as $\Gamma_c$ is nonconfining we may take a basic quadrangulation (proposition \ref{prop:sutures_made_basic}).

As in section \ref{sec:building_morphisms}, we use a lemma to guarantee a way to glue on squares without having to glue along all four edges. The statement is very similar to lemma \ref{lem:building_non-isolating_morphisms}, and the proof is identical; successively remove squares with $\Sigma'$-type edges, and then reverse the order when gluing them on.
\begin{lem}
Let $(\phi, \Gamma_c)$ be a non-isolating non-confining decorated morphism which leaves no vacua, and let $A^c$ be a quadrangulation of $(\Sigma^c, V^c)$ such that $\Gamma_c$ is basic. Then there exists an ordering of the squares $(\Sigma^\square_1, V^\square_1), \ldots, (\Sigma^\square_m, V^\square_m)$ of $A^c$, such that successively gluing the squares on to $(\Sigma, V)$ in order and constructing $(\Sigma', V')$, no square is ever glued in along all four of its edges.
\qed
\end{lem}

Return to our component $(\Sigma^T,V^T)$ of the complement of $(\phi, \Gamma_c)$. If $(\Sigma^T, V^T)$ is a non-isolating component, then using the lemma, we may successively create squares and glue them in along at most three of their edges. As in the proofs of propositions \ref{prop:morphism_composition_of_atomics} and \ref{prop:quadrangulated_surface_construction}, each new square has three troublesome edges, or two consecutive troublesome edges, or no troublesome edges, and we can attach the edges with standard gluings and folds.

Finally suppose $(\Sigma^T, V^T)$ is an isolating component. Again as in the proof of \ref{prop:morphism_composition_of_atomics}, split apart $(\Sigma_p, V_p)$ and $(\Sigma^T, V^T)$ along a common edge $e$. Then the component is non-isolating and we apply the above argument, before re-attaching the split-open edges with a zip.

Thus we have:
\begin{prop}
\label{prop:decorated_morphism_composition}
A non-confining decorated morphism is a composition of decorated creations, gluings, folds and zips.
\qed
\end{prop}

Keeping track of quadrangulations, we obtain a more detailed description. Suppose we start with a quadrangulation $A$ of $(\Sigma,V)$, in addition to the basic quadrangulation $A^c$ of $(\Sigma^c, \Gamma_c)$. (Note $A \cup A^c$ gives a slack quadrangulation of $(\Sigma',V')$.) In the above, we have expressed $\phi: (\Sigma,V) \To (\Sigma',V')$ as a composition
\[
(\Sigma,V) = (\Sigma_0, V_0) \stackrel{\phi_1}{\To} (\Sigma_1, V_1) \stackrel{\phi_2}{\To} (\Sigma_2, V_2) \stackrel{\phi_3}{\To} \cdots \stackrel{\phi_{n-1}}{\To} (\Sigma_{n-1}, V_{n-1}) \stackrel{\phi_n}{\To} (\Sigma_n, V_n) = (\Sigma', V'),
\]
where each $\phi_i$ is a decorated creation, fold, gluing or zip. Given a quadrangulation $A_{i-1}$ on $(\Sigma_{i-1}, V_{i-1})$, then considering $\phi_i$, we can obtain a quadrangulation $A_i$ of $(\Sigma_i, V_i)$ in a natural way, so as to obtain quadrangulations at every step and eventually a quadrangulation on $(\Sigma',V')$:
\begin{itemize}
\item
If $\phi_i$ is a decorated creation, then adjoining the created square to $A_{i-1}$ naturally gives a quadrangulation $A_i$ of $(\Sigma_i)$.
\item
If $\phi_i$ is a standard gluing, then the quadrangulation $A_{i-1}$ naturally gives a quadrangulation $A_i$ with the same squares.
\item
If $\phi_i$ is a fold, then $A_{i-1}$ naturally gives a slack quadrangulation of $(\Sigma_i, V_i)$ with one internal vertex; proposition \ref{prop:slack_square_collapse_exists} guarantees the existence of a slack square collapse which then produces a quadrangulation $A_i$.
\item
If $\phi_i$ is a zip, then $A_{i-1}$ naturally gives a slack quadrangulation of $(\Sigma_i, V_i)$ with two internal vertices; proposition \ref{prop:slack_square_collapse_exists} guarantees two slack square collapses to produce a quadrangulation $A_i$.
\end{itemize}

Note that in the last two cases, the choice of slack square collapse may not be unique. Moreover, performing a slack square collapse may disrupt the existing sutures from being basic with respect to the quadrangulation.

\section{Representing sutured and occupied surfaces}
\label{sec:SQFT}

\subsection{Definitions and basic notions}
\label{sec:SQFT_defn}

We should now like to represent the geometric data of occupied/background structures, sutures, and quadrangulations, by algebraic data, along the following lines. Here we use $\Z_2$ as our base ring and vector spaces over it, but all of our constructions can be done over more general rings and modules.
\begin{itemize}
\item
To an occupied surface $(\Sigma, V)$, we associate an $\Z_2$-vector space $\V(\Sigma,V)$.
\item
To a decorated morphism $(\phi, \Gamma_c) : (\Sigma, V) \To (\Sigma', V')$ we associate an linear map $\V(\Sigma,V) \To \V(\Sigma',V')$. 
\item
To a quadrangulation $A$ of $(\Sigma, V)$ with squares $(\Sigma^\square_i, V^\square_i)$, we associate a \emph{tensor decomposition}
\[
V(\Sigma,V) = \bigotimes_i \V(\Sigma^\square_i, V^\square_i).
\]
\item
To a set of sutures $\Gamma$ on $(\Sigma, V)$, we associate an \emph{element} $c(\Gamma) \in V(\Sigma,F)$. If the sutures are \emph{basic}, we associate a basis element.
\end{itemize}
Over more general base rings, one suture element will not suffice; for instance, over $\Z$ suture elements must be taken up to sign. See \cite{HKM08, Me10_Sutured_TQFT, Massot09} for details.

The first two requirements above, together with requiring that they respect compositions and identity, say that we want a functor from the category $\mathcal{DOS}$ of decorated occupied surfaces to the category $\Z_2 \mathcal{VS}$ of $\Z_2$-vector spaces. This functor must also respect the tensor structures coming from quadrangulations, and treat suture elements in a natural way. We can make this precise.

\begin{defn}
\label{def:SQFT}
A \emph{sutured quadrangulated field theory} is a collection $(\mathcal{D},c)$ where
\begin{enumerate}
\item
$\mathcal{D}$ is a functor $\mathcal{DOS} \To \Z_2 \mathcal{VS}$, associating to an occupied surface $(\Sigma,V)$ a $\Z_2$-vector space $\V(\Sigma,V)$, and to a decorated morphism $(\phi, \Gamma_c)$ a linear map $\D_{\phi,\Gamma_c} : \V(\Sigma,V) \To \V(\Sigma',V')$,
\item
$c$ assigns to each (isotopy class of) sutures $\Gamma$ on $(\Sigma,V)$ an element $c(\Gamma) \in \V(\Sigma,V)$,
\end{enumerate}
satisfying the following conditions.
\begin{enumerate}
\item
Quadrangulations give tensor decompositions. For any (isotopy class of) quadrangulation $A$ of $(\Sigma,V)$ with squares $(\Sigma^\square_i, V^\square_i)$, 
\[
\V(\Sigma,V) = \bigotimes_i \V(\Sigma^\square_i, V^\square_i).
\]
(This includes the ``null quadrangulation'' on any occupied vacuum component and null tensor decomposition, $\V(\Sigma^\emptyset, V^\emptyset) = \Z_2$.)
If $\Gamma$ is a basic set of sutures on $(\Sigma,V)$, restricting to $\Gamma_i$ on $(\Sigma^\square_i, V^\square_i)$, then
\[
c(\Gamma) = \otimes_i c(\Gamma_i).
\]
\item
Suture elements are respected. For any sutures $\Gamma$ on $(\Sigma, V)$,
\[
\D_{\phi, \Gamma_c} (c(\Gamma)) = c(\Gamma \cup \Gamma_c).
\]
\item
Basic means basic. The suture elements $c(\Gamma_+), c(\Gamma_-)$ of the standard positive and negative sutures $\Gamma_+, \Gamma_-$ on the occupied square $(\Sigma^\square, V^\square)$ give a basis for $\V(\Sigma^\square, V^\square)$.
\item
Euler class gives Euler grading. If $\Gamma$ is a set of sutures with Euler class $e$, then $c(\Gamma)$ has Euler grading $e$.
\end{enumerate}
\end{defn}
We will define the Euler grading in section \ref{sec:grading} below. Note the conditions ``basic means basic'' and ``quadrangulations give tensor decompositions'' imply that the $c(\Gamma)$, over all basic $\Gamma$, form a basis of $\V(\Sigma,V)$, for any quadrangulated $(\Sigma,V)$.

We note some simple properties of SQFT. If $\phi: (\Sigma,V) \To (\Sigma', V')$ is an occupied surface homeomorphism it composes with its inverse to give the identity, and functors preserve identity, so $\D_{\phi, \emptyset}$ is an isomorphism. 

Start from the simplest occupied surface. The occupied vacuum $(\Sigma^\emptyset, V^\emptyset)$ has index $0$ and $\V(\Sigma^\emptyset, V^\emptyset) = \Z_2$. The standard vacuum sutures $\Gamma_\emptyset$ have a suture element which is not zero; for instance, gluing it on to a square with standard basic sutures again yields basic sutures. Thus $c(\Gamma_\emptyset) = 1 \in \Z_2$.

Turning to the occupied square, we can write $\V(\Sigma^\square, V^\square)$ as ${\bf V}$, so ${\bf V}$ has dimension $2$ with basis $c(\Gamma_-), c(\Gamma_+)$. In the ``digital'' notation of the introduction, $c(\Gamma_-) = \0$, $c(\Gamma_+) = \1$; we will also write $c(\Gamma_-) = v_-$, $c(\Gamma_+) = v_+$.

Since any occupied surface $(\Sigma,V)$ decomposes into vacuum components and a quadrangulation with $I(\Sigma,V)$ squares, $\V(\Sigma,V) = {\bf V}^{\otimes I(\Sigma,V)}$ has dimension $2^n$, with basis $\0 \otimes \cdots \otimes \0,$ $\0 \otimes \cdots \otimes \1$, $\ldots$, $\1 \otimes \cdots \otimes \1$.

Consider a standard gluing morphism $\phi: (\Sigma,V) \To (\Sigma',V')$, which being surjective is a decorated morphism. A quadrangulation $A$ of $(\Sigma,V)$ with squares $(\Sigma^\square_i, V^\square_i)$ naturally gives a quadrangulation $A'$ of $(\Sigma',V')$ with squares $(\Sigma^{\square'}_i, V^{\square'}_i)$, and the two sets of of squares are naturally bijective. Moreover $\phi$ sends basic sutures to basic sutures preserving signs.
\begin{lem}
\label{lem:standard_gluing_isomorphism}
If $\phi$ is a standard gluing then $\mathcal{D}_{\phi, \emptyset}$ is the identity
\[
\V(\Sigma,V) = \bigotimes_i \V(\Sigma^\square_i, V^\square_i) = {\bf V}^{\otimes n} \To {\bf V}^{\otimes n} = \bigotimes_i \V(\Sigma^\square_i, V^\square_i) = \V(\Sigma',V')
\]
where $n = I(\Sigma,V) = I(\Sigma',V')$.
\qed
\end{lem}

Also note that if we have an isotopy of decorated morphisms $(\phi_t, \Gamma_t): (\Sigma, V) \To (\Sigma', V')$, together with a set of sutures $\Gamma$ on $(\Sigma,V)$, then we obtain an isotopy of sutures on $(\Sigma',V')$ given by $\phi_t(\Gamma) \cup \Gamma_t$. As a suture element depends only on the isotopy class of sutures, running $\Gamma$ over a set of basic sutures on $(\Sigma,V)$ gives the following.
\begin{lem}
\label{lem:decorated_isotopy_equal_map}
If $(\phi, \Gamma_c)$ and $(\phi', \Gamma'_c)$ are decorated-isotopic then $\D_{\phi, \Gamma_c} = \D_{\phi', \Gamma'_c}$.
\qed
\end{lem}

We pause to describe an alternative, equivalent formulation of SQFT, using \emph{surjective} morphisms instead of decorated morphisms. Surjective occupied surface morphisms are a subset of decorated morphisms, and form a subcategory $\mathcal{SOS}$ of $\mathcal{DOS}$. We can restrict the above definition of SQFT to this subcategory in a natural way to obtain a definition of ``surjective SQFT'', which is in fact a simpler formulation, since there are never any complementary sutures to consider. 

Furthermore, it is possible to recover the full structure of SQFT from its restriction to surjective morphisms. To see why, take a decorated surface morphism given by $\phi: (\Sigma,V) \To (\Sigma',V')$ and sutures $\Gamma_c$ on the complement $(\Sigma^c, V^c)$. There is a surjective occupied surface morphism $\Phi: (\Sigma,V) \sqcup (\Sigma^c,V^c) \To (\Sigma',V')$ which simply glues the $(\Sigma^c,V^c)$ on to $(\Sigma,V)$ as specified by $\phi$. From the surjective SQFT map map $\D_\Phi: \V(\Sigma,V) \otimes \V(\Sigma^c, V^c) \To \V(\Sigma',V')$, we can then define a linear map $\D_{\phi, \Gamma_c}: \V(\Sigma,V) \To \V(\Sigma',V')$ for $(\phi, \Gamma_c)$ via
\[
\D_{\phi, \Gamma_c} (c(\Gamma)) = \D_{\Phi}  \left( c(\Gamma \sqcup \Gamma_c) \right) = \D_{\Phi} \left( c(\Gamma) \otimes c(\Gamma_c) \right)
\]
for sutures $\Gamma$ on $(\Sigma,V)$. It is not difficult to prove that the structures of SQFT and surjective SQFT are equivalent in this way.

For the rest of this paper, however, we stick with SQFT as defined above.

\subsection{Sutured Floer homology}

We now construct an example of SQFT from sutured Floer homology, using the TQFT structure introduced by Honda--Kazez--Mati\'{c} in \cite{HKM08}, proving theorem \ref{thm:SFH_gives_SQFT}.

We recall some facts about $SFH$ and refer to the papers cited for details. As defined by Juh\'{a}sz in \cite{Ju06}, $SFH$ associates to a balanced sutured 3-manifold $(M,\Gamma)$ a $\Z_2$-vector space $SFH(M,\Gamma)$ (much more general coefficients can be used), with various properties. To a contact structure $\xi$ on $M$ with convex boundary and dividing set $\Gamma$ is associated a \emph{contact element} $c(\xi) \in SFH(-M,-\Gamma)$ \cite{HKMContClass, OSContact}. Honda--Kazez--Mati\'{c} in \cite{HKM08} showed that an inclusion of a sutured manifold $(M,\Gamma)$ into the interior of another sutured manifold $(M',\Gamma')$, together with a contact structure $\xi'$ on $(M' \backslash \text{Int} \; M, \Gamma \cup \Gamma')$, induces a natural map
\[
\Phi_{\xi'} \; : \; SFH(-M,-\Gamma) \To SFH(-M',-\Gamma') \otimes {\bf V}^{\otimes m},
\]
where $m$ is the number of isolated components of $M' \backslash \text{Int} \; M$ and ${\bf V}$ is a $2$-dimensional vector space over $\Z_2$. This map preserves contact elements: if $\xi$ is a contact structure on $(M,\Gamma)$ then
\[
\Phi_{\xi'} (c(\xi)) = c(\xi \cup \xi') \otimes x^{\otimes m}
\]
where $x$ is a particular element of ${\bf V}$.

We now proceed to construct an SQFT in several steps.

\begin{enumerate}
\item
First, to an occupied surface $(\Sigma,V)$, with corresponding sutured background $(\Sigma,F)$, we assign $\V(\Sigma,V) = SFH(-\Sigma \times S^1, -V \times S^1)$.
\item
Next, to a set of sutures we assign a suture element. A set of sutures $\Gamma$ on $(\Sigma,F)$ describes $\Sigma$ as a convex surface and hence describes an $I$-invariant contact structure on $\Sigma \times I$. Gluing $\Sigma \times \{0\}$ to $\Sigma \times \{1\}$ gives a contact structure on $\Sigma \times S^1$; the boundary $\partial \Sigma \times S^1$ can be taken to be convex with dividing set $V \times S^1$. We then obtain the contact invariant $c(\xi) \in SFH(-\Sigma \times S^1, -V \times S^1) = \V(\Sigma,V)$, which over $\Z_2$-coefficients is a well-defined element. (This is not true over more complicated coefficients.) So we set the suture element $c(\Gamma)$ equal to the contact invariant.
\item
A quadrangulation $A$ of $(\Sigma,V)$ gives a tensor decomposition of $\V(\Sigma,V)$. Let $A$ cut $(\Sigma,V)$ into squares $(\Sigma^\square_i, V^\square_i)$; then $A \times S^1$ cuts $(\Sigma \times S^1, V \times S^1)$ into sutured manifolds $(\Sigma^\square_i \times S^1, V^\square_i \times S^1)$. As explained in \cite[lemma 7.2]{HKM08}, using results of Juh\'{a}sz on surface decompositions \cite{Ju08}, decomposing $\Sigma \times S^1$ along the annuli $A \times S^1$ gives a tensor decomposition over the squares of the quadrangulation
\[
\V(\Sigma,V) = SFH(-\Sigma \times S^1, -V \times S^1) = \bigotimes_i SFH(-\Sigma^\square_i \times S^1, -V^\square_i \times S^1) = \bigotimes_i \V(\Sigma^\square_i, V^\square_i).
\]
The map which glues the squares together is an isomorphism, and thus if we have basic sutures $\Gamma$ restricting to $\Gamma_i$ on square $(\Sigma^\square_i, V^\square_i)$,
\[
c(\Gamma) = \otimes_i c(\Gamma_i).
\]
\item
The basic sets of sutures $\Gamma^+, \Gamma^-$ on the square $(\Sigma^\square, V^\square)$ give a basis $c(\Gamma^-), c(\Gamma^+)$ for $\V(\Sigma^\square,V^\square) = SFH(\Sigma^\square \times S^1, V \times S^1)$,  so (iv) is satisfied. This was discussed in \cite[lemma 7.2]{HKM08} and later in \cite{Me09Paper, Me10_Sutured_TQFT}. Thus each $\V(\Sigma,V)$ has a basis of basic suture elements.
\item
To a decorated morphism $(\phi, \Gamma_c) \; : \; (\Sigma,V) \To (\Sigma',V')$ we associate a map $\D_{\phi, \Gamma_c} \; : \; \V(\Sigma,V) \To \V(\Sigma',V')$. A decorated morphism can always be isotoped so that $\phi$ has an embedding into the interior of $(\Sigma',V')$: in such an isotopy (definition \ref{def:decorated_morphism_isotopy}) we are permitted to push apart edges which are glued together under $\phi$, and insert a standard set of sutures in the new complementary region; then we may also push apart vertices which are glued together. So there is an isotopic decorated morphism $(\phi^*, \Gamma^*_c)$ which is an embedding into the interior of $(\Sigma',V')$. From this we obtain an inclusion of $(\Sigma \times S^1, V \times S^1)$ into the interior of $(\Sigma' \times S^1, V' \times S^1)$, with a contact structure $\xi^*$ on  $(\Sigma' - \text{Int} \; \Sigma, (V \cup V') \times S^1)$ given by $\Gamma^*_c$. We then have the map
\[
\Phi_{\xi^*} \; :  \; \V(\Sigma,V) \To \V(\Sigma', V') \otimes {\bf V}^{\otimes m}
\]
where $m$ is the number of isolated components of $\phi$. Now $\Phi_{\xi^*}$ is natural with respect to contact elements, and so for any contact structure $\xi$ on $(\Sigma \times S^1, V \times S^1)$ we have $\Phi_{\xi^*} (c(\xi)) = c(\xi \cup \xi^*) \otimes x^{\otimes m}$. As $\V(\Sigma,V)$ has a basis of basic suture elements, hence also a basis of contact elements, we see that $\Phi_{\xi^*}$ has image lying in $\V(\Sigma', V') \otimes x^{\otimes m}$ and in fact $\Phi_{\xi^*} = \D_{\phi^*,\Gamma^*_c} \otimes x^{\otimes m}$, where $\D_{\phi^*, \Gamma^*_c}$ is a linear map $\V(\Sigma,V) \To \V(\Sigma',V')$. Now as isotopic decorated morphisms do not affect the isotopy class of sutures, this $\D_{\phi^*, \Gamma^*_c}$ does not depend on the particular choice $(\phi^*, \Gamma^*_c)$ of decorated morphism isotopic to $(\phi, \Gamma_c)$. Hence we may assign a well-defined map $\D_{\phi,\Gamma_c} =\D_{\phi^*, \Gamma^*_c}$ to $(\phi, \Gamma_c)$, and it respects suture elements,
\[
\D_{\phi, \Gamma_c} (c(\Gamma)) = c(\Gamma \cup \Gamma_c).
\]
\item
Theorem 6.1 of \cite{HKM08} establishes that the maps assigned respect identities. In particular, Honda--Kazez--Mati\'{c} prove that an inclusion $(\Sigma,V)$ into its own interior by shrinking a little near the boundary, with complement $\partial \Sigma \times I$, and with straight complementary sutures connecting points of $\partial \Sigma \times \{0\}$ to $\partial \Sigma \times \{1\}$, gives the identity map on $SFH$ (in the mod $2$ case; it is more complicated otherwise). This is a decorated morphism isotopic to the identity, and so our assignments send the identity morphism to the the identity on $\V(\Sigma,V)$.
\item
Theorem 6.2 of \cite{HKM08} establishes that the maps assigned respect composition. When inclusions of sutured manifolds, with complementary contact structures, are composed, the map on $SFH$ assigned to their composition is the composition of the $SFH$ maps assigned to the individual inclusions.
\item
It is also explained in \cite{Me09Paper, Me10_Sutured_TQFT} how Euler class corresponds to an Euler grading on $\V(\Sigma,V)$. In particular, a bypass relation holds, and so by the contact element of a set of sutures $\Gamma$ with Euler class $e$ is a sum of basic contact elements with Euler grading $e$. See also sections \ref{sec:grading} and \ref{sec:alg_bypasses} below for further details.
\end{enumerate}

We have now shown that the above assignments give a well-defined functor $\mathcal{DOS} \To \Z_2 \mathcal{VS}$, and a well-defined assignment of suture elements, satisfying the conditions of definition \ref{def:SQFT}. This completes the proof of theorem \ref{thm:SFH_gives_SQFT}.

\begin{cor}
A sutured quadrangulated field theory exists.
\qed
\end{cor}

\subsection{Grading}
\label{sec:grading}

The Euler class of sutures gives a grading on SQFT.

On the occupied square $(\Sigma^\square, V^\square)$, we have $\V(\Sigma^\square, V^\square) = {\bf V}$ with basis $c(\Gamma_-) = v_- = \0$ and $c(\Gamma_+) = v_+ = \1$. By grading the basis with integers, we obtain a grading on all tensor products ${\bf V}^{\otimes n}$. Let us define gradings as follows:
\begin{enumerate}
\item
$\0$-grading $n_\0$: $\0 = v_-$ has grading $1$ and $\1 = v_+$ has grading $0$.
\item
$\1$-grading $n_\1$: $\0 = v_-$ has grading $0$ and $\1 = v_+$ has grading $1$.
\item
Euler grading $e$: $\0 = v_-$ has grading $-1$ and $\1 = v_+$ has grading $1$.
\end{enumerate}
In \cite{Me10_Sutured_TQFT} these gradings were called $n_x, n_y$ and $e$ respectively. Obviously $e = n_\1 - n_\0$, so any two of these gradings determines the third. (Unfortunately, the ``digital'' analogy forces us to use $\0$ and $\1$, while the Euler class grading, which is topologically the natural one, forces $\0$ to have grading $-1$! We would prefer to take our binary ``digits'' to be $-$ and $+$, rather than $\0$ and $\1$, but this would be a little too unorthodox.)

These gradings naturally extend to any tensor product ${\bf V}^{\otimes n}$. A basis element $v_\pm \otimes v_\pm \otimes \cdots \otimes v_\pm$ has grading equal to the number of $v_-$'s, number of $v_+$'s, or their difference, according to $n_\0, n_\1$ and $e$ respectively.

Suppose now we have a quadrangulated surface with squares $(\Sigma^\square_i, V^\square_i)$. Letting $\Gamma^\pm_i$ denote the basic $\pm$ sutures on $(\Sigma^\square_i, V^\square_i)$, a basic set of sutures can be denoted by $\Gamma^{\bf e} = \cup_i \Gamma^{e_i}_i$, where ${\bf e} = (e_i)$ denotes the sign of the basic sutures on each square, $e_i = \pm 1$. Then $\Gamma^{\bf e}$ has Euler class $\sum_i e_i$, and suture element $c(\Gamma^{\bf e}) = \otimes_i v_{e_i}$, which has Euler grading $\sum_i e_i$.

Thus, without using the ``Euler class gives Euler grading'' axiom of SQFT, it follows that a basis element $c(\Gamma^{\bf e})$ has Euler grading equal to the Euler class of the basic sutures $\Gamma^{\bf e}$. We shall however need the Euler axiom in order to understand bypass surgeries, as we see next.

Note that any ${\bf V}^{\otimes n}$ decomposes as a direct sum of subspaces with Euler gradings from $-n$ to $n$. For a decorated morphism $(\phi, \Gamma_c)$ from $(\Sigma, V) \To (\Sigma',V')$, sutures $\Gamma$ on $(\Sigma,V)$ give sutures $\Gamma' = \phi(\Gamma) \cup \Gamma_c$ on $(\Sigma',V')$ with Euler class $e(\Gamma') = e(\Gamma) + e(\Gamma_c)$. So $(\phi, \Gamma_c)$ adjusts Euler class by $e(\Gamma_c)$, and hence $\D_{\phi,\Gamma_c}$ adjusts Euler grading by $e(\Gamma_c)$. In particular, $\D_{\phi, \Gamma_c}$ decomposes as a direct sum over Euler-graded summands.

\subsection{Algebra of bypass surgeries}
\label{sec:alg_bypasses}

Let $(\Sigma,V)$ be a hexagon, i.e. a disc with $6$ vertices. Take a quadrangulation with two squares, $(\Sigma^\square_0, V^\square_0)$ and $(\Sigma^\square_1, V^\square_1)$, as shown in figure \ref{fig:order_3_morphism_sutures}, and consider the morphism $\phi: (\Sigma,V) \To (\Sigma,V)$ which, if the hexagon is drawn symmetrically, rotates $120^\circ$ anticlockwise. As $\phi$ is a homeomorphism, the map $\D_{\phi, \emptyset}$ is an isomorphism $\V(\Sigma,V) \To \V(\Sigma,V)$. The vector space decomposes as
\[
\V(\Sigma,V) = \V(\Sigma^\square_0, V^\square_0) \otimes \V(\Sigma^\square_1, V^\square_1) = {\bf V} \otimes {\bf V}.
\]
As $\phi$ has order $3$, so too does $\D_{\phi, \emptyset}$.

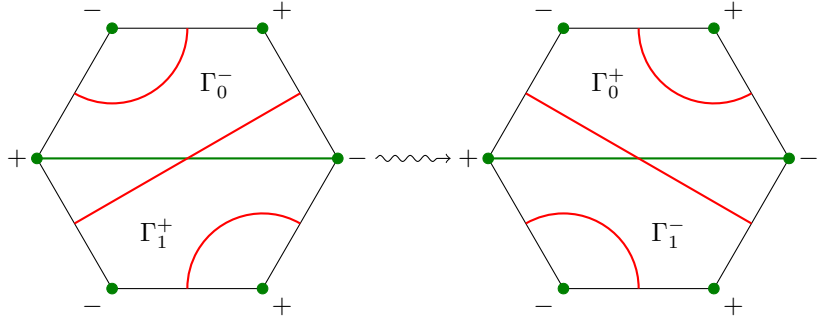
\begin{figure}
\begin{center}

\begin{tikzpicture}[
scale=2, 
boundary/.style={ultra thick}, 
vertex/.style={draw=green!50!black, fill=green!50!black},
decomposition/.style={thick, draw=green!50!black}, 
suture/.style={thick, draw=red}
]

\foreach \x/\rot in {0/0, 3 cm/-60}
{
\draw [xshift=\x, rotate=\rot] (0:1) -- (60:1) -- (120:1) -- (180:1) -- (240:1) -- (300:1) -- cycle;
\draw [xshift=\x, decomposition] (0:1) -- (180:1);
\draw [xshift=\x, rotate=\rot, suture] ($ 0.5*(60:1) + 0.5*(120:1) $) arc (0:-120:0.5);
\draw [xshift=\x, rotate=\rot, suture] ($ 0.5*(0:1) + 0.5*(-60:1) $) arc (60:180:0.5);
\draw [xshift=\x, rotate=\rot, suture] ($ 0.5*(0:1) + 0.5*(60:1) $) -- ($ 0.5*(180:1) + 0.5*(240:1) $);

\foreach \angle in {0, 60, 120, 180, 240, 300}
\fill [vertex, xshift=\x, rotate=\rot] (\angle:1) circle (1pt);

\draw [xshift=\x] (0:1) node [right] {$-$};
\draw [xshift=\x] (60:1) node [above right] {$+$};
\draw [xshift=\x] (120:1) node [above left] {$-$};
\draw [xshift=\x] (180:1) node [left] {$+$};
\draw [xshift=\x] (240:1) node [below left] {$-$};
\draw [xshift=\x] (300:1) node [below right] {$+$};
}

\draw (0.2, 0.5) node {$\Gamma^-_0$};
\draw (-0.2,-0.5) node {$\Gamma^+_1$};
\draw (2.8,0.5) node {$\Gamma^+_0$};
\draw (3.2,-0.5) node {$\Gamma^-_1$};

\draw [shorten >=1mm, -to, decorate, decoration={snake,amplitude=.4mm, segment length = 2mm, pre=moveto, pre length = 1mm, post length = 2mm}]
(1.2,0) -- (1.8,0);

\end{tikzpicture}

\caption{Order $3$ morphism and effect on sutures.}
\label{fig:order_3_morphism_sutures}
\end{center}
\end{figure}

Now we consider the effect of $\phi$ on sutures and $\D_{\phi,\emptyset}$ on suture elements. (Note that we consider the quadrangulation to be fixed, as $\phi$ rotates the hexagon and moves sutures.) Write $\Gamma^\pm_i$ for the standard $\pm$ sutures on $(\Sigma^\square_i, V^\square_i)$. We see that $\phi$ takes $3$ of the $4$ basis sutures to basis sutures:
\[
\begin{array}{ccc}
\Gamma_0^- \cup \Gamma_0^- & \mapsto & \Gamma_0^- \cup \Gamma_0^- \\
\Gamma_0^- \cup \Gamma_1^+ & \mapsto & \Gamma_0^+ \cup \Gamma_1^- \\
\Gamma_0^+ \cup \Gamma_1^+ & \mapsto & \Gamma_0^+ \cup \Gamma_1^+,
\end{array}
\]
however $\Gamma_0^+ \cup \Gamma_1^-$ is not taken to basic sutures.

Thus $\D_{\phi, \emptyset}$ is the identity on the Euler-graded $-2$ and $2$ summands of $\V(\Sigma,V)$, i.e. $\D_{\phi, \emptyset}$ takes $v_- \otimes v_- \mapsto v_ \otimes v_-$ and $v_+ \otimes v_+ \mapsto v_+ \otimes v_+$. On the Euler-graded $0$ summand, with respect to the basis $(v_- \otimes v_+, v_+ \otimes v_-)$, $\D_{\phi, \emptyset}$ has matrix
\[
\begin{bmatrix} 0 & \alpha \\ 1 & \beta \end{bmatrix}
\]
for some $\alpha, \beta \in \Z_2$. As $\phi$ has order $3$ the only possibility is that $\alpha = \beta = 1$. 

Thus, the unique nontrivial non-basis sutures on $(\Sigma,V)$, being given by $\phi(\Gamma_0^+ \cup \Gamma_1^-)$, must have suture element $v_- \otimes v_+ + v_+ \otimes v_- = \0 \otimes \1 + \1 \otimes \0$, i.e. is a ``superposition'' of the two basis elements.

Thus, the three nontrivial sets of sutures of Euler class $0$ on $(\Sigma,V)$ have suture elements $v_- \otimes v_+$, $v_+ \otimes v_-$, and $v_- \otimes v_+ + v_+ \otimes v_-$. So they sum to zero. These three sets of sutures form a bypass triple.

Now consider a decorated morphism $(\phi, \Gamma_c)$ mapping the hexagon $(\Sigma,V)$ into an occupied surface $(\Sigma',V')$. The bypass triple of sutures on $(\Sigma,V)$ combines with $\Gamma_c$ to form a bypass triple of sutures on $(\Sigma',V')$; and in fact any bypass triple of sutures on $(\Sigma',V')$ can be obtained this way. As the three suture elements in $\V(\Sigma,V)$ sum to zero, their images under $\D_{\phi, \Gamma_c}$ do too, and we have proved the following.
\begin{prop}[Bypass relation]
In an SQFT, if $\Gamma_0, \Gamma_1, \Gamma_2$ are a bypass triple of sutures then
\[
c(\Gamma_0) + c(\Gamma_1) + c(\Gamma_2) = 0.
\]
\qed
\end{prop}
(Note that this argument is heavily dependent on $\Z_2$ coefficients. Over $\Z$ coefficients, the situation is a little more complicated; see \cite{Me10_Sutured_TQFT, Massot09}.)

This last result gives us an effective way to compute suture elements in general: by successively applying several bypass surgeries, reducing to several sets of basic sutures.

Consider, for instance, a set of trivial sutures $\Gamma$. Let $\gamma$ be an innermost contractible component of $\Gamma$, i.e. $\gamma$ bounds a disc in $\Sigma \backslash \Gamma$. Take an attaching arc $c$ which intersects $\gamma$ twice and then intersects some other component of $\Gamma$. We see that the effect of bypass surgery along $c$, in either direction, gives the same result, which is identical to the original $\Gamma$. Thus $3c(\Gamma) = 0$, so immediately over $\Z_2$:
\begin{prop}
If $\Gamma$ is a trivial set of sutures, then $c(\Gamma)=0$.
\qed
\end{prop}

In general, if we have a sutured quadrangulated $(\Sigma,\Gamma,A)$, we successively simplify $\Gamma$ along the arcs $a$ of $A$ by bypass surgeries. If an arc $a$ of $A$ has $|a \cap \Gamma| > 1$, then in fact $|a \cap \Gamma| \geq 3$, and we consider bypass surgeries along $a$. Any such surgery, in either direction, reduces $|a \cap \Gamma|$ and, assuming $a$ and $\Gamma$ intersect efficiently, does not introduce trivial sutures (proposition \ref{prop:efficient_surgery}), and $c(\Gamma)$ is equal to the sum of the two suture elements obtained from bypass surgeries. After performing enough such surgeries (and simplifying isotopies at each stage), we express $c(\Gamma)$ as a sum of basic suture elements. 

Performing the above procedure on a disjoint union of occupied surfaces gives:
\begin{prop}
If $\Gamma_1, \Gamma_2$ are sets of sutures on $(\Sigma_1, V_1), (\Sigma_2, V_2)$ respectively, then 
\[
c(\Gamma_1 \sqcup \Gamma_2) = c(\Gamma_1) \otimes c(\Gamma_2) \in \V((\Sigma_1, V_1) \sqcup (\Sigma_2, V_2)) = \V(\Sigma_1, V_1) \otimes \V(\Sigma_2, V_2).
\]
\qed
\end{prop}

In \cite[theorem 16]{Massot09}, Massot shows that, for confining sutures $\Gamma$, $c(\Gamma) = 0$. (In fact we only need his result with $\Z_2$ coefficients; he proves the result over $\Z$ and indeed with twisted coefficients.) Given an isolated region, he shows that one can find bypass discs such that, performing bypass surgeries in either direction, we obtain an isolated region with simpler topology. By the bypass relation, if these simpler sutures have zero suture element, so do the original ones. Massot shows how to reduce to specific simple confining cases (annulus, punctured torus, disc) where one can verify suture elements are zero.
\begin{thm}[Massot \cite{Massot09}]
\label{thm:confining_zero}
If $\Gamma$ is confining then $c(\Gamma)=0$.
\end{thm}

We also note that from the above it is simple to compute the effect of a diagonal slide. A diagonal slide does not change the occupied surface or sutures on it, and hence does not change $\V(\Sigma,V)$ or suture elements, but does give a different quadrangulation. It only involves two joined squares, whose union is a hexagon; let the squares before the slide be $(\Sigma^\square_0, V^\square_0)$ and $(\Sigma^\square_1, V^\square_1)$, and afterwards $(\Sigma^\square_2, V^\square_2)$ and $(\Sigma^\square_3, V^\square_3)$. So a diagonal slide effectively gives an isomorphism
\[
{\bf V} \otimes {\bf V} = \V(\Sigma^\square_0, V^\square_0) \otimes \V(\Sigma^\square_1, V^\square_1) \To \V(\Sigma^\square_2, V^\square_2) \otimes \V(\Sigma^\square_3, V^\square_3)
\]
which we can compute by examining suture elements. We find the isomorphism fixes summands of Euler class $\pm 2$ and on the Euler class $0$ summand is given by a $2 \times 2$ matrix of order $3$.

\subsection{Spin networks}
\label{sec:spin_networks}

Many of the structures found here suggest a connection with \emph{spin networks} \cite{Penrose71, Major99, Kauffman_Lins94, Frenkel_Khovanov97}. We do not pursue these questions here, but we superficially note some connections.

Recall a spin network is essentially a graph with a Lie group representation attached to every edge and an intertwiner associated to every vertex. Taking the relevant group to be $SL(2,\C)$, we have a single irreducible representation of every positive integer dimension, and every finite-dimensional representation decomposes into a direct sum of such irreducibles. We write $V_n$ for the unique irreducible representation of dimension $n$. Spin networks can be generalised to quantum groups \cite{Frenkel_Khovanov97, Kauffman_Lins94} and in many other directions.

We can regard the fundamental 2-dimensional vector space ${\bf V}$ associated to an occupied square as the irreducible 2-dimensional representation. All the vector spaces we have obtained are of the form ${\bf V}^{\otimes n}$, which decomposes as a direct sum of representations of dimensions up to $n+1$.

A common way to represent spin networks is via diagrams (rather like sutures, though without the orientation requirements, sometimes with intersections, sometimes with over- and under-crossings). They are often drawn in the plane, but they can also be drawn on graphs. Each curve represents a ${\bf V}$ tensor factor, and various operations are assigned according to the geometry and topology of these curves, and give maps ${\bf V}^{\otimes m} \To {\bf V}^{\otimes n}$. Boxes are often drawn in the diagram to project to other $V_j$ factors, representing \emph{Jones-Wenzl projectors}. It is therefore possible to obtain from a sutured surface a map between $SL(2)$-representations. 

Moreover, a quadrangulation on the surface gives an explicit graph structure, as mentioned in section \ref{sec:ribbon_graphs}, the dual graph is a ribbon graph. Each edge of this graph corresponds to an arc of the quadrangulation, which intersects the sutures in an odd number of points, which can be interpreted as an $SL(2)$-representation. Each vertex of this graph corresponds to an occupied square with (potentially complicated) sutures which can be interpreted as intertwiners.

In the original $SL(2)$ spin networks of Penrose, a closed loop is given the value $-2$. In the quantum $SL(2)$ formulation of Frenkel--Khovanov, a closed loop is given the value $-q-q^{-1}$. We have the value of $0$ (mod $2$). The bypass relation seems to tell us that all Jones-Wenzl projectors onto dimension $3$ or higher, are zero. These are matters for further investigation.

\subsection{Creation and simple annihilation}

We now consider the operations that arise from creations and annihilations between occupied surfaces. 

Recall a creation (definition \ref{def:creation}) is a morphism which disjointly adds a square, $\phi: (\Sigma,V) \To (\Sigma,V) \sqcup (\Sigma^\square, V^\square)$. A decorated creation (definition \ref{def:decorated_creation_annihilation}) places basic sutures $\Gamma^\pm$ on the created square. The corresponding linear map in an SQFT is
\[
\D_{\phi, \Gamma^\pm} \; : \; \V(\Sigma,V) \To \V(\Sigma,V) \otimes \V(\Sigma^\square, V^\square) = \V(\Sigma,V) \otimes {\bf V}
\]
and since $\D_{\phi, \Gamma^\pm}$ must send
\[
c(\Gamma) \mapsto c(\Gamma \sqcup \Gamma^\pm) = c(\Gamma) \otimes c(\Gamma^\pm) = c(\Gamma) \otimes v_\pm
\]
and so (by varying $\Gamma$ over a basis of sutures), setting we have the descriptions
\[
\begin{array}{ccc}
\D_{\phi, \Gamma^-} \; : \; x &\mapsto& x \otimes v_- = x \otimes \0, \\
\D_{\phi, \Gamma^+} \; : \; x &\mapsto& x \otimes v_+ = x \otimes \1, 
\end{array}
\]
which are precisely digital creation operators as given in the introduction.

Before we consider annihilations in general, we consider a simple special case. Recall an annihilation (definition \ref{def:annihilation}) is a map $\phi: (\Sigma,V) \To (\Sigma',V')$ which adjoins an annihilator square $(\Sigma^\square, V^\square)$ along $3$ consecutive edges of the boundary. A decorated annihilation (definition \ref{def:decorated_creation_annihilation}) places basic sutures $\Gamma^\pm$ on the annihilator square. There are two vertices of $\Sigma$-type, one of each sign, $v_-, v_+$, and three edges of $\Sigma$-type, $e_1, e_2, e_3$ as in figure \ref{fig:annihilation_basic_sutures}.

In this simple case, suppose we have a quadrangulation $A$ of $(\Sigma,V)$ which has a square $(\Sigma^\square_1, V^\square_1)$ with $e_1, e_2, e_3$ as consecutive edges, which we call the ``annihilated square''. From our quadrangulation $A$, we obtain a quadrangulation $A'$ of $(\Sigma,V)$ simply by removing all the edges of the annihilated square, along with the internal vertices. The squares of $A$ (minus the annihilated square) and $A'$ are naturally bijective, though one square of $A'$ is isotoped across the vacuum left by annihilator-annihilated annihilation.

(In general, there may be many edges of $A$ incident to $v_-$ and $v_+$. The notion of slack square collapse, discussed in the next section, can deal with the general situation.)

If the annihilated and annihilator squares both have standard positive sutures, then the sutures obtained on $(\Sigma', V')$ are trivial; similarly if both squares have standard negative sutures. However if annihilator and annihilated squares have standard sutures of opposite sign, then the sutures obtained on $(\Sigma',V')$ form vacuum sutures on the vacuum background consisting of the union of annihilator and annihilated backgrounds. See figure \ref{fig:annihilation_basic_sutures}.

\begin{figure}
\begin{center}

\begin{tikzpicture}[
scale=1.7, 
fill = gray!10,
boundary/.style={ultra thick}, 
vertex/.style={draw=green!50!black, fill=green!50!black},
decomposition/.style={thick, draw=green!50!black}, 
suture/.style={thick, draw=red}
]

{
\fill (-0.5,0) -- (0,0) arc (-180:0:1.5) -- (3.5,0) -- (3.5,1.8) -- (-0.5,1.8) -- cycle;

\coordinate (p1) at (0,0);
\coordinate [label = above:{$v_-$}] (v-) at (1,0);
\coordinate [label = above:{$v_+$}] (v+) at (2,0);
\coordinate (p4) at (3,0);

\draw [boundary] (-0.5,0) -- (p1) arc (-180:0:1.5) -- (3.5,0);
\draw (p1) -- node [below left] {$e_1$} (v-) -- node [below left] {$e_2$} (v+) -- node [below left] {$e_3$} (p4);
\draw [decomposition] (p1) arc (180:0:1.5);
\draw [suture] (1.5,0) arc (180:0:0.5);
\draw [suture] (1.5,0) arc (-180:0:0.5);
\draw [suture] (0.5,0) to [bend left = 45] (1,0.75) to [bend right = 45] (1.5,1.5) -- (1.5,1.8);
\draw [suture] (0.5,0) to [bend right = 45] (1,-0.75) to [bend left = 45] (1.5,-1.5);
\draw (-0.2,1.2) node {$\phi(\Sigma)$};
\draw (2,0.75) node {$(\Sigma^\square_1, V^\square_1)$};
\draw (2,-0.75) node {$(\Sigma^\square, V^\square)$};

\fill (4.5,0) -- (5,0) arc (-180:0:1.5) -- (8.5,0) -- (8.5,1.8) -- (4.5,1.8) -- cycle;

\coordinate (q1) at (5,0);
\coordinate [label = above:{$v_-$}] (v-2) at (6,0);
\coordinate [label = above:{$v_+$}] (v+2) at (7,0);
\coordinate (q4) at (8,0);

\draw [boundary] (4.5,0) -- (q1) arc (-180:0:1.5) -- (8.5,0);
\draw (q1) -- node [below left] {$e_1$} (v-2) -- node [below left] {$e_2$} (v+2) -- node [below left] {$e_3$} (q4);
\draw [decomposition] (q1) arc (180:0:1.5);
\draw [suture] (6.5,0) arc (180:0:0.5);
\draw [suture] (5.5,0) arc (-180:0:0.5);
\draw [suture] (5.5,0) to [bend left = 45] (6,0.75) to [bend right = 45] (6.5,1.5) -- (6.5,1.8);
\draw [suture] (7.5,0) to [bend left = 45] (7,-0.75) to [bend right = 45] (6.5,-1.5);
\draw (4.8,1.2) node {$\phi(\Sigma)$};
\draw (7,0.75) node {$(\Sigma^\square_1, V^\square_1)$};
\draw (6,-0.75) node {$(\Sigma^\square, V^\square)$};

\foreach \point in {p1, v-, v+, p4, q1, v-2, v+2, q4}
\fill [vertex] (\point) circle (2pt);
}

\end{tikzpicture}

\caption{Annihilation with basic sutures on annihilator and annihilated squares.}
\label{fig:annihilation_basic_sutures}
\end{center}
\end{figure}
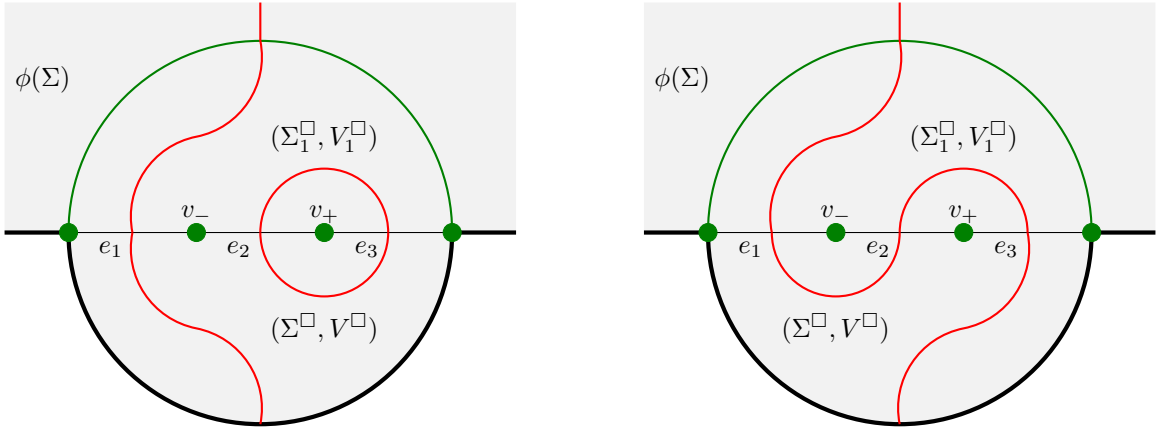

Letting the squares of the quadrangulation $A$ consist of the annihilated square $(\Sigma^\square_1, V^\square_1)$, and other squares $(\Sigma^\square_i, V^\square_i)$, the quadrangulation $A'$ can be written as $(\Sigma^\square_i, V^\square_i)$, given the bijection between squares of the previous paragraph. We thus obtain the SQFT maps $\D_{\phi, \Gamma^\pm}$ as
\[
\begin{array}{rcl}
\D_{\phi, \Gamma^\pm} \; : \; \V(\Sigma^\square_1, V^\square_1) \otimes \bigotimes_i \V(\Sigma^\square_i, V^\square_i) &\To&  \bigotimes_i  \V( \Sigma^{\square}_i, V^{\square}_i )) \\
\D_{\phi,\Gamma^-} \; : \; v_- \otimes x = \0 \otimes x &\mapsto& 0 \\
v_+ \otimes x = \1 \otimes x &\mapsto& x, \\
\D_{\phi,\Gamma^+} \; : \; v_- \otimes x = \0 \otimes x &\mapsto& x \\
v_+ \otimes x = \1 \otimes x &\mapsto& 0.
\end{array}
\]
Thus $\D_{\phi, \Gamma^-} = a_\1 \otimes 1^{\otimes I(\Sigma,V)-1}$ and $\D_{\phi, \Gamma^+} = a_\0 \otimes 1^{\otimes I(\Sigma,V)-1}$, where $a_\0$, $a_\1$ are digital annihilation operators, of the simplest type with only one tensor factor, ${\bf V}^{\otimes 1} \To {\bf V}^{\otimes 0} = \Z_2$. 

(In general, higher order annihilation operators ${\bf V}^{\otimes (n+1)} \To {\bf V}^{\otimes n}$ will occur, according to the number of edges incident at $v_-, v_+$.)

Thus the SQFT maps $\D_{\phi, \Gamma^\pm}$ are general digital annihilation operators. We see that $\D_{\phi, \Gamma^-}$ is $a_\1$ tensor the identity, and the standard negative sutures on the annihilator square (or a $\0$) annihilate standard positive sutures (or delete a $\1$) to give a vacuum; if no such annihilation is possible, $\D_{\phi, \Gamma^-}$ returns $0$. Similarly for $\D_{\phi, \Gamma^+}$ and $a_\0$. So we see the information-processing aspect of a digital annihilation operator arising in the combinatorial way sutures are combined.

We turn next to more complicated square collapses.

\subsection{Square collapse mechanics}
\label{sec:square_collapse_mechanics}

Let us establish the setup for a square collapse. Let $(\Sigma^\square, V^\square)$ be a square in a slack quadrangulation $Q$ of $(\Sigma,V)$, and let $y$ be a slack vertex, which shall be collapsed onto the opposite vertex $x$ in $(\Sigma^\square, V^\square)$. Obviously then $y$ and $x$ have the same sign; without loss of generality let it be positive; the negative case is similar with signs reversed. Let the edges emanating from $y$, in anticlockwise order, be $e_1, \ldots, e_n$, where $e_1$ and $e_n$ are consecutive edges of $(\Sigma^\square, V^\square)$ incident to $v$; write $e_i$, where $i$ is taken mod $n$.) For now we only consider sutures which are basic with respect to the slack quadrangulation. 

First consider the degenerate case $n =1$, so there is only one edge incident to $x$, and the square collapse is as in figure \ref{fig:degenerate_slack_square_collapse}. We see that if $(\Sigma^\square, V^\square)$ has basic positive sutures, then sutures are trivial. If $(\Sigma^\square, V^\square)$ has basic negative sutures, then after collapsing the square we retain basic sutures on all the squares, all of the same sign; we effectively just isotope away a vacuum.

We now assume $n \geq 2$. A neighbourhood of $y$ is split into wedges by the $e_i$; let $w_i$ ($i$ taken mod $n$) be the wedge between $e_i$ and $e_{i+1}$, so $w_0 = w_n$ is a corner of $(\Sigma^\square, V^\square)$ and the $w_i$ are in anticlockwise order. Note that some of the $w_i$ may correspond to opposite corners of the same square. 

The wedge $w_i$ is anticlockwise of $e_i$ and clockwise of $e_{i+1}$; and so we consider the sutures $\gamma_i$ intersecting $e_i$ and $e_{i+1}$ in this region. Since the sutures are basic, each $\gamma_i$ either consists of one suture running from $e_i$ to $e_{i+1}$; or runs from $e_i$ out of the wedge, and from $e_{i+1}$ out of the wedge, and we do not care for the moment about their behaviour further away from $v$. Call $\gamma_i$ \emph{positive} in the first case (since then $\gamma_i$ encloses a positive region in $w_i$ around $v$), and \emph{negative} in the second case. By $\gamma_0 = \gamma_n$ we shall mean the sutures in the collapsed square (not just the wedge near $v$).

(There is a subtlety when two of the wedges $w_i, w_j$ correspond to the same square of the quadrangulation, so that $\gamma_i, \gamma_j$ form part of basic sutures on the same square. We allow that $\gamma_i$ might be negative and $\gamma_j$ positive, but we interpret this as meaning that $\gamma_i$ is ``really'' positive, and that the sutures $\gamma_i$ which run out of the wedge $w_i$ close up further away from $y$. So the $\gamma_i$ are really just taken to represent the sutures in a neighbourhood of $y$, and not necessarily to indicate the overall topology of the sutures, although on a first reading they can be interpreted this way.)

The square collapse operation moves $y$ into $x$, collapses $e_1$ and $e_n$ on to adjacent edges of $(\Sigma^\square, V^\square)$, and modifies the edges $e_2, \ldots, e_{n-1}$ now to end at $x=y$. Call these collapsed edges $e'_1, e'_2, \ldots, e'_n$ in the natural way, so that $e'_i$ is the image of $e_i$ after the collapse. So near $x$ we have edges emanating in anticlockwise order $e'_1, \ldots, e'_n$, and wedge regions $w'_1, \ldots, w'_{n-1}$, where $w'_i$ lies between $e'_i$ and $e'_{i+1}$ and is the image of $w_i$ after the collapse.

We see that if $\gamma_0$, the sutures on the collapsed square, are negative, then the effect on sutures is minimal; each wedge $w'_i$ then has sutures $\gamma'_i$ which behave just as the sutures $\gamma_i$ did in $w_i$. However if the sutures $\gamma_0$ are positive, the situation is much more complicated, the sutures need not remain basic, and the edges $e'_i$ will intersect sutures in three points. See figure \ref{fig:square_collapsing_basic_sutures} for an example. 

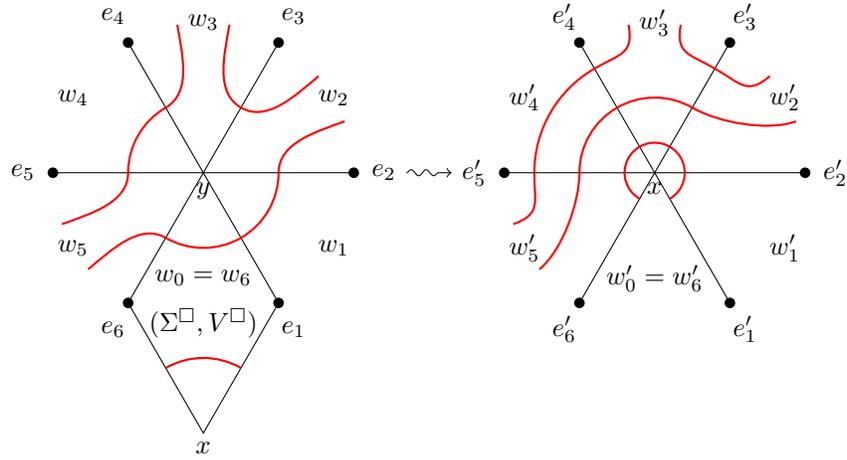
\begin{figure}
\begin{center}

\begin{tikzpicture}[
scale=2, 
boundary/.style={ultra thick}, 
suture/.style={thick, draw=red}
]

\coordinate [label = below:{$y$}] (v) at (0,0);
\coordinate [label = below:{$x$}] (x) at (0,-1.732);
\coordinate [label = below:{$x$}] (x2) at (3,0);

\foreach \angle/\edgelabel in {0/e_2, 60/e_3, 120/e_4, 180/e_5, 240/e_6, 300/e_1}{
\draw (v) -- (\angle:1);
\draw (\angle:1.2) node {$\edgelabel$} ;
\fill [black] (\angle:1) circle (1pt); }

\draw (240:1) -- (x) -- (300:1);

\draw [suture] (220:1) .. controls (220:0.8) and ($ (240:0.5) + (150:0.2) $) .. (240:0.5);
\draw [suture] (-60:0.5) arc (-60:0:0.5);
\draw [suture] (0:0.5) .. controls ($ (0:0.5) + (90:0.2) $) and (20:0.8) .. (20:1);
\draw [suture] (60:0.5) .. controls ($ (60:0.5) + (-30:0.2) $) and (40:0.8) ..  (40:1);
\draw [suture] (60:0.5) .. controls ($ (60:0.5) + (150:0.2) $) and (80:0.8) .. (80:1);
\draw [suture] (120:0.5) .. controls ($ (120:0.5) + (30:0.2) $) and (100:0.8) .. (100:1);
\draw [suture] (120:0.5) arc (120:180:0.5);
\draw [suture] (180:0.5) .. controls (-0.5,-0.2) and (200:0.8) .. (200:1);
\draw [suture] (240:0.5) arc (240:300:0.5);
\draw [suture] ($ 0.5*(240:1) + 0.5*(x)  $) arc (120:60:0.5); 

\draw (-30:1) node {$w_1$};
\draw (30:1) node {$w_2$};
\draw (90:1) node {$w_3$};
\draw (150:1) node {$w_4$};
\draw (210:1) node {$w_5$};
\draw (270:0.7) node {$w_0 = w_6$};
\draw (270:1) node {$(\Sigma^\square, V^\square)$};

\draw [shorten >=1mm, -to, decorate, decoration={snake,amplitude=.4mm, segment length = 2mm, pre=moveto, pre length = 1mm, post length = 2mm}]
(1.3,0) -- (1.7,0);

\foreach \angle/\edgelabel in {0/e'_2, 60/e'_3, 120/e'_4, 180/e'_5, 240/e'_6, 300/e'_1}{
\draw[xshift=3cm] (0,0) -- (\angle:1);
\draw[xshift=3cm] (\angle:1.2) node {$\edgelabel$} ;
\fill[xshift=3cm] [black] (\angle:1) circle (1pt); }

\draw [xshift=3cm, suture] (-60:0.2) arc (-60:240:0.2);
\draw [xshift=3cm, suture] (20:1) .. controls (20:0.8) and ($ (60:0.5) + (-30:0.2) $) .. (60:0.5);
\draw [xshift=3cm, suture] (40:1) .. controls (40:0.8) and ($ (60:0.8) + (-30:0.2) $) .. (60:0.8);
\draw [xshift=3cm, suture] (60:0.5) arc (60:120:0.5);
\draw [xshift=3cm, suture] (80:1) .. controls (80:0.8) and ($ (60:0.8) + (150:0.2) $) .. (60:0.8);
\draw [xshift=3cm, suture] (100:1) .. controls (100:0.8) and ($ (120:0.8) + (30:0.2) $) .. (120:0.8);
\draw [xshift=3cm, suture] (120:0.5) arc (120:180:0.5);
\draw [xshift=3cm, suture] (120:0.8) arc (120:180:0.8);
\draw [xshift=3cm, suture] (200:1) .. controls (200:0.8) and (-0.8,-0.2) .. (180:0.8);
\draw [xshift=3cm, suture] (220:1) .. controls (220:0.8) and (-0.5,-0.2) .. (180:0.5);

\draw [xshift=3cm](-30:1) node {$w'_1$};
\draw [xshift=3cm](30:1) node {$w'_2$};
\draw [xshift=3cm](90:1) node {$w'_3$};
\draw [xshift=3cm](150:1) node {$w'_4$};
\draw [xshift=3cm](210:1) node {$w'_5$};
\draw [xshift=3cm](270:0.7) node {$w'_0 = w'_6$};

\end{tikzpicture}

\caption{Effect of square collapsing on basic sutures. Here the sutures $\gamma_i$ in the wedges $w_i$ are as follows: $s_1 = +$, $s_2 = -$, $s_3 = -$, $s_4 = +$, $s_5 = -$.}
\label{fig:square_collapsing_basic_sutures}
\end{center}
\end{figure}

Let $s_i$ be the sign of the sutures $\gamma_i$, so each $s_i = \pm$. We are now assuming that $s_0 = +$. Note that if all $s_i = +$ then we have trivial sutures; so assume at least one $s_i = -$. Let $m$ be the minimal positive integer such that $s_m = -$. If $\gamma_m$ are the only negative sutures, so that all other $s_i = +$, then after the collapse we still have basic sutures $\gamma'_i$ on each square, which are now all positive. See figure \ref{fig:square_collapsing_all_except_one_positive}.

\begin{figure}
\begin{center}

\begin{tikzpicture}[
scale=2, 
boundary/.style={ultra thick}, 
suture/.style={thick, draw=red}
]

\coordinate [label = above:{$v$}] (v) at (0,0);
\coordinate [label = below:{$x$}] (x) at (0,-1.732);

\foreach \angle in {0, 60, 120, 180, 240, 300}{
\draw (v) -- (\angle:1);
\fill [black] (\angle:1) circle (1pt); }

\draw (240:1) -- (x) -- (300:1);

\draw [suture] (60:0.5) arc (60:-240:0.5);
\draw [suture] (60:0.5) .. controls ($ (60:0.5) + (150:0.2) $) and (80:0.8) .. (80:1);
\draw [suture] (120:0.5) .. controls ($ (120:0.5) + (30:0.2) $) and (100:0.8) .. (100:1);
\draw [suture] ($ 0.5*(240:1) + 0.5*(x) $) arc (120:60:0.5);

\draw (-30:1) node {$w_1$};
\draw (90:1) node {$w_m$};
\draw (270:0.7) node {$w_0 = w_6$};
\draw (270:1) node {$(\Sigma^\square, V^\square)$};

\draw [shorten >=1mm, -to, decorate, decoration={snake,amplitude=.4mm, segment length = 2mm, pre=moveto, pre length = 1mm, post length = 2mm}]
(1.2,0) -- (1.8,0);

\foreach \angle in {0, 60, 120, 180, 240, 300}{
\draw[xshift=3cm] (0,0) -- (\angle:1);
\fill[xshift=3cm] [black] (\angle:1) circle (1pt); }

\draw [xshift=3cm, suture] (-60:0.5) arc (-60:240:0.5);
\draw [xshift=3cm, suture] (80:1) .. controls (80:0.8) and (100:0.8) .. (100:1);

\end{tikzpicture}

\caption{Effect of square collapsing when all $s_i$ except $s_m$ are positive: sutures all become positive.}
\label{fig:square_collapsing_all_except_one_positive}
\end{center}
\end{figure}
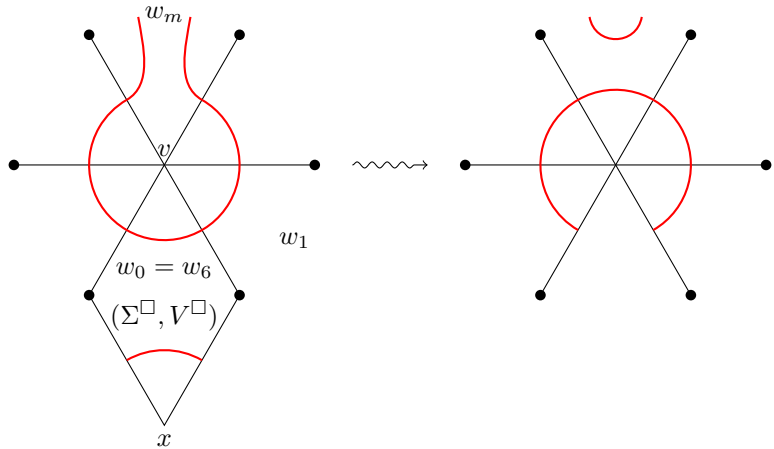

So we now assume that at least two of the $s_i$ are negative; let $p$ be the next integer after $m$ such that $s_p = -$.

Observe that $e'_1, e'_2, \ldots, e'_m$ all intersect the sutures in one point, but $e'_{m+1}$ intersects the sutures in three points. (See figure \ref{fig:bypass_simplifying_collapsed_sutures}.) Perform bypass surgery along the interval of $e'_{m+1}$ intersecting these points, in both directions. 

If $s_m, s_p$ are the only negative signs, then the results of both bypass surgeries are basic: in one case we obtain all $\gamma'_i$ basic with $\gamma'_m$ negative, $\gamma'_p$ positive; in the other case we obtain $\gamma'_m$ positive and $\gamma'_m$ negative; in both cases all $\gamma'_i$ other than $\gamma'_m, \gamma'_p$ have the same sign as the original $\gamma_i$. 

If $s_m, s_p$ are not the only negative signs, then for one surgery (depicted on the right of figure \ref{fig:bypass_simplifying_collapsed_sutures}), we actually obtain basic sutures: the sutures on $w'_m$ are positive, and on all other $w'_i$ the sutures simplify to have the same sign as the original $\gamma_i$. In the diagram, we depict an arc of sutures in $w'_p$ which can be isotoped off the diagram; this isotopy successively simplifies the sutures in each wedge.

(In the subtle case where $\gamma_m$ is ``really'' positive, but the sutures just close up further away from the wedge, we see that these sutures on the right of figure \ref{fig:bypass_simplifying_collapsed_sutures} will be trivial, and so make no contribution to contact elements. So we obtain the same result as if $\gamma_m$ were drawn ``truthfully'' as positive.)

For the other surgery (depicted on the left of figure \ref{fig:bypass_simplifying_collapsed_sutures}), sutures remain basic and positive on $w'_1, \ldots, w'_{m-1}$, negative on $w'_m$, and positive on $w'_{m+1}, \ldots, w'_{p-1}$. However the sutures on $w'_p$ are not basic. Nevertheless, we have essentially reduced to a smaller case of the situation we started with; we consider the situation as if $w'_p$ was the first negative wedge and proceed as before, performing bypass surgeries along $e'_{p+1}$. As we proceed, and perform bypass surgeries, we obtain sutures on the collapsed quadrangulation which are basic. The basic sets of sutures so obtained each consist of basic sutures $w'_i$ on each wedge, of the same sign as the corresponding original wedge $w_i$, except that for precisely one wedge, the sign is changed from $-$ to $+$. 

Using the fact that bypass triples have suture elements summing to zero, we can summarise the above discussion by the following proposition. Let $\gamma_i^{s_i}$ denote the $s_i$-signed sutures on the wedge $w_i$, and $\gamma_i^{s'_i}$ denote $s_i$ signed sutures on the wedge $w'_i$. This also covers the degenerate case $n=1$.
\begin{prop}
\label{prop:square_collapse_map}
If a collapsed square $(\Sigma^\square, V^\square)$, with collapsed positive vertices, has basic sutures $\gamma_0^\pm$, and the wedges $w_i$ around $v$ have basic sutures $\gamma_i^{s_i}$ on the wedges $w_i$ around $v$ (for $i=1, \ldots, n-1)$, then
\begin{align*}
c(\gamma_0^- \cup \gamma_1^{s_1} \cup \cdots \cup \gamma_{n-1}^{s_{n-1}}) &= c(\gamma_1^{s'_1} \cup \cdots \cup \gamma_{n-1}^{s'_{n-1}}), \\
c(\gamma_0^+ \cup \gamma_1^{s_1} \cup \cdots \cup \gamma_{n-1}^{s_{n-1}}) &=
\sum_{i \; : \; s_i = -} c(\gamma_1^{s'_1} \cup \cdots \cup \gamma_{i-1}^{s'_{i-1}} \cup \gamma_i^+ \cup \gamma_{i+1}^{s'_{i+1}} \cup \cdots \cup \gamma_{n-1}^{s'_{n-1}})
\end{align*}
where we assume the sutures are held constant outside of these wedges.
\qed
\end{prop}
Note that this statement holds even if distinct wedges $w_i, w_j$ are corners of the same square of the quadrangulation (so that the same applies to $w'_i, w'_j$), as all calculations are local. If $s_i = s_j = -1$ then we obtain two separate terms in the sum; one turns $s_i$ to $1$, and then $\gamma_j$ ``looks negative but is really positive''; the other turns $s_j$ to $1$ and $\gamma_i$ feigns negativity.

The corresponding statement is clear when the collapsed vertices $y$ and $x$ have negative sign.

\begin{figure}
\begin{center}

\begin{tikzpicture}[
scale=2, 
boundary/.style={ultra thick}, 
suture/.style={thick, draw=red}
]

\coordinate [label = below:{$e'_1$}] (e1) at (0:1);
\coordinate [label = above right:{$e'_m$}] (em) at (60:1);
\coordinate [label = above:{$e'_{m+1}$}] (em1) at (90:1);
\coordinate [label = above left:{$e'_p$}] (ep) at (150:1);
\coordinate [label = below:{$e'_{p+1}$}] (ep1) at (180:1);
\coordinate (v) at (0,0);

\foreach \angle in {0, 30, 60, 90, 120, 150, 180}{
\draw (v) -- (\angle:1);
\fill [black] (\angle:1) circle (1pt); }

\draw [suture] (0:0.2) arc (0:180:0.2);
\draw [suture] (70:1) .. controls (70:0.8) and (0.2,0.5) .. (0,0.5);
\draw [suture] (80:1) .. controls (80:0.8) and (0.2,0.8) .. (0,0.8);
\draw [suture] (90:0.5) arc (90:180:0.5);
\draw [suture] (90:0.8) arc (90:150:0.8);
\draw [suture] (150:0.8) .. controls ($ (150:0.8) + (240:0.2) $) and (160:0.8)  .. (160:1);
\draw [suture] (170:1) .. controls (170:0.8) and (-0.8,0.2) .. (-0.8,0);

\draw [shorten >=1mm, -to, decorate, decoration={snake,amplitude=.4mm, segment length = 2mm, pre=moveto, pre length = 1mm, post length = 2mm}]
(1.2,0) -- (1.8,0);

\foreach \angle in {0, 30, 60, 90, 120, 150, 180}{
\draw[xshift=3cm] (0,0) -- (\angle:1);
\fill[xshift=3cm] [black] (\angle:1) circle (1pt); }

\draw [xshift=3cm, suture] (0:0.5) arc (0:150:0.5);
\draw [xshift=3cm, suture] (70:1) .. controls (70:0.8) and (80:0.8) .. (80:1);
\draw [xshift=3cm, suture] (150:0.5) .. controls ($ (150:0.5) + (240:0.2) $) and (160:0.8) .. (160:1);
\draw [xshift=3cm, suture] (170:1) .. controls (170:0.8) and (-0.8,0.2) .. (-0.8,0);
\draw [xshift=3cm, suture] (-0.2,0) .. controls (-0.2,0.1) and (-0.5,0.1) .. (-0.5,0);

\draw [shorten >=1mm, -to, decorate, decoration={snake,amplitude=.4mm, segment length = 2mm, pre=moveto, pre length = 1mm, post length = 2mm}]
(-1.2,0) -- (-1.8,0);

\foreach \angle in {0, 30, 60, 90, 120, 150, 180}{
\draw[xshift=-3cm] (0,0) -- (\angle:1);
\fill[xshift=-3cm] [black] (\angle:1) circle (1pt); }

\draw [xshift=-3cm, suture] (0.5,0) arc (0:60:0.5);
\draw [xshift=-3cm, suture] (60:0.5) .. controls ($ (60:0.5) + (150:0.1) $) and (70:0.8) .. (70:1);
\draw [xshift=-3cm, suture] (80:1) .. controls (80:0.8) and (0.1,0.5) .. (0,0.5);
\draw [xshift=-3cm, suture] (0,0.5) .. controls (-0.2,0.5) and (-0.3,0.2) .. (-0.3,0);
\draw [xshift=-3cm, suture] (160:1) .. controls (160:0.8) and (-0.5,0.2) .. (-0.5,0);
\draw [xshift=-3cm, suture] (170:1) .. controls (170:0.8) and (-0.8,0.2) .. (-0.8,0);

\end{tikzpicture}
\caption{Bypass surgeries near the first negative wedge $w_m$ simplify sutures.}
\label{fig:bypass_simplifying_collapsed_sutures}
\end{center}
\end{figure}
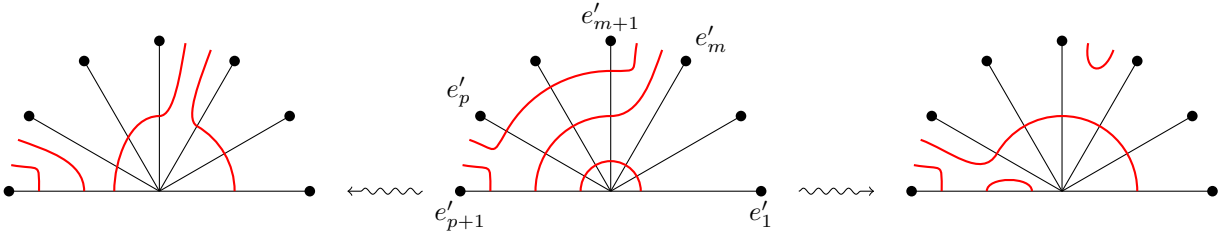

\subsection{Elementary morphism operators}

We consider the SQFT maps associated to the elementary decorated morphisms $(\phi, \Gamma_c): (\Sigma,V) \To (\Sigma',V')$. We have already considered decorated creations; the associated maps are digital creation operators. And we have seen (lemma \ref{lem:standard_gluing_isomorphism}) that a standard gluing gives a bijection between squares of a quadrangulation, so acts as the identity map on corresponding tensor decompositions. It remains to consider decorated annihilations in general, folds and zips.

Suppose we have a quadrangulation $A$ of $(\Sigma,V)$. We recall the discussion at the end of section \ref{sec:building_decorated_morphisms}, which describes how to obtain a quadrangulation $A'$ on $(\Sigma',V')$.

\bigskip

\noindent \emph{Annihilations and folds.} As mentioned in section \ref{sec:defns_decorated_morphisms} (figure \ref{fig:decorated_annihilation_as_fold}), $\pm$ decorated annihilations are decorated-isotopic to $\pm$ folds, and hence (lemma \ref{lem:decorated_isotopy_equal_map}) give equal SQFT maps. So let $\phi$ be a fold, with folded vertex $y$, say positive. The quadrangulation $A$ gives a slack quadrangulation $Q$ of $(\Sigma',V')$ with one internal vertex $y$; the squares of $A$ and $Q$ are naturally bijective. Write the squares of $A$ or $Q$ around $y$ as $(\Sigma^\square_0, V^\square_0), \ldots, (\Sigma^\square_n, V^\square_n)$. Denote the other squares by $(\Sigma_j, V_j)$ and $(\Sigma'_j, V'_j)$.

Suppose we have basic sutures $\Gamma$ on $(\Sigma,V)$ restricting to $\Gamma_i^{s_i}$ on $(\Sigma^\square_i, V^\square_i)$ (where $s_- = \pm$). Let the rest of the surface be $(\Sigma_j, V_j)$, with sutures $\Gamma_j$. So
\[
\begin{array}{cll}
c(\Gamma) &= c(\Gamma_0^{s_0} \cup \cdots \cup \Gamma_n^{s_n} \cup \Gamma_j) &\in \V(\Sigma,V) \\
&= c(\Gamma_0^{s_0}) \otimes \cdots \otimes c(\Gamma_n^{s_n}) \otimes c(\Gamma_j) &\in \V(\Sigma^\square_0, V^\square_0) \otimes \cdots \otimes \V(\Sigma^\square_n, V^\square_n) \otimes \V(\Sigma_j, V_j).
\end{array}
\]
Now we perform a slack square collapse to remove the positive internal vertex $y$ and apply proposition \ref{prop:square_collapse_map}. First suppose all the squares $(\Sigma^\square_i, V^\square_i)$ around $y$ are distinct.
We can assume that the collapsed square is $(\Sigma^{\square'}_0, V^{\square'}_0)$. After the collapse, we have a quadrangulation $A'$ of $(\Sigma',V')$ given by $(\Sigma^{\square'}_1, V^{\square'}_1), \ldots, (\Sigma^{\square'}_n, V^{\square'}_n)$, naturally corresponding with the pre-existing non-collapsed squares; and the previous $(\Sigma_j, V_j)$. The notation is then as in the proposition, so depending on the sign $s_0$, we have
\begin{align*}
c(\Gamma_0^{-} \cup \Gamma_1^{s_1} \cup \cdots \cup \Gamma_n^{s_n} \cup \Gamma_j) &= c(\Gamma_1^{s'_1} \cup \cdots \cup \Gamma_n^{s'_n}), \\
c(\Gamma_0^+ \cup \Gamma_1^{s_1} \cup \cdots \cup \Gamma_n^{s_n} \cup \Gamma_j) &= \sum_{i : s_i = -} c(\Gamma_1^{s'_1} \cup \cdots \cup \Gamma_{i-1}^{s'_{i-1}} \cup \Gamma_i^{+'} \cup \Gamma_{i+1}^{s'_{i+1}} \cup \cdots \cup \Gamma_n^{s'_n} \cup \Gamma_j)
\end{align*}
where $\Gamma_i^{s'_i}$ denotes $s_i$ signed sutures on $(\Sigma^{\square'}_i, V^{\square'}_i)$. Writing in terms of tensor powers, we see that $\D_{\phi,\emptyset}$ takes
\begin{align*}
\0 \otimes x_1 \otimes \cdots \otimes x_n \otimes c(\Gamma_j) &\mapsto x_1 \otimes \cdots \otimes x_n \otimes c(\Gamma_j) \\
\1 \otimes x_1 \otimes \cdots \otimes x_n \otimes c(\Gamma_j) &\mapsto \sum_{x_i = \0} x_1 \otimes \cdots \otimes x_{i-1} \otimes \1 \otimes x_{i+1} \otimes \cdots \otimes x_n \otimes c(\Gamma_j).
\end{align*}
Thus $\D_{\phi, \emptyset}$ is of the form $a_\0 \otimes 1$, for a digital annihilation operator $a_\0 : {\bf V}^{\otimes (n+1)} \To {\bf V}^{\otimes n}$, and an identity operator $1 : \V(\Sigma_j, V_j) \To \V(\Sigma_j,V_j)$ --- a general digital annihilation operator.

Now suppose two of the squares found around $y$ are opposite corners of the same square $(\Sigma^\square_i, V^\square_i)$. Only if $(\Sigma^\square_i, V^\square_i)$ has negative sutures does it contribute to the sum, so assume the sutures are negative. As discussed in section \ref{sec:square_collapse_mechanics}, we have a separate term in the sum for adjusting each corner; and when the signs on the two corners disagree, the $+$ sign describes the sutures on the square. Effectively $(\Sigma^\square_i, V^\square_i)$ contributes two terms to the sum, but the two terms are equal and (mod $2$) cancel. So the tensor factor $\V(\Sigma^\square_i, V^\square_i)$ has no effect on the sum, and can be considered as one of the factors on which $\D_{\phi, \emptyset}$ acts as the identity. We still obtain a general digital annihilation operator.

Obviously, for a negative annihilation or fold, the same argument applies and we again obtain a general digital annihilation operator $a_\1 \otimes 1$.

\bigskip

\noindent \emph{Zips.} If $\phi$ is a zip and $A$ is a quadrangulation of $(\Sigma,V)$, then we obtain a slack quadrangulation on $(\Sigma',V')$ with two internal vertices, one of each sign. We may remove these two internal vertices with two slack square collapses, one of each sign. By the same argument as given above, each square collapse gives a general digital annihilation operator. So $\D_{\phi, \emptyset}$ is a composition of two general digital annihilation operators, one of each sign.

Now we proved in proposition \ref{prop:decorated_morphism_composition} that any non-confining decorated morphism is a composition of decorated creations, gluings, folds and zips. Above we have concluded that every such map gives an SQFT operator which is a digital creation operator, the identity, a general digital annihilation operator, or two digital annihilation operators. 

On the other hand, if $(\phi, \Gamma_c): (\Sigma,V) \To (\Sigma',V')$ is a confining decorated morphism, then for any sutures $\Gamma$ on $(\Sigma,V)$, the sutures $\phi(\Gamma) \cup \Gamma'$ on $(\Sigma',V')$ are confining, and by theorem \ref{thm:confining_zero} have zero suture element; thus $\D_{\phi, \Gamma_c} = 0$, which is certainly achievable as a composition of creation and general annihilation operators (e.g. create a $\1$ and then try to delete a $\0$ on that factor).

We have now proved the following precise version of our main theorem \ref{thm:main_thm}.
\begin{thm}
In an SQFT, for any decorated morphism $(\phi, \Gamma_c)$, $\D_{\phi, \Gamma_c}$ is a composition of digital creation operators and general digital annihilation operators.
\qed
\end{thm}

Combined with theorem \ref{thm:SFH_gives_SQFT}, this immediately gives \ref{cor:SFH_digital}.

\addcontentsline{toc}{section}{References}

\small

\bibliography{danbib}
\bibliographystyle{amsplain}

\end{document}